\documentclass{article}
	\usepackage[margin=2.5cm,includefoot,footskip=30pt]{geometry}
	\usepackage{longtable}
	\usepackage{enumerate}	
	\usepackage{comment}
	\usepackage{url}
		
	\usepackage{amsmath}
	\usepackage{amsthm,mathdots,cite}	
	\usepackage{amssymb}

	\usepackage[ansinew]{inputenc}
	\usepackage[all,knot,arc,curve,color,frame]{xy}
		\input{xy}
		\xyoption{all}
	\usepackage{extarrows}
	\usepackage{tikz}
		\usetikzlibrary{arrows}
		\usetikzlibrary{positioning}

\theoremstyle{plain}
    \newtheorem{lemma}{Lemma}[section]
    
    \newtheorem{theorem}[lemma]{Theorem}
    \newtheorem{cor}[lemma]{Corollary}
    
    \newtheorem{prop}[lemma]{Proposition}
\theoremstyle{definition}
    \newtheorem{que}[lemma]{Problem}
    \newtheorem{defi}[lemma]{Definition}
    \newtheorem{example}[lemma]{Example}
\theoremstyle{remark}
    \newtheorem{rem}[lemma]{Remark}
    \newtheorem{nota}[lemma]{Notation}

\newcommand{\cgn}{\mathrel{\sim_{\mathrm{u}}}}

\newcommand{\cln}{\mathrel{\sim_l}}
\newcommand{\coon}{\mathrel{\sim_{\mathrm{o}}}}

\newcommand{\cwn}{\mathrel{\sim_{\mathrm{w}}}}
\newcommand{\ctr}{\mathrel{\sim_{\mathrm{tr}}}}

\newcommand{\cpn}{\mathrel{\sim_{\mathrm{p}}}}
\newcommand{\cpns}{\mathrel{\sim_{\mathrm{p}}^*}}

\newcommand{\frn}{\mathfrak{n}}
\newcommand{\cfn}{\mathrel{\sim_{\frn}}}

\newcommand{\con}{\mathrel{\sim_{\mathrm{c}}}}

\newcommand{\cin}{\mathrel{\sim_{i}}}
\newcommand{\cins}{\mathrel{\sim^*_{i}}}

\newcommand{\cbp}{\mathrel{\sim_{\mathrm{bp}}}}

\newcommand{\cli}{\mathrel{\sim_{\mathrm{lin}}}}

\newcommand\gj{\mathcal{J}}
\newcommand\gd{\mathcal{D}}
\newcommand\gl{\mathcal{L}}
\newcommand\gr{\mathcal{R}}
\newcommand\gh{\mathcal{H}}

\newcommand{\greenJ}{\mathrel{\gj}}
\newcommand{\greenD}{\mathrel{\gd}}
\newcommand{\greenL}{\mathrel{\gl}}
\newcommand{\greenR}{\mathrel{\gr}}
\newcommand{\greenH}{\mathrel{\gh}}

\newcommand{\Pax}{\mathcal{P}_X} 
\newcommand{\Pan}{\mathcal{P}_n} 
\newcommand{\PBan}{\mathcal{BP}_n} 
\newcommand{\Ban}{\mathcal{B}_n} 
\newcommand{\Poln}{\mathsf{P}_{\!n}} 

\newcommand{\bn}{\mathbf{n}}
\newcommand{\psn}{\mathbf{S}_n}   

\DeclareMathOperator{\Epi}{Epi}
\DeclareMathOperator{\Sym}{Sym}
\DeclareMathOperator{\End}{End}

\DeclareMathOperator\dom{dom}
\DeclareMathOperator\ima{im}
\DeclareMathOperator\rank{rank}
\DeclareMathOperator\coker{coker}
\DeclareMathOperator{\codom}{codom}

\newcommand{\inv}{^{-1}}                

\newcommand{\sgf}{\sigma}         

\newcommand{\sgs}{\Xi}         

\newcommand\txr{T(X,\rho,R)}

\newcommand\tla{\rho({a})}
\newcommand\tlb{\rho({b})}

\DeclareMathOperator\spa{span}
\DeclareMathOperator\id{id}

\DeclareMathOperator\sym{Sym}

\newcommand{\aar}{\rightarrow}

\newcommand{\ox}{\Omega(X)}

\newcommand{\al}{\alpha}
\newcommand{\bt}{\beta}

\newcommand{\del}{\delta}

\newcommand{\lam}{\lambda}
\newcommand{\sig}{\sigma}
\newcommand{\ome}{\omega}

\newcommand{\vep}{\varepsilon}
\newcommand{\up}{\Upsilon}
\newcommand{\ga}{\Gamma}
\newcommand{\gae}{\Gamma^e}
\newcommand{\Sig}{\Sigma}

\newcommand{\idx}{\id_{\mbox{\tiny $X$}}}

\newcommand\pp{\mathbb{P}}

\newcommand\lan{\langle}
\newcommand\ran{\rangle}

\newcommand\jo{\sqcup}

\newcommand\mi{\mathcal{I}}
\newcommand\mj{\mathcal{J}}

\newcommand\oin{\mathcal{OI}_n}
\newcommand\on{\mathcal{O}_n}

\title{Conjugacy in Abstract Semigroups, Transformation and Diagram Monoids, and Conjugacy Growth}

\author{Jo\~{a}o Ara\'{u}jo\footnote{Center for Mathematics and Applications (NOVA Math) \& Department of Mathematics, NOVA School of Science and Technology, NOVA University of Lisbon, 2829–516 Caparica, Portugal;
\texttt{jj.araujo@fct.unl.pt}},
Wolfram Bentz\footnote{Center for Mathematics and Applications (NOVA Math), NOVA School of Science and Technology, NOVA University of Lisbon, 2829–516 Caparica, Portugal; and Universidade Aberta, R. Escola Polit\'ecnica, 147, 1269--001 Lisboa, Portugal;
\texttt{Wolfram.Bentz@uab.pt}},
Michael Kinyon\footnote{Department of Mathematics, University of Denver, Denver, CO 80208, USA;
\texttt{michael.kinyon@du.edu}},
Janusz Konieczny\footnote{Department of Mathematics, University of Mary Washington, Fredericksburg, VA 22401, USA;
\texttt{jkoniecz@umw.edu}},\\
Ant\'onio Malheiro\footnote{Center for Mathematics and Applications (NOVA Math) \& Department of Mathematics, NOVA School of Science and Technology, NOVA University of Lisbon, 2829–516 Caparica, Portugal;
\texttt{ajm@fct.unl.pt}},
and
Valentin Mercier\footnote{Center for Mathematics and Applications (NOVA Math), NOVA School of Science and Technology, NOVA University of Lisbon, 2829–516 Caparica, Portugal;
\texttt{valen.mercier@gmail.com}}}
\date{}
\begin{document}

\maketitle

\begin{abstract}
We study conjugacy relations on semigroups and monoids, focusing on the relation $a \cfn b$, defined by the existence of $g,h \in S^1$ such that $ag = gb$, $bh = ha$, $hag = b$, and $gbh = a$. This notion emerged as one that yields particularly elegant results. The interplay between $\cfn$ and other standard conjugacy relations is analyzed, and some results on special classes of abstract semigroups are established. 
We then specialize to the case of transformation semigroups. A complete classification of $\cfn$-classes is obtained for the full transformation monoid $\mathcal{T}_n$, the symmetric inverse monoid $\mathcal{I}_n$, and the endomorphism monoid of $G$-sets, among others.
We also investigate the natural conjugacy in diagram semigroups, including the partition monoid, the Brauer monoid, and the partial Brauer monoid. 
Finally, we investigate the conjugacy growth function in polycyclic monoids and obtain a precise asymptotic estimate.
The paper concludes with some open problems. 

\vskip 2mm

\noindent\emph{$2020$ Mathematics Subject Classification\/}. {20E45, 20M10, 20M20.}

\end{abstract}

\tableofcontents

\section{Introduction}\label{Sec:int}
\subsection{Motivation and main goals}
A conjugacy notion on a class of algebras containing groups is an equivalence relation defined  in the language of the class, that coincides with standard group-theoretic conjugacy when the algebra is a group.

As a conjugacy notion is an equivalence relation, it provides a natural way of organizing objects by similarity. It also serves as a shared language across numerous areas of mathematics, such as linear algebra, groups and semigroups, representation theory, Galois theory, geometry, topology, symbolic dynamics, etc.

In this paper, we explore several notions of conjugacy in classes of semigroups, possibly with additional operations. One notion, which we call \emph{natural conjugacy}, stands out for the elegance of its results. We examine its properties and its relationship with other notions in transformation and diagram monoids.

Another key feature of this paper is the following: Over the past two decades, substantial research has been conducted on the conjugacy growth of finitely generated infinite groups. This paper extends this concept to monoids and investigates conjugacy growth in polycyclic monoids, an important family of infinite finitely generated monoids.

In a semigroup $S$, define the natural conjugacy relation $\cfn$ as follows: for all $a,b\in S$,
\begin{equation}\label{e1dcon}
a\cfn b\,\iff\, \exists_{g,h\in S^1}\ (\,ag=gb,\,\, bh=ha,\,\, hag=b,\,\textnormal{ and }\,gbh=a\,)\,.  \tag{$\cfn$}
\end{equation}

The main goals of this paper are the following:
\begin{enumerate}
\item Extend \cite{ArKiKoMaTA} by exploring the links between  conjugacy relations not covered before, especially natural conjugacy as well as two others coming from representation theory and dynamical systems.

\item Describe the natural conjugacy classes--as well as the classes for other notions of conjugacy--in transformation monoids, including the full transformation monoid on a set and the endomorphism monoid of a finite abelian $G$-set. The latter case is motivated by the widespread use of $G$-set endomorphisms in such areas as equivariant topology \cite{may}, homogeneous spaces \cite{tim}, topological dynamics \cite{nek}, representation theory \cite{martin}, algebraic topology \cite{LR05}, and statistical inference (equivariant estimators) \cite{casella,zachs,lehmann}.

\item Describe the natural and other conjugacies in diagram monoids, including the partition monoid, the Brauer monoid, and the partial Brauer monoid. Diagram monoids and their associated algebras and categories arise invariant theory, classical groups, representation theory, logic, knot theory, and statistical mechanics; see, e.g., \cite{ram1,graham,HaRa05,hartman,konig1,konig2,wenzl} as well as the excellent overview in the introduction of \cite{EaGr17}.

\item In the last two decades, many deep papers have been published on conjugacy growth in finitely generated infinite groups \cite{Evetts23,CHM23,CEH20,BdlH18,AC16h,CH14,BCLM13,HO13}. This is a young but fascinating area of research. We extend the group theoretic notion of \emph{conjugacy growth} to monoids and specifically investigate it in polycyclic monoids, a natural family of finitely generated infinite monoids. 
\end{enumerate}

\subsection{Semigroup conjugacies}

In this subsection, we introduce the notions of semigroup conjugacy we will discuss in this paper. 

If $S$ is a semigroup, we let $S^1$ denote $S$ if $S$ is a monoid; otherwise $S^1 = S\cup \{1\}$, where $1$ is an adjoined identity element.

Elements $a,b$ in a group $G$ are conjugate if and only if $g^{-1}ag=b$ for some $g\in G$. To compare this with various generalizations to semigroups, it is helpful to write it without using inverses: $ag = gb$.

Group conjugacy can be generalized directly as follows. For a semigroup $S$, let $G(S^1)$ denote the group of units (invertible elements) of the monoid $S^1$. Define a relation $\cgn$ (unit group conjugacy) on $S$ by
\begin{equation}\label{ecg}
a\cgn b\,\iff\, \exists_{g\in G(S^1)}\ ag = gb\,.
\end{equation}
This is an equivalence relation (symmetry follows from $g\inv a=bg\inv$), but it has an obvious disadvantage:
if $S\neq S^1$ or if $S$ itself is a monoid with trivial group of units, then $\cgn$ is just the equality relation.

If we drop the restriction that the conjugator $g$ must be a unit, then we obtain a relation $\cln$ (left conjugacy), which
has been used to define a notion of conjugacy in arbitrary semigroups \cite{Ot84,Zh91,Zh92}:
\begin{equation}\label{ecl}
a\cln b\,\iff\, \exists_{g\in S^1}\ ag=gb\,.
\end{equation}
This relation is reflexive and transitive, but not, in general, symmetric, and so $\cln$ is not a semigroup
conjugacy for the class of all semigroups. Its symmetrization \cite{Ot84} is denoted by $\coon$, that is,
\begin{equation}\label{econo}
a\coon b\,\iff\, \exists_{g,h\in S^1}\ (\,ag=gb \text{ and } bh=ha\,)\,.
\end{equation}
We call the pair $g,h\in S^1$ \emph{conjugators} for $a,b\in S$. Almost all of the conjugacy relations we discuss in this paper are contained in $\coon$, and so we refer to $g,h$ as conjugators in those cases too.
The relation $\coon$ has its own disadvantage: in any semigroup $S$ with zero, $\coon$ is the universal relation $S\times S$.

One possible remedy to this deficiency is the following relation defined in arbitrary semigroups $S$ \cite{AKM14}:
\begin{equation}\label{econc}
a\con b\,\iff\, \exists_{g\in\pp(a)}\exists_{h\in\pp(b)}\ (\,ag=gb\text{ and }bh=ha\,)\,,
\end{equation}
where for $a\neq 0$, $\pp(a) = \{g\in S^1\,:\,\forall_{m\in S^1}\,(ma\ne0\implies (ma)g\ne0)\}$, and $\pp(0)=\{1\}$.
(See \cite[{\S}2]{AKM14} for the motivation behind this definition.) The relation $\con$ is an equivalence,
does not reduce to $S\times S$ if $S$ has a zero, and is equal to $\coon$ if $S$ does not have a zero.

Another inverse-free formulation of conjugacy for elements $a,b$ of a group $G$ is that $a=uv$ and $b=vu$ for some $u,v\in G$. This formulation has been used to define a relation $\cpn$ (primary conjugacy) \cite{KuMa09,LW1,LW2} in an arbitrary
semigroup:
\begin{equation}\label{ecp}
a\cpn b\,\iff\, \exists_{u,v\in S^1}\ (\,a=uv \text{ and } b=vu\,)\,.
\end{equation}
This relation is reflexive and symmetric, but not transitive, and so again, it is not a semigroup conjugacy on the class
of all semigroups. Its transitive closure \cite{KuMa07,KuMa09} is denoted
by $\cpns$, that is,
\begin{equation}\label{econp}
a\cpns b\,\iff\, \exists_{u_1,\ldots,u_k,v_1,\ldots,v_k\in S^1}\ (\,a=u_1 v_1,\, v_1 u_1=u_2 v_2,\ldots,\, v_k u_k = b\,)\,.
\end{equation}
Conjugators that witness the implication $a\cpns b\,\implies\,a\coon b$ are given by $g = u_1\cdots u_k$ and $h = v_k\cdots v_1$. For recent work on $\cpn$ and $\cpns$ and how they are related to the least commutative congruence of a semigroup, see \cite{Mesyan}.

While $\cln$ and $\cpn$ are not semigroup conjugacies in the class of all semigroups (in spite of their names),
they are semigroup conjugacies in free semigroups, and in fact, in that case all the relations defined in this subsection (except $\cgn$) coincide:
$\cln\,=\,\coon\,=\,\con\,=\,\cpn\,=\,\cpns$  \cite{La79}.

In inverse semigroups, there is a notion of conjugacy, called $i$-\emph{conjugacy}, which is a natural extension of group
conjugacy in its inverse form \cite{ArKiKo_Inv,TJack}:
\begin{equation}\label{eci}
a\cin b\,\iff\, \exists_{g\in S^1}\ (\,g\inv ag = b \text{ and } gbg\inv = a\,)\,.
\end{equation}
(This is explicitly symmetrized because in inverse semigroups, $g\inv ag=b$ does not imply $gbg^{-1}=a$.)
In general, none of the relations $\coon$, $\cpns$, or $\con$ coincide with $\cin$ in inverse semigroups.

In 2018, the fourth author \cite{Ko18} defined a conjugacy $\cfn$ on any semigroup $S$ by \eqref{e1dcon} above, that is,
\begin{equation}\label{e1dcon2}
a\cfn b\,\iff\, \exists_{g,h\in S^1}\ (\,ag=gb\,,\,\, bh=ha\,,\,\, hag=b \text{ and } gbh=a\,)\,.
\end{equation}
The relation $\cfn$ is an equivalence relation on any semigroup $S$, it does not reduce to $S\times S$ if $S$ has a zero,
and it coincides with $\cin$ if $S$ is an inverse semigroup (see Proposition~\ref{Prp:inverse}).
The relation $\cfn$ is essentially the smallest known ``interesting'' conjugacy for general semigroups.
(The relation $\cgn$ is generally smaller, but, as already noted, it is often just equality.)
For these reasons, and others that will come up in the course of this paper, we call $\cfn$ the \emph{natural conjugacy} for semigroups.

The relations $\cpns$ and $\con$ are not comparable with respect to inclusion \cite[Prop.~2.3]{Ko18}. For detailed comparison and analysis
in various classes of semigroups of the conjugacies $\cpns$, $\coon$, $\con$, as well as the \emph{trace conjugacy} $\ctr$, see \cite{ArKiKoMaTA}.

We define a ``new'' conjugacy $\cwn$ in arbitrary semigroups as follows:
\begin{equation}\label{ecwn}
a\cwn b\iff\exists_{g,h\in S^1, m\in \mathbb{Z}^+}\ (\,ag=gb\,,\ bh=ha\,,\ gh=a^m \text{ and } hg=b^m\,)\,.
\end{equation}
This turns out to be an equivalence relation and $\cpns\,\,\subseteq\,\,\cwn\,\,\subseteq\,\,\coon$.
We put ``new'' in quotation marks because, although the relation seems to be new in semigroup theory, both $\cwn$ and $\cpns$ are well known
to specialists in dynamical systems \cite{LindMarcus}.

Regarding the conjugacies discussed so far, in any semigroup, we have:
\[
\begin{tikzpicture}[scale=1, every node/.style={inner sep=1pt}]
  \node (cgn) at (0,0) {$\cgn$};
  \node (cfn) at (0,1) {$\cfn$};
  \node (cpns) at (1.2,2.2) {$\cpns$};
  \node (con)  at (-1.2,2.2) {$\con$};
  \node (cwn)  at (0.6,2.75) {$\cwn$};
  \node (coon) at (0,3.3) {$\coon$};

  \draw (cgn) -- (cfn);
  \draw (cfn) -- (cpns);
  \draw (cfn) -- (con);
  \draw (cpns) -- (cwn);
  \draw (cwn) -- (coon);
  \draw (con) -- (coon);
\end{tikzpicture}
\]

A semigroup conjugacy closely related to $\cwn$ is not defined for all semigroups, but is defined for
epigroups \cite{ArKiKoMaTA}. A semigroup $S$ is an \emph{epigroup} if for each $a\in S$, there exists
a positive integer $n$ such that $a^n$ belongs to a subgroup $H$ of $S$ (see \S\ref{Sec:epi} for more details).
We denote by $a^{\omega}$ the identity element of $H$ \cite[\S2]{Shevrin}, and we set $a^{\ome+1}=a^\ome a$.
Every finite semigroup, or more generally, every periodic semigroup is an epigroup, and
in this case, $a^{\ome}$ itself is a power of $a$.  We define the \emph{trace conjugacy} relation $\ctr$ on any epigroup $S$ as follows \cite{ArKiKoMaTA}:
\begin{equation}\label{ectr}
a\ctr b\iff\exists_{g,h \in S^1}\ (\,ghg=g,\, hgh=h,\, gh=a^\ome,\, hg=b^\ome \text{ and } ha^{\ome+1}g=b^{\ome+1}\,)\,.
\end{equation}
Trace conjugacy, which is an equivalence relation on any epigroup, was inspired
by the representation theory of finite monoids; elements $a,b$ of a finite monoid $S$ satisfy
$a\ctr b$ if and only if $\chi(a)=\chi(b)$ for every irreducible character $\chi$ of $S$  \cite{Steinberg15}.
It is not immediately evident from the definitions that $\ctr\,\subseteq\,\coon$, but there is an equivalent formulation we will need later \cite[Thm.~4.5]{ArKiKoMaTA}:
\begin{equation}\label{ectr2}
a\ctr b\iff\exists_{g,h \in S^1}\ (\,ag=gb,\ bh=ha,\ gh=a^\ome \text{ and } hg=b^\ome\,)\,.
\end{equation}
We will show that $\cwn\,=\,\ctr$ in epigroups, and thus $\cwn$ is the correct generalization of $\ctr$ from epigroups to all semigroups.

Closely related to trace conjugacy is linear conjugacy for finite semigroups \cite{St19}. Let $M_n(\mathbb{F})$ denote the ring of $n\times n$ matrices over a field $\mathbb{F}$. Let $S$ be a finite semigroup. For $a,b\in S$, we say that $a$ is \emph{linearly conjugate} to $b$,
written $a\cli b$, if for every linear representation $\rho:S\to M_n(\mathbb{F})$, there is an invertible matrix $A\in M_n(\mathbb{F})$ such that
$b\rho=A^{-1}(a\rho)A$. (We apply functions on the right and compose from left to right: $x(fg)=(xf)g$.)
Linear conjugacy $\cli$ was introduced and studied in detail in \cite{St19}. In \cite[Thm.~1]{St19}, $\cli$ was characterized as follows:
\begin{equation}\label{eclin}
a\cli b\iff (\ a\ctr b\quad\text{and}\quad a^k\greenJ b^k\text{ for every integer }k\geq 1\ )\,,
\end{equation}
where $\gj$ denotes one of Green's relations; see \S\ref{Sbs:green}.
As was done with $\ctr$ itself, \eqref{eclin} can be taken as a definition of linear conjugacy in any epigroup. 

If $H$ is a subgroup of a group $G$, the conjugacy relation $g\inv ag = b$ between elements $a,b\in H$ can be considered in two ways: internally,
where the conjugator $g\in H$, or externally, where we allow any $g\in G$ such that $g\inv ag=b$ is satisfied. This comes up naturally when $H$ is normal; for example, for the alternating groups $A_n$ ($n\geq 3$) sitting inside the symmetric groups $S_n$, the $S_n$-conjugacy classes
and the $A_n$-conjugacy classes differ.

We will generally not be concerned with this issue for subsemigroups of semigroups, with one specific exception where we will allow external conjugation. For a set $X$, let $\Sym(X)$ denote the symmetric group of permutations on $X$, let $P(X)$ be the semigroup of partial transformations of $X$, and let $S$ be any subsemigroup of $P(X)$. For $\al,\bt\in S$, we say that \emph{$\al$ is conjugate to $\bt$ by permutation}, written $\al\cbp\bt$, if $\bt=\sig^{-1}\al\sig$ for some $\sig\in\Sym(X)$. Note that if $\sym(X)\subseteq S$, then $\cbp$ coincides with unit group conjugacy $\cgn$ (see \eqref{ecg}), and so in this case, $\cbp\,\subseteq\,\cfn$.  Also note that for each subsemigroup $S$ of $P(X)$ and all $\al,\bt\in S$, $\al\cbp\bt$ in $S$ if and only if $\al\cbp\bt$ in $P(X)$.

\subsection{Outline}\label{Sbs:outline}

We conclude the introduction with an outline of the rest of the paper. In \S\ref{Sbs:char}, we provide various alternative definitions of $\cfn$, which we will use throughout the paper. It was stated in \cite{ArKiKoMaTA} that ``\ldots in general, Green's relations and the conjugacies under
consideration are not comparable with respect to inclusion.'' However, in \S\ref{Sbs:green}, we will show a very nice feature of $\cfn$, namely that in any semigroup, $\cfn$ is included in Green's relation $\greenD$, and that $\cfn$ and $\greenD$ coincide when restricted to idempotents. In \S\ref{Sec:inverse}--\ref{Sec:epi}, we study $\cfn$ in inverse and stable semigroups, and in epigroups and completely regular semigroups. 

In \S\ref{Sec:w-ds}, we examine $\cwn$, its connection to $\ctr$, and its origins in dynamical systems. This section answers a question left open in \cite{ArKiKoMaTA}: how can $\ctr$ be extended from epigroups to all semigroups?

Section \S\ref{Sec:tra} deals with transformation semigroups. In  \S\ref{subnat}, which extends the results obtained in \cite{Ko18},
we characterize conjugacy $\cfn$ in some well-known semigroups of transformations, using the representation of transformations by directed graphs (see \S\ref{Subfdr}). In  \S\ref{subbp}, we compare conjugacy by permutations (the traditional idea of {\em changing the label}) with the other concepts we are dealing with, in order to clarify the interconnections between the abstract definitions and the intuitive notion associated with conjugacy.
Subsection \S\ref{ssec:conjGset} has the description of $\cfn$ in the endomorphism monoid of a  $G$-set, for $G$  abelian.
%
%
%
%
%

Section \S\ref{Sec:par} characterizes $\cfn$ in several finite partition monoids, namely the partition monoid itself, the Brauer monoid, and the partial Brauer monoid. We also characterize other notions of conjugacy ($\ctr$, $\cpns$, $\coon$, and $\con$) in these monoids.

In \S\ref{Sec:polycyclic}, we characterize $\cfn$ in the polycyclic monoids, and give closed formulas for the conjugacy growth series of the polycyclic monoid for $\cfn$, $\cpns$, and $\coon$. Finally, \S\ref{Sec:questions} provides a list of open problems.

\section{General results on $\cfn$}
In this section we study $\cfn$ in a manner analogous to \cite{ArKiKoMaTA}.
\subsection{Characterizations of $\cfn$}\label{Sbs:char}

For a semigroup $S$, $a,b\in S$ and $g,h\in S^1$, consider the following equations:
\begin{center}
  \begin{tabular}{rccrc}
  (i)   & $ag = gb$       & \qquad & (ii)   & $bh= ha$ \\
  (iii) & $hag = b$       & \qquad & (iv)   & $gbh = a$ \\
  (v)   & $hg\cdot b = b$ & \qquad & (vi)   & $gh\cdot a = a$ \\
  (vii) & $b\cdot hg = b$ & \qquad & (viii) & $a\cdot gh = a$.
\end{tabular}
\end{center}
Our definition of $\cfn$ is based on (i), (ii), (iii) and (iv). We now give some characterizations which will be useful later. In particular, we could have defined $\cfn$ less symmetrically.

\begin{lemma}\label{lem:alternatives}
Let $S$ be a semigroup, and let $a,b\in S$ and $g,h\in S^1$. Then:
\begin{itemize}
  \item[\textup{(a)}]\quad \textup{(i)}$\implies$\textup{( (iii)}$\iff$\textup{(v) )};
  \item[\textup{(b)}]\quad \textup{(i)}$\implies$\textup{( (iv)}$\iff$\textup{(viii) )};
  \item[\textup{(c)}]\quad \textup{(ii)}$\implies$\textup{( (iv)}$\iff$\textup{(vi) )};
  \item[\textup{(d)}]\quad \textup{(ii)}$\implies$\textup{( (iii)$\iff$(vii) )};
  \item[\textup{(e)}]\quad $\{$\textup{(iii),(vi)}$\}\implies\{$\textup{(i),(v)}$\}$;
  \item[\textup{(f)}]\quad $\{$\textup{(iv),(v)}$\}\implies\{$\textup{(ii),(vi)}$\}$;
  \item[\textup{(g)}]\quad $\{$\textup{(iv),(vii)}$\}\implies\{$\textup{(i),(viii)}$\}$;
  \item[\textup{(h)}]\quad $\{$\textup{(iii),(viii)}$\}\implies\{$\textup{(ii),(vii)}$\}$.
\end{itemize}
\end{lemma}
\begin{proof}
If (i) holds, then $hg\cdot b = hag$ and $a\cdot gh = gbh$. The first of these implies (a), the second implies (b).

If (ii) holds, then $gh\cdot a = gbh$ and $b\cdot hg = hag$. The first of these implies (c), the second implies (d).

For (e), $ag = ghag = gb$ and then (v) follows from (a). For (f), $bh = hgbh = ha$ and then (vi) follows from (c).
For (g), $gb = gbhg = ag$ and then (viii) follows from (b). For (h), $ha = hagh = bh$ and then (vii) follows from (d).
\end{proof}

\begin{prop}\label{Prp:alternatives}
Let $S$ be a semigroup, and let $a,b\in S$ and $g,h\in S^1$. Each of the following sets of equations
implies all of \textup{(i)--(viii)}, and thus $a\cfn b$.
\begin{center}
\begin{tabular}{rlcrl}
\textup{(1)} & $\{$\textup{(i),(iii),(iv)}$\}$ & \qquad &
\textup{(2)} & $\{$\textup{(ii),(iii),(iv)}$\}$ \\
\textup{(3)} & $\{$\textup{(i),(iii),(viii)}$\}$ & \qquad &
\textup{(4)} & $\{$\textup{(ii),(iv),(vii)}$\}$ \\
\textup{(5)} & $\{$\textup{(i),(iv)(v)}$\}$ & \qquad &
\textup{(6)} & $\{$\textup{(ii),(iii),(vi)}$\}$ \\
\textup{(7)} & $\{$\textup{(i),(v),(viii)}$\}$ & \qquad &
\textup{(8)} & $\{$\textup{(ii),(vi),(vii)}$\}$ \\
\textup{(9)} & $\{$\textup{(iii),(iv),(v)}$\}$ & \qquad &
\textup{(10)} & $\{$\textup{(iii),(iv),(vi)}$\}$ \\
\textup{(11)} & $\{$\textup{(iii),(iv),(vii)}$\}$ & \qquad &
\textup{(12)} & $\{$\textup{(iii),(iv),(viii)}$\}$ \\
\textup{(13)} & $\{$\textup{(iii),(vi),(viii)}$\}$ & \qquad &
\textup{(14)} & $\{$\textup{(iv),(v),(vii)}$\}$ \\
\textup{(15)} & $\{$\textup{(i),(ii),(v),(vii)}$\}$ & \qquad &
\textup{(16)} & $\{$\textup{(i),(ii),(vi),(viii)}$\}$ \\
\end{tabular}
\end{center}
\end{prop}
\begin{proof}
Each case follows from tracking implications in Lemma \ref{lem:alternatives}. We prove case (1)
and leave the rest to the reader. Thus assume (i),(iii),(iv) hold. Then (v) and (viii) hold by
parts (a) and (b) of Lemma \ref{lem:alternatives}. Then (ii) holds by part (f), and so (vi) and
(vii) hold by parts (c) and (d).
\end{proof}

A semigroup with zero, denoted $0$, is said to be $3$-\emph{nilpotent} if it satisfies $xyz=0$ for all $x,y,z$. It is generally believed (but has never been proved) that ``almost all'' finite semigroups are $3$-nilpotent. From the previous result it immediately follows that the natural conjugacy in these $3$-nilpotent semigroups is the identity relation. (For a more general result see Proposition \ref{Prp:trivial}.)

For a semigroup $S$, if $a,b\in S$ satisfy $a\cfn b$, then there exist conjugators $g,h\in S^1$ satisfying all of the conditions (i)--(viii). We will use (i)--(viii) freely in calculations without explicit reference.

As already noted, we refer to $\cfn$ as natural conjugacy or just $\frn$-conjugacy, for short. For $a\in S$ we write $[a]_{\frn} = \{b\in S\,:\,b\cfn a\}$ for the conjugacy class of $a$ relative to $\cfn$.

\begin{rem}\label{Rem:n_conj_class_01}
Note that in any semigroup with a zero, $[0]_{\frn} =\{0\}$, and in any monoid $M$ with the identity $1$, $[1]_{\frn} = \{gh\in M\,:\,hg = 1\}$.
\end{rem}

As our first application of Proposition \ref{Prp:alternatives}, we show that natural conjugacy preserves powers. For similar results for other semigroup conjugacies, see (\cite{ArKiKoMaTA}, Thms. 5.14, 5.15). 

\begin{prop}\label{Prp:subsemi}
Let $S$ be a semigroup, let $g,h\in S^1$, and let
\[
K_{g,h} = \{ (a,b)\in S\times S\,:\, a\cfn b\text{ with }g,h\text{ as conjugators }\}\,.
\]
Then $K_{g,h}$ is a subsemigroup of $S\times S$.
\end{prop}
\begin{proof}
  Assume $(a,b),(c,d)\in K_{g,h}$. Then $acg = agd = gbd$, $bdh = bhc=hac$, $hacg = bhcg = bd$ and $gbdh = agdh = ac$. Therefore $(a,b)(c,d)=(ac,bd)\in K_{g,h}$.
\end{proof}
\begin{cor}\label{Cor:powers}
Let $S$ be a semigroup and assume $a,b\in S$ satisfy $a\cfn b$ with conjugators $g,h\in S^1$. Then for all positive integers $k$, $a^k\cfn b^k$ with conjugators $g,h$.
\end{cor}

\subsection{Conjugacy $\cfn$, Green's relations and idempotents}\label{Sbs:green}

If $S$ is a semigroup and $a,b\in S$, we say that $a\greenL b$ if $S^1a = S^1b$, $a\greenR b$ if $aS^1 = bS^1$, and $a\greenJ b$ if $S^1aS^1 = S^1bS^1$. We define $\gh$ as the intersection of $\gl$ and $\gr$, and $\gd$ as the join of $\gl$ and $\gr$, that is, the smallest equivalence relation on $S$ containing both $\gl$ and $\gr$. These five equivalence relations are known as \emph{Green's relations} \cite[p.~45]{Howie},
one of the most important tools in studying semigroups. The relations $\gl$ and $\gr$ commute \cite[Prop.~2.1.3]{Howie}, and consequently $\gd = \gl\circ \gr = \gr\circ \gl$. We have $\gd\subseteq \gj$, and in epigroups, such as finite or periodic semigroups,
$\gd = \gj$ \cite[Prop.~2.1.4]{Howie}.

If $a\greenD b$, then there exists $c\in S$ such that $a\greenR c\greenL b$, and thus there exist $g_1, g_2, h_1, h_2\in S^1$ such that $ag_1 = c$, $c h_1 = a$, $g_2 b = c$ and $h_2 c = a$. Eliminating $c$ from this, we may express the relation $\gd$ as follows:
\[
a\greenD b\,\iff\, \exists_{g_1,g_2,h_1,h_2\in S^1}
(\ ag_1 = g_2 b,\quad a g_1 h_1 = a,\quad h_2 g_2 b = b\ )\,.
\]
Comparing this with Proposition \ref{Prp:alternatives}, we have the following.

\begin{prop}\label{Prp:D}
In any semigroup, $\cfn\ \subseteq\ \greenD$.
\end{prop}

\begin{cor}\label{Cor:Dpowers}
Let $S$ be a semigroup and assume $a,b\in S$ satisfy $a\cfn b$. Then $a^k\greenD b^k$ for all positive integers $k$.
\end{cor}
\begin{proof}
This follows from Corollary \ref{Cor:powers} and Proposition \ref{Prp:D}.
\end{proof}

\begin{example}\label{Exm:n_proper}
From Proposition~\ref{Prp:D} and \cite[Prop.~2.3]{Ko18}, we have $\cfn\ \subseteq\ \greenD\cap \cpn\cap \con$. (Although the cited reference states $\cfn\ \subseteq\, \cpns$, it actually proves the stronger result $\cfn\ \subseteq\ \cpn$.) This inclusion is strict in general. Consider the monoid $S$ defined by the Cayley table
\[
\begin{array}{c|cccccccc}
\cdot & 0 & 1 & 2 & 3 & 4 & 5 & 6 & 7\\
\hline
    0 & 0 & 0 & 0 & 0 & 0 & 0 & 0 & 0 \\
    1 & 0 & 1 & 2 & 3 & 4 & 5 & 6 & 7 \\
    2 & 0 & 2 & 6 & 6 & 3 & 2 & 6 & 2 \\
    3 & 0 & 3 & 6 & 6 & 3 & 2 & 6 & 2 \\
    4 & 0 & 4 & 6 & 6 & 4 & 5 & 6 & 5 \\
    5 & 0 & 5 & 6 & 6 & 4 & 5 & 6 & 5 \\
    6 & 0 & 6 & 6 & 6 & 6 & 6 & 6 & 6 \\
    7 & 0 & 7 & 2 & 3 & 4 & 5 & 6 & 7
\end{array}
\]
We have $2 = 3\cdot 7$ and $3 = 7\cdot 3$, so $2\cpn 3$. Next, $2\cdot 4 = 3$ and $3\cdot 5 = 2$, and so $2\greenR 3$ (and thus certainly $2\greenD 3$). Finally, for all $x,y\in S\backslash \{0\}$, $xy\neq 0$, and thus $x\con y$ in $S$ if and only if $x\coon y$ in $S\backslash \{0\}$. In the latter semigroup, $\coon$ is the universal relation because $6$ is a zero, and so $2\con 3$. However, $2\nsim_{\mathfrak{n}} 3$ because, as can be checked, there are no suitable conjugators.
\end{example}

Next we consider how $\frn$-conjugacy interacts with idempotents. First we note that if an $\frn$-conjugacy class contains an idempotent,
then it consists only of idempotents.

\begin{prop}\label{Prp:idem_only}
    Let $S$ be a semigroup and let $e,a\in S$ where $e$ is an idempotent. If $e\cfn a$, then $a$ is also an idempotent.
\end{prop}
\begin{proof}
Let $g,h\in S^1$ be conjugators for $a$ and $e$. Then $aa = gehgeh = geeh = geh = a$.
\end{proof}

Restricted to idempotents, $\frn$-conjugacy and the $\greenD$-relation turn out to coincide. A pair $g,h$ of elements of a semigroup $S$ are said to be \emph{mutually inverse} if $ghg=g$ and $hgh=h$.

\begin{theorem}\label{Thm:D-idem}
  Let $S$ be a semigroup and let $e,f\in S$ be idempotents. Then $e\cfn f$ if and only if $e\greenD f$. When this is the case, there exist mutually inverse conjugators $g,h$ in the same $\greenD$-class
  as $e,f$.
\end{theorem}
\begin{proof}
One direction is covered by Proposition \ref{Prp:D}, so assume $e\greenD f$. We just follow the proof of \cite[Thm. 2.3.4]{Howie}, noting that the construction therein gives mutually inverse conjugators. Indeed, by assumption, there exist $g,h_1,h_2\in S^1$ such that $eg = g = gf$, $gh_1 = e$ and $h_2 g = f$. (Here we are using the fact that an idempotent $e$ is a left identity element for the $\greenR$-class $R_e$ and a right identity element for the $\greenL$-class $L_e$ \cite[Prop.~2.3.3]{Howie}.) Set $h = fh_1 e$ and check that $gh = gfh_1 e = gh_1 e = ee = e$ and $hg = fh_1 eg = fh_1 g =
h_2 g h_1 g = h_2 e g = h_2 g = f$. Since $eg = gf$, $egh = e$ and $hgf = f$, it follows from Proposition \ref{Prp:alternatives} that $e\cfn f$ with $g,h$ as conjugators. Finally $ghg = eg = g$ and $hgh=fh = h$.
\end{proof}

A \emph{band} is a semigroup in which every element is an idempotent.

\begin{cor}\label{Cor:band}
    In any band, $\cfn\ =\ \greenD$.
\end{cor}

We conclude this subsection with a brief discussion of the two extreme cases: where $\frn$-conjugacy is the universal relation, that is, $\cfn = S\times S$, and where $\cfn$ is the equality relation. In neither case will we arrive at a complete characterization, but each case still entails interesting necessary conditions.

A semigroup is \emph{bisimple} if $\greenD$ is the universal relation. A \emph{rectangular band} is an idempotent semigroup satisfying $xyx = x$; every rectangular band is isomorphic to one of the form $I\times J$ for sets $I,J$ with multiplication $(i,j)\cdot (k,\ell) = (i,\ell)$.

\begin{prop}\label{Prp:bisimple}
    If $S$ is a semigroup in which $\cfn$ is universal, then $S$ is bisimple. If, in addition, $S$ has an idempotent, then $S$ is a rectangular band.
\end{prop}
\begin{proof}
The first assertion follows from Proposition~\ref{Prp:D} and the second follows from Proposition~\ref{Prp:idem_only}.
\end{proof}

At the other extreme, we have the following.

\begin{prop}\label{Prp:trivial}
  Let $S$ be a semigroup in which $\cfn$ is the equality relation. Then each $\greenD$-class has at most one  idempotent, and each regular $\greenD$-class is an $\greenH$-class.
\end{prop}
\begin{proof}
  The first assertion follows from Theorem \ref{Thm:D-idem}. For the second, assume $e$ is an idempotent and $c\greenD e$. Then $c$ is regular and hence there exists an idempotent $f$ such that $c\greenL f$. But then $f\greenD e$ and so by assumption $e = f$, that is, $c\greenL e$. By a similar argument, $c\greenR e$ and so
  $c\greenH e$.
\end{proof}

As noted in the introduction, it was shown in \cite[{\S}3]{ArKiKoMaTA} that Green's relations and the four notions of conjugacy considered in that paper are not particularly well related. The results of this subsection show that the story is different for $\cfn$.

\subsection{Conjugacy $\cfn$ in inverse semigroups}\label{Sec:inverse}

As noted in \S\ref{Sec:int}, natural conjugacy $\cfn$ coincides with $i$-conjugacy $\cin$ \eqref{eci} in inverse semigroups. This was first proved in \cite[Thm.~2.6]{Ko18} using the Wagner-Preston representation of inverse semigroups as semigroups of partial injective transformations \cite[Thm.~5.1.7]{Howie}. In this brief subsection, we give a purely equational proof.

\begin{prop}\label{Prp:inverse}
 In inverse semigroups, $\cfn\ =\ \cin$.
\end{prop}
\begin{proof}
Let $S$ be an inverse semigroup. The inclusion $\cin\ \subseteq\ \cfn$ follows from \cite[Prop.~1.3]{ArKiKo_Inv}, but we give a brief proof here
to keep the discussion self-contained. Suppose $a\cin b$ for some $a,b\in S$. Then $g\inv ag= b$ and $gbg\inv = a$ for some $g\in S^1$. We have $a\cdot gg\inv = gbg\inv gg\inv = gbg\inv = a$ and $gg\inv \cdot a = gg\inv gbg\inv = gbg\inv = a$. Now condition (7) of Proposition \ref{Prp:alternatives} holds with $h = g\inv$ and so $a\cfn b$.

Now suppose $a\cfn b$ for some $a,b\in S$, and let $g,h\in S^1$ be conjugators. Then
\begin{align*}
  g\inv\cdot \underbrace{ag} &= g\inv g\cdot b && \text{(by (i))} \\
   &= \underbrace{g\inv g \cdot b b\inv} \cdot b && \\
   &= \underbrace{b}b\inv\cdot g\inv g\cdot b && \text{(since idempotents commute)} \\
   &= hg\cdot \underbrace{bb\inv \cdot g\inv g\cdot} b && \text{(by (v))}\\
   &= h\cdot \underbrace{g g\inv g} \cdot \underbrace{bb\inv b} && \text{(since idempotents commute)} \\
   &= hg\cdot b && \\
   &= b && \text{(by (v))}
\end{align*}
The equality $gbg\inv = a$ is proved similarly, and so $a\cin b$.
\end{proof}

\subsection{Conjugacy $\cfn$ in stable semigroups}\label{Sec:stable}

The \emph{natural partial order} (or \emph{Mitsch order}) $\leq$ in a semigroup $S$ is defined as follows:
\[
a\leq b\,\iff\,\exists_{s,t\in S^1}\ sa=a=sb\text{ and }at=a=bt\,;
\]
see \cite{Mitsch}. We now consider how natural conjugacy and the natural partial order interact.

A semigroup $S$ is \emph{left stable} if, for all $a,b\in S$, $S^1 a \subseteq S^1 ab$ implies $S^1 a = S^1 ab$, that is, $a\greenL ab$. This can be equivalently formulated as $a\in S^1 ab$ implies $ab\in S^1 a$ for all $a,b\in S$. \emph{Right stability} is defined dually, and a semigroup is said to be \emph{stable} if it is both left and right stable \cite[Vol.~II, p.~31]{ClPr64}. Every epigroup, and in particular, every periodic or finite semigroup, is stable.

An example of a semigroup which is not stable is the \emph{bicyclic monoid} $C$ \cite[Vol.~I, pp.~43--45]{ClPr64}. As a semigroup, $C$ has the presentation
\[
C = \langle a,b\mid aba=a=aab, bab=b=abb\rangle\,;
\]
as a monoid, it can be presented more simply as $C = \langle a,b\mid ab=1\rangle$. The monoid $C$ is not stable essentially because $a$ and $b$ satisfy $ab=1$ but not $ba = 1$. More precisely, $Ca\subseteq C = Cab$, but since $xa\neq 1$ for any $x\in C$, we have $Ca\neq C$.

\begin{lemma}\label{lem:abba}
Let $S$ be a semigroup and let $a,b\in S$ be mutually inverse. Then $ab\cfn ba$.
\end{lemma}
\begin{proof}
This follows from Theorem~\ref{Thm:D-idem} or can also be seen directly: $a\cdot ba\cdot b = ab$, $b\cdot ab\cdot a = ba$, $ba\cdot b = b\cdot ab$ and $ab\cdot a = a\cdot ba$.
\end{proof}

The following result generalizes \cite[Thm.~6.3]{ArKiKo_Inv}.

\begin{theorem}\label{thm:stable}
Let $S$ be a semigroup and consider the following statements:
\begin{itemize}
\item[\textup{(1)}] $S$ is stable;
\item[\textup{(2)}] $\cfn\cap \leq$ is the identity relation;
\item[\textup{(3)}] $S$ does not contain an isomorphic copy of the bicyclic monoid.
\end{itemize}
Then \textup{(1)}$\implies${\textup(2)}$\implies${\textup(3)}. In addition, if $S$ is regular, then all three statements are equivalent.
\end{theorem}
\begin{proof}
(1)$\implies$(2): Assume $a\cfn b$ and $a\leq b$ for some $a,b\in S$. Let $g,h\in S^1$ be conjugators for $a,b$ and let $s,t\in S^1$ witness $a\leq b$, that is, $sa = a = sb$ and $at = a = bt$. We have $a = sb = shag\in S^1ag$. By (left) stability, $ag\in S^1a$, so there exists $u\in S^1$ such that $ag = ua$. Thus $ua = uat = agt = gbt = ga$, hence $ag = ga$. Now $a = bt = hgbt = hga = hag = b$, as claimed.

(2)$\implies$(3): Assume $a,b\in S$ satisfy the defining relations of the bicyclic monoid. Then $a,b$ are mutually inverse, hence $ab\cfn ba$ by Lemma \ref{lem:abba}. Also $ab\cdot ba = ba = ba\cdot ab$, hence $ba\leq ab$. By assumption, $ab = ba$. Thus $a,b$ do not generate a copy of $C$.

For the remaining assertation, it is well known that a regular semigroup is stable if and only if it does not contain an isomorphic copy of $C$ \cite[Ex. A.2.2(8), p. 595]{RS}.
\end{proof}

There are semigroups in which larger notions of conjugacy intersect nontrivially with the natural partial order, and hence those notions cannot be used in place of $\cfn$ in Theorem~\ref{thm:stable}(2).

\begin{example}
(1) Let $S = \{0,1,2\}$, where $0$ is a zero, $2$ is a left identity element and all other products are $0$. Then $1\cdot 2 = 0$ and $2\cdot 1 = 1$ so that $0\cpn 1$, but $0 < 1$. Thus $\cpn\,\cap\,\leq$ is nontrivial.
\smallskip

\noindent (2) Let $S=\{0,1,2\}$ where $0$ is a zero and all other products equal $1$. Then it can be easily checked that $1\con 2$ but $1 < 2$. Thus $\con\,\cap\,\leq$ is nontrivial.
\end{example}

\subsection{Conjugacy $\cfn$ in completely $0$-simple semigroups}\label{Sec:0simple}

A semigroup $S$ with zero $0$ is said to be $0$-\emph{simple} if $S^2\neq 0$ and $S$ has no ideals other than $\{0\}$ and $S$. A $0$-simple semigroup is \emph{completely $0$-simple} if it has a primitive idempotent, that is, a nonzero idempotent which is not an identity element for any other nonzero idempotent.

\begin{prop}\label{cor:cs}
Let $S$ be a $0$-simple semigroup with a nonzero idempotent. Then $S$ is completely $0$-simple if and only if $\cfn\,\cap\,\leq$ is the identity relation.
\end{prop}
\begin{proof}
Completely $0$-simple semigroups are stable (in fact, they are epigroups \cite[Vol.~I, Thm.~2.55, p.~81]{ClPr64}), so the necessity follows from Theorem~\ref{thm:stable}. Conversely, if $\cfn\,\cap\,\leq$ is the identity relation, then Theorem~\ref{thm:stable} implies that $S$
does not contain a copy of the bicyclic monoid. In that case, $S$ is completely $0$-simple by \cite[Vol.~I, Thm.~2.54, p.~81]{ClPr64}.
\end{proof}

Given a group $\Gamma$, two nonempty sets $I$ and $\Lambda$, and a $\Lambda\times I$ matrix $P = (p_{\alpha A})_{\alpha\in \Lambda,A\in I}$ with entries in $\Gamma\cup \{0\}$, the $0$-\emph{Rees matrix semigroup} $\mathcal{M}^0(G;I,\Lambda; P)$ is the set $(I\times \Gamma\times \Lambda)\cup \{0\}$ with multiplication
\[
(A,g,\alpha)(B,h,\beta) := \begin{cases} (A,gp_{\alpha B}h,\beta) & \text{ if } p_{\alpha B}\neq 0 \\ 0 &\text{ if } p_{\alpha B}= 0 \end{cases}
\]
and where $0$ is a zero. In case $P$ has no rows or columns of all $0$s, the semigroup $\mathcal{M}^0(G;I,\Lambda; P)$ is completely $0$-simple.
Rees' Theorem says that the converse is also true: every completely $0$-simple semigroup is isomorphic to some $0$-Rees matrix semigroup $\mathcal{M}^0(G;I,\Lambda; P)$ where $P$ has no rows or columns of all $0$s \cite[Thm.~3.23, p.~72]{Howie}.

We now characterize $\frn$-conjugacy in $0$-Rees matrix semigroups, and thus via the aforementioned isomorphism, in completely $0$-simple
semigroups.

\begin{theorem}
Let $\mathcal{M}^0(G;I,\Lambda; P)$ be the $0$-Rees matrix semigroup determined by the group $\Gamma$, nonempty sets $I$ and $\Lambda$, and
$\Lambda\times I$ matrix $P$ with entries in $\Gamma\cup \{0\}$. For $(A,a,\alpha),(B,b,\beta)\in \mathcal{M}^0(G;I,\Lambda; P)\setminus\{0\}$,
\[
(A,a,\alpha)\cfn(B,b,\beta) \ \text{ if and only if }\ p_{\beta B}\neq 0\neq p_{\alpha A}\ \& \ \exists_{g\in \Gamma}\ p_{\beta B}b = g^{-1} a p_{\alpha A}g.
\]
\end{theorem}
\begin{proof}
We start by proving the necessity. By definition, $(A,a,\alpha)\cfn(B,b,\beta)$ implies that there exist $(G,g,\gamma),(H,h,\eta)\in  {\mathcal M}^0(G;I,\Lambda; P)$ such that
\begin{align*}
(A,a,\alpha)(G,g,\gamma) &= (G,g,\gamma)(B,b,\beta) \\
 (B,b,\beta) &= (H,h,\eta)(A,a,\alpha)(G,g,\gamma) \\
 (A,a,\alpha)&= (G,g,\gamma)(B,b,\beta)(H,h,\eta)\,.
\end{align*}
From the first equality we get $G=A$ and $\gamma=\beta$, from the second we get $H=B$, and  from the third we get $\eta=\alpha$.  Therefore,
\begin{alignat*}{3}
(A,ap_{\alpha A}g,\beta) =
   (A,a,\alpha)(A,g,\beta) &= (A,g,\beta)(B,b,\beta)              &&= (A,gp_{\beta B}b,\beta)\\
               (B,b,\beta) &= (B,h,\alpha)(A,a,\alpha)(A,g,\beta) &&= (B,hp_{\alpha A}ap_{\alpha A}g,\beta) \\
               (A,a,\alpha)&= (A,g,\beta)(B,b,\beta)(B,h,\alpha)  &&= (A,gp_{\beta B}bp_{\beta B}h,\alpha)\,.
\end{alignat*}
The second line of equalities implies that  $p_{\alpha A}\neq 0$ (otherwise $(B,b,\beta)$ would equal $0$ in $\mathcal{M}^0(G;I,\Lambda; P)$, contrary to our assumptions). Similarly, the third line implies that $p_{\beta B}\neq 0$. The first line implies that $ap_{\alpha A}g=gp_{\beta B}b$, that is, $g^{-1}ap_{\alpha A}g=p_{\beta B}b$ as claimed.

Conversely, let $(A,a,\alpha),(B,b,\beta)\in \mathcal{M}^0(G;I,\Lambda; P)$ such that $p_{\beta B}\neq 0\neq p_{\alpha A}$ and there exists $g\in \Gamma$ such that $p_{\beta B}b = g^{-1} a p_{\alpha A}g$. Consider the elements $(A,g,\beta),(B,p^{-1}_{\beta B}g^{-1}p^{-1}_{\alpha A},\alpha)\in \mathcal{M}^0(G;I,\Lambda; P)$.
Then
\[
(A,a,\alpha)(A,g,\beta) = (A,ap_{\alpha A}g,\beta) \overset{a p_{\alpha A}g = gp_{\beta B}b}{=} (A,gp_{\beta B}b,\beta) = (A,g,\beta)(B,b,\beta).
\]
On the other hand,
\[
(B,p^{-1}_{\beta B}g^{-1}p^{-1}_{\alpha A},\alpha) (A,a,\alpha)(A,g,\beta) = (B,p^{-1}_{\beta B}g^{-1}p^{-1}_{\alpha A}p_{\alpha A}ap_{\alpha A}g,\beta) = (B,p^{-1}_{\beta B}g^{-1}ap_{\alpha A}g,\beta) = (B,b,\beta)\,.
\]
Similarly,
\[
(A,g,\beta)(B,b,\beta) (B,p^{-1}_{\beta B}g^{-1}p^{-1}_{\alpha A},\alpha) = (A,gp_{\beta B}b p_{\beta B}p^{-1}_{\beta B}g^{-1}p^{-1}_{\alpha A},\alpha) = (A,gp_{\beta B}b g^{-1}p^{-1}_{\alpha A},\alpha) = (A,a,\alpha)\,.
\]
The result follows.
\end{proof}

\subsection{Conjugacy $\cfn$ in epigroups and completely regular semigroups}
\label{Sec:epi}

An element $a$ of a semigroup $S$ is an \emph{epigroup element} (classically, a \emph{group-bound element}) if there exists a positive integer $n$ such that $a^n$ is contained in a subgroup of $S$. The smallest $n$ for which this is satisfied is the \emph{index} of $a$, and for all $k\geq n$, $a^k$ is contained in the group $\greenH$-class of $a^n$. The set of all epigroup elements of $S$ is denoted by $\Epi(S)$ and the subset consisting of elements of index no more than $n$ is denoted by $\Epi_n(S)$. We have $\Epi_m(S)\subseteq \Epi_n(S)$ for $m\leq n$ and $\Epi(S) = \bigcup_{n\geq 1} \Epi_n(S)$. The elements of $\Epi_1(S)$ are called \emph{completely regular} (or \emph{group elements}); thus $\Epi_1(S)$ is the union of all group $\greenH$-classes of $S$.

For $a\in \Epi_n(S)$, let $a^{\omega}$ denote the identity element of the group $\greenH$-class $H$ of $a^n$. Then $a^{\omega+1} := aa^{\omega} = a^{\omega}a$ is in $H$. The \emph{pseudo-inverse} $a'$ of $a$ is $a' = (a^{\omega+1})^{-1}$, the inverse of $a^{\omega+1}$ in the group $H$ \cite[(2.1)]{Shevrin}. We have the following characterization: $a\in \Epi(S)$ if and only if there exists a positive integer $n$ and a (unique) $a'\in S$ such that the following hold \cite[{\S}2]{Shevrin}:
\begin{equation}\label{etfh}
a'aa' = a'\,,\quad aa'=a'a\,,\quad a^{n+1} a' = a^n,
\end{equation}
where the smallest $n$ such that $a^{n+1} a' = a^n$ is the index of $a$. If $a$ is an epigroup element, then so is $a'$ with $a'' = aa'a$. The element $a''$ is always completely regular and $a''' = a'$. We have $a^\omega = aa' = a'a = a''a' = a'a''$, $(a')^{\omega} = (a'')^{\omega} = a^{\omega}$, and $a^\omega = (a')^m a^m = a^m(a')^m$ for all $m > 0$. In periodic semigroups, $a^{\omega}$ is called the \emph{idempotent power} of $a$. We borrow this term for all epigroup elements, even though $a^{\omega}$ need not be a power of $a$ in general. For a completely regular element $a$, it is customary to denote $a^{\omega}$ by $a^0$, and to denote $a'$ by $a^{-1}$, referring to the latter as the commuting inverse of $a$.

A semigroup $S$ is said to be an \emph{epigroup} if $\Epi(S) = S$. If $\Epi_1(S) = S$, that is, if $S$ is a union of groups, then $S$ is called a \emph{completely regular} semigroup. For $n > 0$, the class $\mathcal{E}_n$ consists of all epigroups $S$ such that $S = \Epi_n(S)$; thus $\mathcal{E}_1$ is the class of completely regular semigroups.

Powers of naturally conjugate elements of a semigroup are also naturally conjugate (Corollary \ref{Cor:powers}). In epigroups, this extends to pseudo-inverses and idempotent powers.

\begin{lemma}\label{Lem:epi-powers}
Let $S$ be a semigroup, let $a,b\in \Epi(S)$, and assume $ag=gb$ for some $g\in S^1$. Then $a^{\omega}g= gb^{\omega}$ and $a'g = gb'$.
\end{lemma}
\begin{proof}
There exist $m,n\in \mathbb{Z}^+$ such that $a^{n+1}a' = a^n$ and $b^{n+1}b' = b^n$. We now prove
\begin{equation}\label{Eqn:epi-powers1}
a^m g = a^m g b^{\omega} \qquad\text{and}\qquad g b^n = a^{\omega} g b^n\,.
\end{equation}
Indeed, $a^m g = g b^m = g b^{m+1} b' = gb^m b^{\omega} = a^m g b^{\omega}$ and the other equation follows similarly. Now $a^{\omega} g = (a^{\omega})^m g = (a')^m a^m g = (a')^m a^m g b^{\omega} = (a^{\omega})^m g b^{\omega} = a^{\omega} g b^{\omega}$, using \eqref{Eqn:epi-powers1} in the third step. Similarly, $g b^{\omega} = a^{\omega} g b^{\omega}$. This establishes the first claim. Finally, $a' g = a' a^{\omega} g = a' g b^{\omega} = a' g bb' = a'agb' = a^{\omega} gb' = g b^{\omega} b' = gb'$. This establishes the second claim and completes the proof.
\end{proof}

\begin{prop}\label{Prp:epi-n}
Let $S$ be a semigroup, let $a,b\in \Epi(S)$, and assume $a\cfn b$ with conjugators $g,h\in S^1$. Then $a^{\omega}\cfn b^{\omega}$ and $a'\cfn b'$ with conjugators $g,h$ in both cases.
\end{prop}
\begin{proof}
We have $a^{\omega}gh = a'agh = a'a = a^{\omega}$ and $a'gh = a'a^{\omega}gh = a'a^{\omega} = a'$. Similarly, $hgb^{\omega} = b^{\omega}$ and $hgb' = b'$. The result now follows from Lemma \ref{Lem:epi-powers} and Proposition \ref{Prp:alternatives}.
\end{proof}

Next we examine how natural conjugacy in epigroups is related to linear conjugacy. We already know that in epigroups, $\cfn\,\subseteq\,\cpns\,\subseteq\,\ctr$.

\begin{prop}\label{Prp:nat_lin}
Let $S$ be a semigroup and let $a,b\in \Epi(S)$. If $a\cfn b$, then $a\cli b$.
\end{prop}
\begin{proof}
We already know $a\ctr b$. By Corollary \ref{Cor:Dpowers}, $a^k\greenD b^k$ for each positive integer $k$. Thus $a^k\greenJ b^k$ for each positive integer $k$. By \eqref{eclin}, $a\cli b$. 
\end{proof}

\begin{cor}\label{Cor:epi_nat_lin}
In any epigroup $\cfn\,\subseteq\,\cli\,\subseteq\,\ctr$.
\end{cor}

In general, neither of $\cpns$ or $\cli$ is contained in the other in epigroups.

\begin{example}
Let $S = \{0,a,b\}$, where $0$ is a zero, $ba=a$, $bb=b$, and all other products equal $0$. This is a semigroup in which $0\cpn a$ since $0=ab$ and $a=ba$. However, $S^1 0 S^1 = \{0\}$ and $S^1 a S^1 = \{0,a\}$, so that $0$ and $a$ are not $\gj$-related. Thus $0$ and $a$ are not linearly conjugate.
\end{example}

\begin{example}
Let $S = \{0,a,b,c,d\}$, where the operation is given by the following Cayley table:
\[
\begin{array}{c|ccccc}
\cdot & 0 & a & b & c & d \\
0     & 0 & 0 & 0 & 0 & 0 \\
a     & 0 & 0 & 0 & b & a \\
b     & 0 & 0 & 0 & a & b \\
c     & 0 & b & a & d & c \\
d     & 0 & a & b & c & d
\end{array}
\]
Then $0'=a'=b'=0$, $c'=c$ and $d'=d$. Since $aa'=bb'=0$, we have $a\ctr b$ with conjugators $g=h=0$. Also $S^1 aS^1 = \{0,a,b\} = S^1 bS^1$, and so $a\gj b$. Since $a^k = b^k = 0$ for all $k>1$, we have $a^k\gj b^k$ for all $k>0$. Thus $a\cli b$. However, since $S$ is commutative, $\cpn$ is the identity relation, so $a\not\cpns b$.
\end{example}

The restriction of $\cpn$ to the set $\Epi_1(S)$ of completely regular elements of a semigroup $S$ is transitive (that is, $\cpn\,=\,\cpns$) and coincides with $\ctr$ \cite[Cor.~4.9]{ArKiKoMaTA}. We extend this result to $\cfn$.

\begin{theorem}\label{Thm:p-conj}
  Let $S$ be a semigroup. Then on $\Epi_1(S)$, $\cfn\,=\,\cpn\,=\,\cli\,=\,\ctr$.
\end{theorem}
\begin{proof}
The inclusions $\cfn\,\subseteq\,\cpn\,\subseteq\,\ctr$ and $\cfn\,\subseteq\,\cli\,\subseteq\,\ctr$ hold in all epigroups. Assume $a,b\in \Epi_1(S)$ and $a\ctr b$. Then there exist $g,h\in S^1$ such that $ag=gb$, $bh=ha$, $gh=a^{\omega}$ and $hg=b^{\omega}$. Since $a$ is completely regular, $agh = aa^{\omega} = a = a^{\omega}a = gha$. By Proposition \ref{Prp:alternatives}, $a\cfn b$.
\end{proof}

\begin{cor}\label{Cor:CR}
  In a completely regular semigroup, $\cfn\,=\,\cpn\,=\,\cli\,=\,\ctr$.
\end{cor}

There are many other epigroups in which $\cfn\,=\,\cpn$, suggesting that a characterization is not feasible.

\begin{example}\label{Exm:notsame}
For $n > 0$, let $S = \{0,1,\ldots,n\}$ where $0$ is assumed to be a zero, and for $x,y\neq 0$, define $\cdot$ by $x\cdot y = \begin{cases} x + y & \text{if }x+y\leq n \\ 0 &\text{if }x+y>n \end{cases}$. Then $S$ is an epigroup in $\mathcal{E}_{n+1}$ but not in $\mathcal{E}_n$. Since $S$ is commutative, $\cpn$ is the identity relation, and thus, so is $\cfn$.
\end{example}

\begin{theorem}\label{Thm:reg_epi}
  Let $S$ be a semigroup in which $\cfn\,=\,\cpn$ and let $c$ be a regular epigroup element. Then $c$ is completely regular.
\end{theorem}
\begin{proof}
Let $c^*$ denote an inverse of $c$, that is, $cc^*c=c$ and $c^*cc^* = c^*$. Since $c$ is an epigroup element, there exists a smallest $n\geq 1$ such that $c^{n+1}c' = c^n$. Assume $n>1$. Since $c^*c\cdot c\cpn c\cdot c^*c = c$ and $\cfn\,=\,\cpn$, it follows that $c^*c^2\cfn c$. Thus there exist conjugators $g,h\in S^1$ for $c^*c^2, c$. By Corollary~\ref{Cor:powers}, $g,h$ are also conjugators for $(c^*c^2)^k, c^k$ for any positive integer $k$. Note that $(c^* c^2)^k = c^* c^{k+1}$. Thus $gc^k = c^* c^{k+1}g$, which we will use multiple times in the following calculation:
\begin{alignat*}{5}
gc^n c' &= c^* c^{n+1} gc'  &&= c^*c\cdot c^ngc' &&= c^*c\cdot c'c^{n+1}gc' \\
     &= c^*c'\cdot c^{n+2}gc'   &&= c^*c'\cdot c\underbrace{c^*c^{n+2}g}c' &&= c^*c'cg\underbrace{c^{n+1}c'} \\
     &= c^*c'cgc^n              &&= c^*c'\underbrace{cc^*c^{n+1}}g &&= c^*\underbrace{c'c^{n+1}}g \\
     &= c^* c^n g = g c^{n-1}\,. && &
\end{alignat*}
Thus $c^nc' = hgc^nc' = hgc^{n-1} = c^{n-1}$. Since $n$ was chosen to be the smallest positive integer $\ell$ such that $c^{\ell}c'=c^{\ell-1}$,
we have a contradiction. Thus $n=1$, that is, $c^2c'=c$. Therefore $c$ is completely regular.
\end{proof}

Combining Theorem~\ref{Thm:reg_epi} with Corollary~\ref{Cor:CR}, we obtain the following.

\begin{cor}\label{Cor:reg_epi}
An epigroup $S$ is completely regular if and only if it is regular and $\cfn\,=\,\cpn$.
\end{cor}

The following extends the second part of Theorem \ref{Thm:D-idem} from idempotents to all completely regular elements.

\begin{prop}\label{Prp:CR-mutual}
  Let $a,b$ be completely regular elements of a semigroup $S$ such that $a\cfn b$. Then there exist mutually inverse conjugators in the $\greenD$-class of $a$ and $b$.
\end{prop}
\begin{proof}
Let $g,h\in S^1$ be conjugators for $a,b$, and set $\bar{g} = a^0 g = g b^0$ and $\bar{h} = b^0 h = h a^0$, the equalities following from Proposition \ref{Prp:epi-n}. Then $a\bar{g} = aa^0 g = gb^0 b = \bar{g}b$, $a\bar{g}\bar{h} = gb^0 b b^0h = gbg = a$, and
$\bar{h}\bar{g}b = ha^0 a^0 g b = ha^0 ag = hag = b$. By Proposition \ref{Prp:alternatives}, $\bar{g},\bar{h}$ are conjugators for $a,b$. Next, we have $\bar{g}\bar{h} = a^0 gha^0 = a^0$ and $\bar{h}\bar{g} = b^0 hg b^0 = b^0$. Putting together the equalities proved so far, we have $a\greenH a^0 \greenR \bar{g}\greenL b^0\greenH b$, which shows $a,b\greenD \bar{g}$, and similarly, $a,b\greenD \bar{h}$. Finally, $\bar{g}\bar{h}\bar{g} = a^0 a^0 g = \bar{g}$ and
$\bar{h}\bar{g}\bar{h} = b^0 b^0 h = \bar{h}$.
\end{proof}

For the remainder of this subsection, we discuss characterizations of $\cfn$ in a completely regular semigroup $S$ in terms of a single conjugator $g\in S^1$ instead of a pair $g,h\in S^1$.

\begin{theorem}\label{thm:CR}
  Let $S$ be a completely regular semigroup. Then, for all $a,b\in S$,
  \[
  a\cfn b \,\iff\, \exists g\in S^1\ (\ ag = gb,\ g^0 a = a,\ bg^0 = b\ ).
  \]
\end{theorem}
\begin{proof}
  Fix $a,b\in S$, assume $a\cfn b$ and let $g,h\in S^1$ be conjugators. Then
  \begin{align*}
    g^0 a &= g^0\cdot gha = gha = a\quad\text{and} \\
    b g^0 &= bhg\cdot g^0 = bhg = b\,.
  \end{align*}

  For the converse, assume that there exists $g\in S^1$ such that $ag=gb$, $g^0 a = a$ and $b g^0 = b$. Set $h = bg\inv a\inv$. We use Lemma \ref{Lem:epi-powers} in the following:
  \begin{align*}
    hg &= bg\inv \underbrace{a\inv g}  = bg\inv g b\inv  = \underbrace{b g^0} b\inv  = bb^{-1}=b^0 \\
    \intertext{and}
    gh &= \underbrace{gb}g\inv a\inv = agg\inv a\inv = a g^0 \underbrace{a\inv a}a\inv = a\underbrace{g^0 a}a\inv a\inv = aaa\inv a\inv = a^0\,.
  \end{align*}
  Thus $hg\cdot b = b$ and $a\cdot gh = a$, and therefore $a\cfn b$ by Proposition \ref{Prp:alternatives}.
\end{proof}

We have already seen that $\frn$-conjugacy is equivalent to $i$-conjugacy in inverse semigroups. It is natural to wonder if something analogous to $i$-conjugacy makes sense in completely regular semigroups using the commuting inverse. For a completely regular semigroup $S$, define $\cin$ by:
\[
a\cin b \,\iff\, \exists_{g\in S^1} (\ g\inv a g = b\ \text{and}\ g b g\inv = a\ )\,.
\]
This coincides with the previously defined $i$-conjugacy in Clifford (completely regular, inverse) semigroups.

\begin{example}
The following table defines a smallest example of a completely regular semigroup in which $\cin$ is not transitive:
\[
\begin{array}{c|ccccccc}
\cdot & 0 & 1 & 2 & 3 & 4 & 5 & 6\\
\hline
    0 & 0 & 0 & 0 & 0 & 0 & 0 & 0 \\
    1 & 1 & 1 & 1 & 1 & 1 & 1 & 1 \\
    2 & 2 & 2 & 2 & 2 & 2 & 2 & 2 \\
    3 & 0 & 1 & 0 & 3 & 3 & 5 & 5 \\
    4 & 2 & 1 & 2 & 4 & 4 & 6 & 6 \\
    5 & 1 & 0 & 1 & 5 & 5 & 3 & 3 \\
    6 & 1 & 2 & 1 & 6 & 6 & 4 & 4
\end{array}
\]
The commuting inverse is just the identity map: $x^{-1} = x$. Set $a = 0$, $b = 1$, $c = 2$, $g = 5$, and $h = 6$. We have $g\inv ag = 5\cdot 0\cdot 5 = 1 = b$ and $gbg\inv= 5\cdot 1\cdot 5 = 0 = a$, and so $a\cin b$. Also $h\inv bh = 6\cdot 1\cdot 6 = 2 = c$ and $hch\inv = 6\cdot 2\cdot 6 = 1 = b$, and so $b\cin c$. Suppose, however, that $x^{-1}ax=c$ and $xcx^{-1}=a$. Then, we must have $x=2$ or $x=4$, but $2c2=2\cdot2\cdot2=2\ne0=a$ and $4c4=4\cdot2\cdot4=2\ne0=a$,
so $a\not\cin c$.
\end{example}

It turns out that $\cin$ is transitive in the variety of cryptogroups (completely regular semigroups in which $\gh$ is a congruence) but not in the variety of orthogroups (completely regular semigroups in which the idempotents form a band). We omit the verifications of these claims. In any case, a characterization of completely regular semigroups in which $\cin$ is transitive seems out of reach at present.

Let $\cins$ denote the transitive closure of $\cin$; thus,
\[
a\cins b\,\iff\, \exists_{g_1,\ldots,g_n\in S^1} (\ g_n\inv \cdots g_1\inv a g_1\cdots g_n = b
\text{ and }g_1\cdots g_n b g_n\inv\cdots g_1\inv = a\ )\,.
\]

\begin{theorem}\label{Thm:i-conj-CR}
Let $S$ be a completely regular semigroup. Then $\cins\,=\,\cfn$.
\end{theorem}
\begin{proof}
First assume $a\cin b$ with $g\inv ag = b$ and $gbg\inv = a$ for some $g\in S^1$. We have $bg^0 = g\inv agg^0 = g\inv ag = b$, and thus $ag = gbg\inv g = gbg^0 = gb$. Similarly, $bg\inv = g\inv a$. Hence $a\cfn b$ and therefore $\cin\,\subseteq\,\cfn$. Since $\cfn$ is transitive, $\cins\,\subseteq\,\cfn$.

Conversely, assume $a\cfn b$ with $ag=gb$, $g^0 a=a$ and $bg^0 = b$, using Theorem \ref{thm:CR}. Set $g_1 = a$, $g_2 = g$ and $g_3 = b\inv$. Then
\[
g_3\inv g_2\inv g_1\inv\cdot a\cdot g_1g_2g_3 = bg\inv \underbrace{a\inv aag}b\inv = bg\inv g\underbrace{b\inv bbb\inv}
= \underbrace{bg^0} b^0 = bb^0 = b\,
\]
using Lemma \ref{Lem:epi-powers} in the second step. Similarly,
$g_1 g_2 g_3\cdot b\cdot g_3\inv g_2\inv g_1\inv = a$. Thus $a\cins b$. Therefore $\cfn\,\subseteq\,\cins$.
\end{proof}

\section{Conjugacies $\ctr$, $\cwn$, and dynamical systems}
\label{Sec:w-ds}

In this brief section, we take a break from our general discussion of natural conjugacy to address a problem arising from the definition of trace conjugacy $\ctr$. As mentioned in \S{1.1}, trace conjugacy is motivated by representation theory; for finite monoids, the notion is precisely what is needed for the characterization that two elements are conjugate if and only if all irreducible characters agree on them \cite{Steinberg15}. The extension of $\ctr$ from finite semigroups to epigroups is formally obvious, but leaves open a question not addressed in, e.g., \cite{ArKiKoMaTA}:
\emph{is there a notion of conjugacy which (i) makes sense in all semigroups and (ii) coincides with $\ctr$ in epigroups?}

As we will see in Theorem \ref{Thm:w-in-tr}, the correct generalization of $\ctr$ from epigroups to arbitrary semigroups is precisely the conjugacy $\cwn$ defined by \eqref{ecwn}. After proving that, we give a few more results connecting $\cwn$ to other conjugacies and conclude with a discussion of the dynamical systems origin of $\cwn$.

\begin{prop}\label{Prp:basic-w}
Let $S$ be a semigroup. Then $\cwn$ is an equivalence relation and $\cpns\,\subseteq\,\cwn\,\subseteq\,\coon$.
\end{prop}
\begin{proof}
That $\cwn$ is reflexive and symmetric is clear.
Assume $a\cwn b$ and $b\cwn c$. Then there are conjugators $g_1,h_1\in S^1$ and $g_2,h_2\in S^1$, respectively,
and positive integers $m,n$ such that $g_1 h_1 = a^m$, $h_1 g_1 = b^m$, $g_2 h_2 = b^n$ and $h_2 g_2 = c^n$. Then $a g_1 g_2 = g_1 g_2 c$, $c h_2 h_1 = h_2 h_1 a$, $g_1 g_2 h_2 h_1 = g_1 b^n h_1 = g_1 h_1 a^n = a^{m+n}$ and $h_2 h_1 g_1 g_2 = h_2 b^m g_2 = h_2 g_2 c^m = c^{m+n}$. Therefore $\cwn$ is transitive.

Now assume $a\cpn b$ so that $a = gh$, $b=hg$ for some $g,h\in S^1$. Then repeating the proof that $\cpn\,\subseteq\,\coon$, we have $ag = ghg = gb$ and $bh = hgh = ha$. Thus $a\cwn b$, and therefore
$\cpn\,\subseteq\,\cwn$. Since $\cwn$ is transitive, $\cpns\,\subseteq\,\cwn$. Finally $\cwn\,\subseteq\,\coon$ is clear from the definitions.
\end{proof}

\begin{theorem}\label{Thm:w-in-tr}
Let $S$ be a semigroup and let $a,b\in \Epi(S)$. Then $a\cwn b$ if and only if $a\ctr b$.
\end{theorem}
\begin{proof}
Assume first that $a\cwn b$. Let $g,h\in S^1$ be conjugators so that $ag=gb$, $bh=ha$, $gh=a^m$ and $hg=b^m$ for some $m>0$.
If $m=1$, then $a\cpn b$, in which case we already know that $a\ctr b$. Thus assume $m>1$.
Let $\bar{g} = gb' = a'g$ and $\bar{h} = (b')^{m-1}h = h(a')^{m-1}$, using Lemma \ref{Lem:epi-powers}.
Then $a\bar{g} = agb' = gbb' = gb'b = \bar{g}b$ and
$b\bar{h} = bh(a')^{m-1} = hb(b')^{m-1} = h(b')^{m-1}b = \bar{h}b$. Next,
$\bar{g}\bar{h} = a'gh(a')^{m-1} = a'a^m(a')^{m-1} = a^m(a')^m = a^{\omega}$ and
$\bar{h}\bar{g} = (b')^{m-1}hgb' = (b')^{m-1}b^mb' = (b')^m b^m = b^{\omega}$.
By the characterization \eqref{ectr2}, $a\ctr b$.

Conversely, assume $a\ctr b$. Let $g,h\in S^1$ be conjugators such that $ag=gb$, $bh=ha$, $gh=a^{\omega}$
and $hg=b^{\omega}$. Since $a,b\in \Epi(S)$, there exist positive integers $k,\ell$ such that
$a^{k+1}a' = a^k$ and $b^{\ell+1}b' = b^{\ell}$. Without loss of generality, assume $k\leq \ell$. 
Let $\bar{g} = g$ and $\bar{h} = b^{\ell} h = h a^{\ell}$. Then $a\bar{g} = ag=ga = \bar{g}a$ and
$b\bar{h} = bha^{\ell} = ha^{\ell+1} = \bar{h}a$. Next, 
$\bar{g}\bar{h} = gha^{\ell} = aa'a^{\ell} = a^{\ell+1}a'$. If $k=\ell$, then $\bar{g}\bar{h} = a^k = a^{\ell}$.
If $k < \ell$, then $\bar{g}\bar{h} = a^{\ell-k}a^{k+1}a' = a^{\ell-k}a^k = a^{\ell}$. Finally,
$\bar{h}\bar{g} = b^{\ell} hg = b^{\ell} bb' = b^{\ell}$. Thus $a\cwn b$.
\end{proof}

\begin{cor}\label{Cor:w_tr}
In any epigroup, $\cwn\,=\,\ctr$.
\end{cor}

\begin{lemma}\label{Lem:w-idem}
Let $S$ be a semigroup, let $e,f\in E(S)$, and assume $e\cwn f$. Then $e\greenD f$.
\end{lemma}
\begin{proof}
We have $e\ctr f$ (Theorem \ref{Thm:w-in-tr}), hence $e\cfn f$ (Theorem \ref{Thm:p-conj}), and so $e\greenD f$ (Theorem \ref{Thm:D-idem}).
\end{proof}

\begin{prop}\label{Prp:w=o}
Let $S$ be a semigroup in which $\cwn\, =\, \coon$. Then $E(S)$ is an antichain.
\end{prop}
\begin{proof}
Assume $e,f\in E(S)$ satisfy $e\leq f$. Then $ee=e=ef$ and $fe=e=ee$, and so $e\coon f$ with
conjugators $e,e$. By assumption, $e\cwn f$, so $e\greenD f$ by Lemma \ref{Lem:w-idem}.
But $e\leq f$ and $e\greenD f$ imply $e=f$.
\end{proof}

The following slightly improves \cite[Thm.~4.2]{ArKiKoMaTA}.

\begin{theorem}\label{Thm:w-CS}
Let $S$ be a semigroup without zero. Then $S$ is regular and $\cwn\, =\, \coon$ if and only if $S$ is completely simple.
\end{theorem}
\begin{proof}
  If $S$ is regular and $\cwn\, =\, \coon$, then $E(S)$ is an antichain, and so the result is standard \cite[Thm.~3.33]{Howie}. The converse follows from \cite[Thm.~4.21]{ArKiKoMaTA} because completely simple semigroups are (completely) regular and satisfy $\cpn\, =\, \coon$.
\end{proof}

\begin{prop}\label{Prp:w-univ}
Let $S$ be a semigroup in which $\cwn\, = S\times S$. Then $S$ has at most one regular $\gd$-class and every subgroup of $S$ is trivial.
\end{prop}
\begin{proof}
  By Lemma \ref{Lem:w-idem}, any two idempotents of $S$ are $\gd$-related. Now assume $e\in E(S)$ and $a\greenH e$. Then $a$ is completely regular with $a^0 = e$. By assumption there exist
  $g,h\in S^1$, $m>0$ such that $eg=ga$, $ah=he$, $gh=e^m = e$ and $hg=a^m$. Now $ge = gaa\inv = ega\inv$ and thus $e\cdot ge = ega\inv a = gea = ga = eg$. Hence $eg = ege = eega\inv = ega\inv$. Using this, we have $a^{m+1} = hga = heg = hega\inv = a^{m+1}a\inv = a^m$.
  Finally $a = a^{-m}a^{m+1} = a^{-m}a^m = e$. Thus every group $\gh$-class of $S$ is trivial.
\end{proof}

\begin{rem}
As mentioned in \S{1.1}, though $\cwn$ seems to be new in semigroup theory, it is well known in dynamical systems as we now briefly discuss, leaving details to the literature \cite{LindMarcus,Williams1,Williams2,Williams3}. Given a strongly connected digraph $G$, the \emph{edge shift} $X_G$ is the set of all bi-infinite walks on the edges, that is, sequences of edges where a vertex at the source of an edge in the walk is the target of the next edge. The dynamics are given by the shift map which sends any sequence in the shift to the sequence obtained by shifting every entry one place to the left. The adjacency matrix $A$ of $G$ has nonnegative integer entries and no row or column consists entirely of $0$s. Since $A$ determines $G$ up to graph isomorphism, the edge shift $X_G$ is essentially determined by~$A$.

If $A$ and $B$ are adjacency matrices, not necessarily of the same size, then an \emph{elementary equivalence} from $A$ to $B$ is a pair $(R,S)$ of rectangular nonnegative integer matrices such that $A = RS$ and $B = SR$. A \emph{strong shift equivalence} from $A$ to $B$ is a sequence $(R_1,S_1),\ldots,(R_k,S_k)$ of elementary equivalences such that $A=R_1 S_1$, $S_1 R_1 = R_2 S_2$,\ldots, $S_k R_k = B$. Strong shift equivalence is an equivalence relation on adjacency matrices.
R. F. Williams' Classification Theorem \cite[Thm.~7.2.7, p.~229]{LindMarcus} \cite{Williams2} states that two edge shifts $X_G$ and $X_H$ are topologically conjugate if and only if their corresponding adjacency matrices $A$ and $B$ are strong shift equivalent.

Let $M$ denote the multiplicative semigroup of all infinite matrices with rows and columns indexed by $\mathbb{Z}^+$, nonnegative integer entries and finite support, that is, only finitely many entries are nonzero. Any $m\times n$ nonnegative integer matrix $C$ can be viewed as an element of $M$ by placing it in the upper left corner of an infinite matrix and filling the rest of the matrix with $0$s. It is clear that elementary equivalence of adjacency matrices, viewed as elements of $M$, coincides with primary conjugacy $\cpn$ in $M$ and strong shift equivalence coincides with $\cpns$.

Since it is difficult to determine if two matrices are strong shift equivalent, Williams \cite{Williams2} defined a weaker notion which is easier to compute. Adjacency matrices $A$ and $B$ are \emph{shift equivalent} if there exist rectangular matrices $R,S$
and a positive integer $m$ such that $AR = RB$, $BS = SA$, $RS = A^m$ and $SR = B^m$. Strong shift equivalence implies shift equivalence. The dynamical significance is that $A$ and $B$ are shift equivalent if and only if their edge shifts are eventually conjugate.

If we view shift equivalence of adjacency matrices as occurring in the semigroup $M$, we are led immediately to the conjugacy $\cwn$ defined by \eqref{ecwn}. The subscript ``w'' is in honor of R. F. Williams.
\end{rem}

\section{Conjugacies in semigroups of transformations}
\label{Sec:tra}
\setcounter{equation}{0}
\setcounter{figure}{0}

In this section we study natural and by-permutation conjugacies in certain semigroups of transformations. We will use the representation of transformations by directed graphs.

\subsection{Functional directed graphs}
\label{Subfdr}

A \emph{directed graph} (or a \emph{digraph}) is a pair $\Gamma = (A,E)$ where $A$ is a set (not necessarily finite and possibly empty) and $E$ is a binary relation on $A$. Any element $x\in A$ is called a \emph{vertex} of~$\Gamma$, and any pair $(x,y)\in E$ is called an \emph{edge} of $\Gamma$. A vertex $x$ of $\ga$ is called \emph{initial}
if there is no vertex $y$ such that $(y,x)\in E$; $x$ is called \emph{terminal} if there is no vertex $y$ such that $(x,y)\in E$.
Let $\Gamma = (A,E)$ and $\up = (B,F)$ be digraphs. A function $\phi:A\to B$ is called a \emph{homomorphism} from $\Gamma$ to $\up$ if
for all $x,y\in A$, $(x,y)\in E$ implies $(x\phi,y\phi)\in F$.
A bijection $\phi:A\to B$ is called an \emph{isomorphism} from $\Gamma$ to $\up$ if for all $x,y\in A$, $(x,y)\in E$ if and only if $(x\phi,y\phi)\in F$; in this case we say that $\Gamma$ and $\up$ are \emph{isomorphic}, written $\Gamma\cong\up$.

Let $P(X)$ be the semigroup of partial transformations on a set $X$.
For $\al\in P(X)$, we denote by $\dom(\al)$ and $\ima(\al)$ the domain and image of $\al$, respectively. We define the \emph{span} of $\al$, written $\spa(\al)$, to be $\dom(\al)\cup \ima(\al)$, and the \emph{kernel} of $\al$ as the equivalence relation $\ker(\al) = \{(x,y): x,y\in\dom(\al)\text{ and }x\al = y\al\}$. Any $\al\in P(X)$ can be represented by the digraph $\Gamma(\al) = (A,E)$, where $A = \spa(\al)$
and for all $x,y\in A$, $(x,y)\in E$ if and only if $x\in \dom(\al)$ and $x\al = y$. Any digraph $\Gamma = (A,E)$ such that $\Gamma = \Gamma(\al)$ for some $\al\in P(X)$, where $A\subseteq X$, is called a \emph{functional digraph}. For the structure of functional digraphs, see \cite{AKM14}.

The following definitions and theorem are fundamental to studying conjugacies in semigroups of transformations.

\begin{defi}\label{dbun}
	Let $\Gamma=(A,E)$ be a digraph. An initial vertex $x$ of $\ga$ will be called \emph{bottom initial} if for all vertices $y,z$ of $\Gamma$, if $(x,y)\in E$ and $(z,y)\in E$, then $z$ is initial.
	
	Let $\al\in P(X)$, $x$ be a bottom initial vertex of $\ga(\al) = (A,E)$, and $y$ be a unique vertex in $\ga(\al)$ such that $(x,y)\in E$ ($y = x\al$). We will call the set $y\al^{-1} = \{z\in A : (z,y)\in E\}$ the \emph{initial bundle} in $\ga(\al)$
	containing~$x$. Note that every vertex in an initial bundle in $\ga(\al)$ is bottom initial.
\end{defi}

For example, the functional digraph presented in Figure~\ref{fig1} on the left has four initial bundles.

\begin{defi}{\bf(\hskip-0.2mm\cite[Def.~3.1]{Ko18})}\label{drh}
	Let $\Gamma=(A,E)$ and $\up=(B,F)$ be digraphs. A homomorphism $\phi:A\to B$ is called a \emph{restricted homomorphism} (or an \emph{r-homomorphism}) from $\Gamma$ to $\up$ if:
	\begin{enumerate}
		\item[(1)] for every terminal vertex $x$ of $\Gamma$, $x\phi$ is a terminal vertex of $\up$;
		\item[(2)] for every bottom initial vertex $x$ of $\Gamma$, $x\phi$ is an initial vertex of $\up$.
	\end{enumerate}
\end{defi}

\begin{defi}{\bf(\hskip-0.2mm\cite[Def.~3.4]{Ko18})}\label{drts}
	We say that a subsemigroup $S$ of $P(X)$ is \emph{closed under restrictions to spans} if for all $\al,\bt\in S$ such that $\spa(\al)\subseteq \dom(\bt)$, $\bt|_{\spa(\al)}\in S$.
	
	Note that every semigroup of full transformations on $X$ is closed under restrictions to spans.
\end{defi}

\begin{theorem}{\bf(\hskip-0.2mm\cite[Thm.~3.5]{Ko18})}\label{tstn}
	Let $S$ be a subsemigroup of $P(X)$ that is closed under restrictions to spans, and let $\al,\bt\in S$. Then $\al\cfn \bt$ in $S$ if and only if there are $\phi,\psi\in S^1$ such that $\phi$ is an r-homomorphism from $\Gamma(\al)$ to $\Gamma(\bt)$, $\psi$ is an r-homomorphism from $\Gamma(\bt)$ to $\Gamma(\al)$, $x(\phi\psi) = x$ for every non-initial vertex $x$ of $\ga(\al)$, and $u(\psi\phi) = u$ for every non-initial vertex $u$ of $\ga(\bt)$.
\end{theorem}

Conjugacy $\cfn$ in $P(X)$ and $T(X)$ (the semigroup of full transformations on $X$) was characterized in \cite{Ko18} in terms of a trim of a functional digraph.

\begin{defi}{\bf(\hskip-0.2mm\cite[Def.~4.3]{Ko18})}\label{dtri}
	For $\al\in P(X)$, we define a \emph{trim} of $\ga(\al)$ as a digraph obtained from $\ga(\al)$ by removing all initial vertices except that we retain exactly one vertex from each initial bundle.
	Any two trims of $\ga(\al)$ are isomorphic. We denote by $\ga^t(\al)$ any trim of $\ga(\al)$.
\end{defi}

However, the concept of a trim of $\ga(\al)$, where $\al\in P(X)$, can be replaced by a simpler concept of the prune of $\ga(\al)$.

\begin{defi}\label{dpru}
	Let $\al\in P(X)$. The digraph $\ga^p(\al)$ obtained from
	$\ga(\al)$ by removing all initial vertices of $\ga(\al)$ will be called the \emph{prune} of $\ga(\al)$.
\end{defi}

The prune of $\ga(\al)$, where $\al\in P(X)$, is a subdigraph of a trim of $\ga(\al)$ since in the latter some initial vertices of $\ga(\al)$ may be preserved. Note that the prune of $\ga(\al)$ is unique (not just unique up to isomorphism). Figure~\ref{fig1} presents an example of a functional digraph, its trim, and its prune.

\begin{figure}[ht]
	\[
	\xy
	(0,44)*{\bullet}="1";
	(0,55.5)*{\bullet}="2";
	(-8,36)*{\bullet}="4";
	(8,36)*{\bullet}="5";
	(-16,28)*{\bullet}="4a";
	(-8,28)*{\bullet}="4b";
	(0,28)*{\bullet}="4c";
	(-22,20)*{\bullet}="4aa";
	(-16,20)*{\bullet}="4ab";
	(-10,20)*{\bullet}="4ac";
	(-6,20)*{\bullet}="4ca";
	(-2,20)*{\bullet}="4cb";
	(2,20)*{\bullet}="4cc";
	(6,20)*{\bullet}="4cd";
	(15,58)*{\vdots};
	(15,52)*{\bullet}="3";
	(15,44)*{\bullet}="6";
	(15,36)*{\bullet}="7";
	(15,28)*{\bullet}="8";
	(15,20)*{\bullet}="9";
	(20,22)*{\bullet}="8a";
	(20,31)*{\bullet}="7a";
	(25,26)*{\bullet}="7aa";
	{\ar@/_3ex/ "1";"2"};
	{\ar@/_3ex/ "2";"1"};
	{\ar@{->} "4a";"4"};
	{\ar@{->} "4b";"4"};
	{\ar@{->} "4c";"4"};
	{\ar@{->} "4aa";"4a"};
	{\ar@{->} "4ab";"4a"};
	{\ar@{->} "4ac";"4a"};
	{\ar@{->} "4ca";"4c"};
	{\ar@{->} "4cb";"4c"};
	{\ar@{->} "4cc";"4c"};
	{\ar@{->} "4cd";"4c"};
	{\ar@{->} "4";"1"};
	{\ar@{->} "5";"1"};
	{\ar@{->} "9";"8"};
	{\ar@{->} "8";"7"};
	{\ar@{->} "7";"6"};
	{\ar@{->} "6";"3"};
	{\ar@{->} "8a";"8"};
	{\ar@{->} "7a";"7"};
	{\ar@{->} "7aa";"7a"};
	(60,44)*{\bullet}="1";
	(60,55.5)*{\bullet}="2";
	(52,36)*{\bullet}="4";
	(44,28)*{\bullet}="4a";
	(60,28)*{\bullet}="4c";
	(38,20)*{\bullet}="4aa";
	(66,20)*{\bullet}="4cd";
	(75,58)*{\vdots};
	(75,52)*{\bullet}="3";
	(75,44)*{\bullet}="6";
	(75,36)*{\bullet}="7";
	(75,28)*{\bullet}="8";
	(75,20)*{\bullet}="9";
	(80,31)*{\bullet}="7a";
	(85,26)*{\bullet}="7aa";
	{\ar@/_3ex/ "1";"2"};
	{\ar@/_3ex/ "2";"1"};
	{\ar@{->} "4a";"4"};
	{\ar@{->} "4c";"4"};
	{\ar@{->} "4aa";"4a"};
	{\ar@{->} "4cd";"4c"};
	{\ar@{->} "4";"1"};
	{\ar@{->} "9";"8"};
	{\ar@{->} "8";"7"};
	{\ar@{->} "7";"6"};
	{\ar@{->} "6";"3"};
	{\ar@{->} "7a";"7"};
	{\ar@{->} "7aa";"7a"};
	(120,44)*{\bullet}="1";
	(120,52)*{\bullet}="2";
	(112,36)*{\bullet}="4";
	(104,28)*{\bullet}="4a";
	(120,28)*{\bullet}="4c";
	(135,58)*{\vdots};
	(135,52)*{\bullet}="3";
	(135,44)*{\bullet}="6";
	(135,36)*{\bullet}="7";
	(135,28)*{\bullet}="8";
	(140,31)*{\bullet}="7a";
	{\ar@/_3ex/ "1";"2"};
	{\ar@/_3ex/ "2";"1"};
	{\ar@{->} "4a";"4"};
	{\ar@{->} "4c";"4"};
	{\ar@{->} "4";"1"};
	{\ar@{->} "8";"7"};
	{\ar@{->} "7";"6"};
	{\ar@{->} "6";"3"};
	{\ar@{->} "7a";"7"};
	
	\endxy
	\]
	\caption{A functional digraph (left), its trim (middle), and its prune (right).}\label{fig1}
\end{figure}

For a function $f\colon A\to B$ and $A_1\subseteq A$, denote by $f|_{A_1}$ the restriction of $f$ to $A_1$.

\begin{prop}\label{ptrpr}
	For all $\al,\bt\in P(X)$, $\ga^t(\al)\cong \ga^t(\bt)$ if and only if $\ga^p(\al)\cong \ga^p(\bt)$.
\end{prop}
\begin{proof}
	Let $\al,\bt\in P(X)$ with $\ga^t(\al) = (A_t,E_t)$, $\ga^p(\al) = (A_p,E_p)$, $\ga^t(\bt) = (B_t,F_t)$, and $\ga^p(\bt) = (B_p,F_p)$. Suppose $\ga^t(\al)\cong \ga^t(\bt)$ and let $\sig:A_t\to B_t$ be an isomorphism from $\ga^t(\al)$ to $\ga^t(\bt)$. The set $A_p$ consists
	of the non-initial vertices of $\ga^t(\al)$, and the subdigraph of $\ga^t(\al)$ induced by $A_p$ is equal to $\ga^p(\al)$. The corresponding statement is true for $\bt$. Since $\sig$ maps the set of non-initial vertices of $\ga^t(\al)$ onto the set of non-initial vertices of $\ga^t(\bt)$, it follows that $\sig|_{A_p}$ is an isomorphism from $\ga^p(\al)$ to $\ga^p(\bt)$.
	
	Conversely, suppose $\ga^p(\al)\cong \ga^p(\bt)$ and let $\del:A_p\to B_p$ be an isomorphism from $\ga^p(\al)$ to $\ga^p(\bt)$. Let $\{y_i\}_{i\in I}$ be the set of initial vertices of $\ga^p(\al)$, where $I$ is an index set (possibly empty). Then $\{v_i\}_{i\in I}$, where $v_i=y_i\del$ for each $i\in I$, is the set of initial vertices of $\ga^p(\bt)$. By the definitions of a trim and the prune of a functional graph, for every $i\in I$, there is a unique initial vertex $x_i$ of $\ga^t(\al)$ such that $(x_i,y_i)\in E$, and $\{x_i\}_{i\in I}$ is the set of initial vertices of $\ga^t(\al)$. Similarly, for every $i\in I$, there is a unique initial vertex $u_i$ of $\ga^t(\bt)$ such that $(u_i,v_i)\in E$, and $\{y_i\}_{i\in I}$ is the set of initial vertices of $\ga^t(\bt)$. Hence $\sig:A_t\to B_t$ that extends $\del$ in such a way that $x_i\sig = u_i$, for every $i\in I$, is an isomorphism from $\ga^t(\al)$ to $\ga^t(\bt)$.
\end{proof}

\subsection{Natural conjugacy in transformation semigroups}
\label{subnat}

In the semigroups $P(X)$ and $T(X)$, $\al\cfn \bt$ if and only if $\ga^t(\al)\cong \ga^t(\bt)$ \cite[Thms.~4.8 and~4.11]{Ko18}.
Thus the following theorem follows immediately from Proposition~\ref{ptrpr}.

\begin{theorem}\label{tpxtxpr}
	In the semigroups $P(X)$ and $T(X)$, $\al\cfn \bt$ if and only if $\ga^p(\al)\cong \ga^p(\bt)$.
\end{theorem}

Denote by $\mi(X)$ the symmetric inverse semigroup of partial injective transformations on $X$, and by $\mj(X)$ the semigroup of full injective
transformation on $X$. (Note that if $X$ is finite, then $\mj(X) = \sym(X)$.) In the semigroups $\mi(X)$ and $\mj(X)$,
\[
\al\cfn \bt \text{ if and only if } \ga(\al)\cong \ga(\bt)
\]
\cite[Cor.~5.2 and Thm.~5.3]{Ko18}.

We will now characterize $\cfn$ in some transformation semigroups not considered in \cite{Ko18}. We begin with the semigroup $\ox$ of surjective transformations on $X$, which was studied in \cite{KoTA}, for which the result is the same as for $\mi(X)$ and $\mj(X)$.

\begin{theorem}\label{tcox}
	In the semigroup $\ox$, $\al\cfn \bt$ if and only if $\ga(\al)\cong \ga(\bt)$.
\end{theorem}
\begin{proof}
	Let $\al,\bt\in \ox$. Suppose that $\al\cfn \bt$ in $\ox$. Then, $\al\cfn \bt$ in $T(X)$, and so $\ga^p(\al)\cong \ga^p(\bt)$ by Theorem~\ref{tpxtxpr}. Since the digraph of any surjective transformation does not have any initial vertices, $\ga^p(\al) = \ga(\al)$ and $\ga^p(\bt) = \ga(\bt)$, and so $\ga(\al)\cong \ga(\bt)$.
	
	Suppose that $\ga(\al)\cong \ga(\bt)$, and let $\phi$ be an isomorphism from $\ga(\al) = (X,E)$ to $\ga(\bt) = (X,F)$. Then $\psi = \phi^{-1}$ is an isomorphism from $\ga(\bt) = (X,E)$ to $\ga(\bt) = (X,F)$. Thus, $\phi,\psi\in\ox$, and they are $r$-homomorphisms satisfying the conditions listed in Theorem~\ref{tstn}. Hence $\al\cfn \bt$ in $\ox$.
\end{proof}

We now turn to semigroups of transformations whose image is restricted by a prescribed set. Such semigroups have been studied extensively; see, for example, \cite{Sy75,Ne05,MeSu11,TiKo16a,TiKo16b}. Let $X$ be an arbitrary set and $\emptyset\neq Y\subseteq X$. Then $T(X,Y) = \{\al\in T(X) : \ima(\al)\subseteq Y\}$ is a subsemigroup of $T(X)$, consisting of transformations whose image is restricted by $Y$.

\begin{lemma}\label{lresphi}
	Let $S$ be a subsemigroup of $P(X)$ that is closed under restrictions to spans, and let $\al,\bt\in S$. Suppose $\phi,\psi\in S^1$ are $r$-homomorphisms as in \textup{Theorem~\ref{tstn}}.
	Let $A_p$ and $B_p$ be the sets of vertices of $\ga^p(\al)$ and $\ga^p(\bt)$, respectively. Then $\phi|_{A_p}$ is an isomorphism from $\ga^p(\al)$ to $\ga^p(\bt)$ and $(\phi|_{A_p})^{-1} = \psi|_{B_p}$.
\end{lemma}
\begin{proof}
	By \cite[Lem.~4.6]{Ko18}, for every non-initial vertex $y$ of $\ga(\al)$, $y\phi$ is not initial in $\ga(\bt)$, and an analogous statement is true for $\psi$. Thus $\phi|_{A_p}$ is a homomorphism from $\ga^p(\al)$ to $\ga^p(\bt)$, and $\psi|_{B_p}$ is a homomorphism from $\ga^p(\bt)$ to $\ga^p(\al)$. Moreover,
	$\phi|_{A_p}$ and $\psi|_{B_p}$ are inverses of each other, which implies that they are isomorphisms.
\end{proof}

\begin{theorem}\label{ttxy}
	Let $X$ and $Y$ be sets such that $\emptyset\neq Y\subseteq X$, and let $\al,\bt\in T(X,Y)$. Then $\al\cfn \bt$ in $T(X,Y)$ if and only if either $\al = \bt$ or $\al\neq\bt$ and the following conditions are satisfied:
	\begin{enumerate} 
		\item[\textup{(a)}] $\ga^p(\al)\cong \ga^p(\bt)$;
		\item[\textup{(b)}] if $Z$ is an initial bundle in $\ga(\al)$ or in $\ga(\bt)$, then $Z\cap Y\neq \emptyset$.
	\end{enumerate}
\end{theorem}
\begin{proof}
	Let $\ga(\al) = (X,E)$, $\ga(\bt) = (X,F)$, $\ga^p(\al) = (A,E_p)$, and $\ga^p(\bt) = (B,F_p)$. Suppose $\al\cfn \bt$ in $T(X,Y)$. Let $\al\neq \bt$. Then there exist $r$-homomorphisms $\phi,\psi\in T(X,Y)$ as in Theorem~\ref{tstn}, where $S = T(X,Y)$. By Lemma~\ref{lresphi}, $\ga^p(\al)\cong \ga^p(\bt)$, so (a) holds.
	Let $Z$ be an initial bundle in $\ga(\bt)$. Then $Z = v\bt^{-1}$ for some initial vertex $v$ in $\ga^p(\bt)$. Let $y = v\psi$. Then $y$ is an initial vertex in $\ga^p(\al)$ (since, by Lemma~\ref{lresphi}, $\psi|_B$ is an isomorphism form $\ga^p(\bt)$ to $\ga^p(\al)$),
	and $y\al^{-1}$ is an initial bundle in $\ga(\al)$ (by \cite[Lem.~4.6]{Ko18}). Let $x\in y\al^{-1}$. Since $\phi$ is a homomorphism and $(x,y)\in E$, we have $(x\phi,v) = (x\phi,v(\psi\phi)) = (x\phi,y\phi)\in F$. Thus $x\phi\in Z$, and so $Z\cap Y\neq \emptyset$ since $x\phi\in Y$. By symmetry, we have $Z\cap Y\neq \emptyset$ for every initial bundle $Z$ in $\ga(\al)$. Hence, (b) holds.
	
	Conversely, if $\al = \bt$, then $\al\cfn \bt$ in $T(X,Y)$. Suppose that $\al\neq \bt$ and (a) and (b) are satisfied. Then by (a), there exists an isomorphism $\del\colon A\to B$ from $\ga^p(\al)$ to $\ga^p(\bt)$. Let $v\in B$. If $v$ is not initial in $\ga^p(\bt)$, then fix $v^*\in B$ such that $(v^*,v)\in F$. If $v$ is initial in $\ga^p(\bt)$, then fix $v^*\in Y$ such that $(v^*,v)\in F$ (possible since $Z = \{u\in X : (u,v)\in F\}$ is an initial bundle in $\ga(\al)$, and so, by (b), $Z\cap Y\ne\emptyset$). Define $\phi:X\to X$ by
	\[
	x\phi = \begin{cases}
		x\del & \text{if }x\in A,\\
		(y\del)^* & \text{if }x\text{ is initial in }\ga(\al)\text{ and }(x,y)\in E.
		\end{cases}
	\]
	It is straightforward to check that $\phi\in T(X,Y)$ and $\phi$ is an $r$-homomorphism from $\ga(\al)$ to $\ga(\bt)$. Symmetrically, we can define $\psi\in T(X,Y)$ such that $\psi$ is an $r$-homomorphism from $\ga(\bt)$ to $\ga(\al)$ with $v\psi = v\del^{-1}$ for every $v\in B$. Hence, $\al\cfn \bt$ in $T(X,Y)$ by Theorem~\ref{tstn}.
\end{proof}

Next we consider the semigroup of order-preserving full transformations on a chain with $n$ elements, where $n\geq1$, say $X_n = \{1<\ldots<n\}$.
Viewing $X_n$ as a set, we denote by $T_n$ the semigroup $T(X_n)$. Let $\on$ be the subset of $T_n$ consisting of order-preserving transformations, that is,
\[
\on = \{\al\in T_n : \forall_{x,y\in X_n}(x\leq y\implies x\al\leq y\al)\}.
\]
The semigroup $\on$ has been studied in numerous papers since the 1960s (see \cite[14.5.1]{GaMa10}).

\begin{nota}\label{ngpb}
	Let $\al,\bt\in P(X_n)$. Suppose $\ga'(\al) = (A',E')$ and $\ga'(\bt) = (B',F')$ are subdigraphs of $\ga(\al)$ and $\ga(\bt)$,
	respectively, where $A' = \{x_1 < \ldots < x_k\}$ and $B' = \{y_1 < \ldots < y_k\}$ ($k\geq 0$). We denote by $\ga'_{\!\bt}(\al)$ the digraph obtained from $\ga'(\al)$ by replacing every vertex $x_i$ with $y_i$.
\end{nota}

\begin{theorem}\label{ton}
	Let $\al,\bt\in \on$ with $\ga(\al) = (X,E)$, $\ga(\bt) = (X,F)$, $\ga^p(\al) = (A,E_p)$, and $\ga^p(\bt) = (B,F_p)$, where
	$A = \{x_1 < \ldots < x_k\}$ and $B = \{y_1 < \ldots < y_m\}$ ($k,m\geq 0$). Then $\al\cfn \bt$ in $\on$ if and only if $k = m$ and $\ga^p_{\!\bt}(\al) = \ga^p(\bt)$.
\end{theorem}
\begin{proof}
	Suppose $\al\cfn \bt$ in $\on$. Let $\phi,\psi\in \on$ be $r$-homomorphisms as in Theorem~\ref{tstn}. (Note that $\on^1 = \on$ since $\idx\in\on$.) By Lemma~\ref{lresphi}, $\phi_p = \phi|_{A}$ is an isomorphism from $\ga^p(\al)$ to $\ga^p(\bt)$, $\psi_p = \psi|_{B}$ is an isomorphism from $\ga^p(\bt)$ to $\ga^p(\al)$, and $\psi_p = \phi_p^{-1}$. This gives $k = m$. Further, $\ga^p_{\!\bt}(\al) = (B,E_0)$, where $(y_i,y_j)\in E_0$ if and only if $(x_i,x_j)\in E_p$. It remains to show that $E_0 = F_p$. Since $\phi_p$ preserves order, we have $x_1\phi_p < \ldots < x_k\phi_p$, which implies $x_i\phi_p = y_i$ for every~$i$. The equality $E_0 = F_p$ follows since for all $i,j$, $(x_i,x_j)\in E_p$ if and only if $(y_i,y_j) = (x_i\phi_p,x_j\phi_p)\in F_p$. Hence $\ga^p_{\!\bt}(\al) = \ga^p(\bt)$.
	
	Conversely, suppose that $k = m$ and $\ga^p_{\!\bt}(\al) = \ga^p(\bt)$. Let $i\in\{1,\ldots,k\}$. Fix $y^*_i\in X$ such that $(y^*_i,y_i)\in F$ (such a $y^*_i$ exists since $y_i$ is not initial in $\ga(\bt)$). Let $A_i = \{x_j : (x_j,x_i)\in E\}$. Let $x$ be an initial vertex in $\ga(\al)$. Then $x\al = x_i$ (so $(x,x_i)\in E$) for some $i$. Note that $x$ is bottom initial in $\ga(\al)$ if and only if $A_i = \emptyset$.
	
	Suppose $A_i\neq \emptyset$. Write $A_i = \{x_{j_1} < \ldots < x_{j_w}\}$ where $w\geq 1$, and define $m_x\in\{j_1,\ldots,j_w\}$ as follows: $m_x = j_1$ if $x < x_{j_1}$, $m_x = j_w$ if $x_w < x$,
	and $m_x = j_s$ if $x_{j_s} < x < x_{j_{s+1}}$. Now, define $\phi:X\to X$ by
	\[
	x\phi=\begin{cases}
		y_i & \text{if }x = x_i,\\
		y^*_i & \text{if }x\text{ is bottom initial in }\ga(\al)\text{ (so }A_i = \emptyset\text{) and }(x,x_i)\in E,\\
		y_{m_x} & \text{if }x\text{ is initial, but not bottom initial, in }\ga(\al)\text{ (so }A_i\neq \emptyset\text{) and }(x,x_i)\in E.
	\end{cases}
	\]
	Note that $x_i\phi = y_i$ for every $i$. First, we will prove that $\phi$ is an $r$-homomorphism from $\ga(\al)$ to $\ga(\bt)$. Since $\ga^p_{\!\bt}(\al) = \ga^p(\bt)$, $(x_i,x_j)\in E$ if and only if $(y_i,y_j)\in F$, for all $i$ and $j$. Moreover, for every $i$, $(y^*_i,y_i)\in F$ and if $x$ is initial, but not bottom initial, in $\ga(\al)$ with $x\al = x_i$, then $(y_{m_x},y_i)\in F$  (since $(x_{m_x},x_i)\in E$). It follows that $\phi$ is a homomorphism.
	Since $\ga(\al)$ does not have any terminal vertices, (1) of Definition~\ref{drh} is vacuously satisfied. Let $x$ be a bottom initial vertex of $\ga(\al)$ and let $x_i = x\al$ (so $(x,x_i)\in E$). Suppose to the contrary that $x\phi$ is not initial in $\ga(\bt)$. Then $x\phi = y_j$, for some~$j$, and $(y_j,y_i) = (x\phi,x_i\phi)\in F$. Thus $(x_j,x_i)\in E$, which is a contradiction since $(x,x_i)\in E$ and $x$ is bottom initial. Hence $x\phi$ is initial in $\ga(\bt)$. Therefore, $\phi$ is an $r$-homomorphism from $\ga(\al)$ to $\ga(\bt)$.
	
	Next, we will prove that $\phi\in\on$. Let $x,z\in X$ with $x < z$, and let $x_i = x\al$ and $x_j = z\al$ (so $(x,x_i)\in E$ and $(z,x_j\in E$). Since $\al\in \on$, we have $x_i\leq x_j$.
	We want to prove that $x\phi\leq z\phi$. Consider three possible cases.
	\vskip 1mm
	\noindent\textbf{Case 1.} $x$ and $z$ are not initial in $\ga(\al)$.
	\vskip 1mm
	Then $x = x_s$ and $z = x_t$ for some $s$ and $t$. Thus $x_s < x_t$, and so $x\phi = x_s\phi = y_s < y_t = x_t\phi = z\phi$.
	\vskip 1mm
	\noindent\textbf{Case 2.} $x$ or $z$ is initial in $\ga(\al)$, and $i\neq j$.
	\vskip 1mm
	Then $x_i < x_j$, and so $y_i < y_j$. Since $\phi$ is a homomorphism from $\ga(\al)$ to $\ga(\bt)$, we have $(x\phi,y_i) = (x\phi,x_i\phi)\in F$ and $(z\phi,y_j) = (z\phi,x_j\phi)\in F$, that is, $(x\phi)\bt = y_i$ and $(z\phi)\bt = y_j$. Since $\bt\in\on$, $z\phi\leq x\phi$ would imply $y_j\leq y_i$, which would contradict $y_i < y_j$. Hence $x\phi < z\phi$.
	\vskip 1mm
	\noindent\textbf{Case 3.} $x$ or $z$ is initial in $\ga(\al)$, and $i = j$.
	\vskip 1mm
	If $A_i = \emptyset$, then both $x$ and $z$ are bottom initial in $\ga(\al)$, and so $x\phi = y^*_i = z\phi$. Let $A_i = \{x_{j_1} < \ldots < x_{j_w}\}\neq \emptyset$. Suppose $x$ is initial in $\ga(\al)$. Then $x\phi = y_{m_x}$. Suppose $z$ is not initial in $\ga(\al)$. Then $z = x_{j_q}$ for some $q$. Since $x < z = x_{j_q}$, we have $x_{m_x}\leq x_{j_q}$ (by the definition of $m_x$), and so $x\phi = y_{m_x}\leq y_{j_q} = x_{j_q}\phi = z\phi$.
	Suppose $z$ is initial in $\ga(\al)$. Then $z\phi = y_{m_z}$.
	Since $x<z$, $x_{m_x}\leq x_{m_z}$, and so $x\phi = y_{m_x}\leq y_{m_z} = z\phi$. If $z$ is initial in $\ga(\al)$,
	then we obtain $x\phi\leq z\phi$ by a similar argument.
	
	Hence, in all cases, $x\phi\leq z\phi$, that is, $\phi\in\on$. By symmetry, there exists an $r$-homomorphism $\psi$ from $\ga(\bt)$ to $\ga(\al)$ such that $y_i\psi = x_i$ for all $i$, and $\psi\in \on$.
	Then for every $i$, $x_i(\phi\psi) = x_i$ and $y_i(\psi\phi) = y_i$. Hence $\phi$ and $\psi$ are as in Theorem~\ref{tstn}, and so $\al\cfn \bt$ in $\on$.
\end{proof}

\begin{example}\label{spraex2}
	Consider $\al,\bt,\del\in \mathcal{O}_6$ whose digraphs are given in Figure~\ref{fig2}. The prunes of the digraphs are presented in Figure~\ref{fig3}, with the orderings of vertices:
	$4 < 5 < 6$ in $\ga^p(\al)$, $3 < 4 < 5$ in $\ga^p(\bt)$, and $2 < 4 < 5$ in $\ga^p(\del)$. Replacing the vertices in $\ga^p(\al)$ according to these orderings, we obtain $\ga^p_\bt(\al)$ and $\ga^p_\del(\al)$ as in Figure~\ref{fig4}. We can see that $\ga^p_\bt(\al) = \ga^p(\bt)$, but $\ga^p_\del(\al)\neq \ga^p(\del)$. Thus by Theorem~\ref{ton}, $\al$ and $\bt$ are $\frn$-conjugate in $\mathcal{O}_6$, but $\al$ and $\del$ are not.
\end{example}

\begin{figure}[h]
	\[
	\xy
	(25,44)*{\bullet}="1";
	(25,36)*{\bullet}="4";
	{\ar@(ur,ul)};
	(18,28)*{\bullet}="4a";
	(25,28)*{\bullet}="4c";
	(35,44)*{\bullet}="7";
	(32,28)*{\bullet}="8";
	{\ar@{->} "4a";"4"};
	{\ar@{->} "4c";"4"};
	{\ar@{->} "4";"1"};
	{\ar@{->} "8";"4"};
	{\ar@(ur,ul)};
	(27,44)*{\text{\footnotesize{$5$}}};
	(34,28)*{\text{\footnotesize{$3$}}};
	(27,28)*{\text{\footnotesize{$2$}}};
	(16,28)*{\text{\footnotesize{$1$}}};
	(37,44)*{\text{\footnotesize{$6$}}};
	(27,36)*{\text{\footnotesize{$4$}}};
	(70,44)*{\bullet}="1";
	(70,36)*{\bullet}="4";
	{\ar@(ur,ul)};
	(70,44)*{\bullet}="1";
	(70,36)*{\bullet}="4";
	{\ar@(ur,ul)};
	(62,36)*{\bullet}="4a";
	(70,28)*{\bullet}="4c";
	(80,44)*{\bullet}="7";
	(80,36)*{\bullet}="8";
	{\ar@{->} "4a";"1"};
	{\ar@{->} "4c";"4"};
	{\ar@{->} "4";"1"};
	{\ar@{->} "8";"7"};
	{\ar@(ur,ul)};
	(72,44)*{\text{\footnotesize{$4$}}};
	(72,36)*{\text{\footnotesize{$3$}}};
	(72,28)*{\text{\footnotesize{$2$}}};
	(60,36)*{\text{\footnotesize{$1$}}};
	(82,44)*{\text{\footnotesize{$5$}}};
	(82,36)*{\text{\footnotesize{$6$}}};
	(115,44)*{\bullet}="11";
	(115,36)*{\bullet}="14";
	{\ar@(ur,ul)};
	(107,36)*{\bullet}="14a";
	(115,28)*{\bullet}="14c";
	(125,44)*{\bullet}="17";
	(125,36)*{\bullet}="18";
	{\ar@{->} "14a";"11"};
	{\ar@{->} "14c";"14"};
	{\ar@{->} "14";"11"};
	{\ar@{->} "18";"17"};
	{\ar@(ur,ul)};
	(117,44)*{\text{\footnotesize{$5$}}};
	(117,36)*{\text{\footnotesize{$4$}}};
	(117,28)*{\text{\footnotesize{$3$}}};
	(105,36)*{\text{\footnotesize{$6$}}};
	(127,44)*{\text{\footnotesize{$2$}}};
	(127,36)*{\text{\footnotesize{$1$}}};
	\endxy
	\]
	\caption{$\ga(\al)$ (left), $\ga(\bt)$ (middle), and $\ga(\del)$ (right).}\label{fig2}
\end{figure}

\begin{figure}[ht]
	\[
	\xy
	(25,44)*{\bullet}="1";
	(25,36)*{\bullet}="4";
	{\ar@(ur,ul)};
	(35,44)*{\bullet}="7";
	{\ar@{->} "4";"1"};
	{\ar@(ur,ul)};
	(27,44)*{\text{\footnotesize{$5$}}};
	(37,44)*{\text{\footnotesize{$6$}}};
	(27,36)*{\text{\footnotesize{$4$}}};
	(70,44)*{\bullet}="1";
	(70,36)*{\bullet}="4";
	{\ar@(ur,ul)};
	(80,44)*{\bullet}="7";
	{\ar@{->} "4";"1"};
	{\ar@(ur,ul)};
	(72,44)*{\text{\footnotesize{$4$}}};
	(72,36)*{\text{\footnotesize{$3$}}};
	(82,44)*{\text{\footnotesize{$5$}}};
	(115,44)*{\bullet}="11";
	(115,36)*{\bullet}="14";
	{\ar@(ur,ul)};
	(125,44)*{\bullet}="17";
	{\ar@{->} "14";"11"};
	{\ar@(ur,ul)};
	(117,44)*{\text{\footnotesize{$5$}}};
	(117,36)*{\text{\footnotesize{$4$}}};
	(127,44)*{\text{\footnotesize{$2$}}};
	\endxy
	\]
	\caption{$\ga^p(\al)$ (left), $\ga^p(\bt)$ (middle), and $\ga^p(\del)$ (right).}\label{fig3}
\end{figure}

\begin{figure}[ht]
	\[
	\xy
	(25,44)*{\bullet}="1";
	(25,36)*{\bullet}="4";
	{\ar@(ur,ul)};
	(35,44)*{\bullet}="7";
	{\ar@{->} "4";"1"};
	{\ar@(ur,ul)};
	(27,44)*{\text{\footnotesize{$4$}}};
	(37,44)*{\text{\footnotesize{$5$}}};
	(27,36)*{\text{\footnotesize{$3$}}};
	(70,44)*{\bullet}="1";
	(70,36)*{\bullet}="4";
	{\ar@(ur,ul)};
	(80,44)*{\bullet}="7";
	{\ar@{->} "4";"1"};
	{\ar@(ur,ul)};
	(72,44)*{\text{\footnotesize{$4$}}};
	(72,36)*{\text{\footnotesize{$2$}}};
	(82,44)*{\text{\footnotesize{$5$}}};
	\endxy
	\]
	\caption{$\ga^p_\bt(\al)$ (left) and $\ga^p_\del(\al)$ (right).}\label{fig4}
\end{figure}

Recall that for an integer $n\geq1$, $X_n=\{1 < \ldots < n\}$. Viewing
$X_n$ as a set, we denote by $\mi_n$ the symmetric inverse semigroup $\mi(X_n)$. Let $\oin$ be the subset of $\mi_n$ consisting of partial injective order-preserving transformations, that is,
\[
\oin = \{\al\in\mi_n : \forall_{x,y\in \dom(\al)}(x<y\implies x\al < y\al)\}.
\]
Then $\oin$ is an inverse semigroup \cite{Fe97,Fe01}. We will now describe $\frn$-conjugacy in $\oin$.

Let $\ga$ be a digraph and let $v_0,v_1,\ldots,v_k$, $k\geq1$, be pairwise distinct vertices of $\ga$. Suppose that
\begin{align}
	&v_0\aar v_1\aar \cdots\aar v_{k-1}\aar v_0,\label{cy}\\
	&v_0\aar v_1\aar \cdots\aar v_{k-1}\aar v_k\label{ch}
\end{align}
are sub-digraphs of $\ga$. We call \eqref{cy} and \eqref{ch}, respectively, a \emph{cycle} of length $k$ (or a $k$-cycle), written $(v_0\,v_1\ldots\, v_{k-1})$, and a \emph{chain} of length $k$ (or a $k$-chain), written $[v_0\,v_1\ldots\, v_k]$, in $\ga$. We can view $(v_0\,v_1\ldots\, v_{k-1})$ and $[v_0\,v_1\ldots\, v_k]$ as partial injective transformations on the set of vertices of $\ga$, both with domain $\{v_0,v_1,\ldots,v_{k-1}\}$, and the values calculated according to \eqref{cy} and \eqref{ch}.

\begin{defi}\label{dcom}
	For a set $X$, let $\al\in P(X)$ and let $x\in\spa(\al)$. The subdigraph of $\ga(\al)$ induced by the set
	\[
	\{y\in \spa(\al) :\ \al^k(y)=\al^m(x)\text{ for some integers }k,m\geq 0\}
	\]
	is called the \emph{component} of $\ga(\al)$ containing $x$. The components of $\ga(\al)$ correspond to the connected components of the underlying undirected graph of $\ga(\al)$.
\end{defi}

If $\al\in\mi_n$, then each component of $\ga(\al)$ is either a cycle or a chain, that is, $\ga(\al)$ is a disjoint union of cycles and chains. We will use the language ``a cycle [chain] in $\al$'' to mean ``a component in $\ga(\al)$ that is a cycle [chain].'' If $\al\in \oin$, then each cycle in $\al$ has length $1$, and if $[v_0\,v_1\ldots\, v_m]$ is a chain in $\al$, then either $v_0 < v_1 < \ldots < v_m$ or $v_0 > v_1 > \ldots > v_m$.

For the meaning of $\ga_{\!\bt}(\al)$ in the following theorem, see Notation~\ref{ngpb}.

\begin{theorem}\label{toin}
	Let $\al,\bt\in \oin$ with $\spa(\al) = \{x_1<\ldots<x_k\}$ and $\spa(\bt) = \{y_1<\ldots<y_m\}$. Then $\al\cfn \bt$ in $\oin$ if and only if $k = m$ and $\ga_{\!\bt}(\al) = \ga(\bt)$.
\end{theorem}
\begin{proof}
	Suppose $\al\cfn\bt$ in $\oin$. Since $\oin$ is closed under restrictions to spans, there is $\phi\in \oin$ such that $\phi$ is an isomorphism from $\ga(\al)$ to $\ga(\bt)$ (by \cite[Thm.~5.1]{Ko18}). (Note that $\oin^1 = \on$ since $\idx\in \on$.) Thus $k = m$. Let $\ga(\al) = (A,E)$ and $\ga(\bt) = (B,F)$. We have $\ga_{\!\bt}(\al) = (B,E_0)$, where $(y_i,y_j)\in E_0$ if and only if $(x_i,x_j)\in E$. It remains to show that $E_0 = F$.
	Since $\phi$ preserves order, we have $x_1\phi < \ldots < x_k\phi$, which implies $x_i\phi = y_i$ for every~$i$. The equality $E_0 = F$ follows since for all $i,j$, $(x_i,x_j)\in E$ if and only if $(y_i,y_j) = (x_i\phi,x_j\phi)\in F$. Hence $\ga_{\!\bt}(\al) = \ga(\bt)$.
	
	Conversely, suppose that $k = m$ and $\ga_{\!\bt}(\al) = \ga(\bt)$. Define $\phi:A\to B$ by $x_i\phi = y_i$ for every $i$. Then $\phi\in \oin$ and for all $i,j$, $(x_i,x_j)\in E\iff (y_i,y_j)\in E_0\iff (y_i,y_j)\in F\iff (x_i\phi,x_j\phi)\in F$. Thus $\phi$ is an isomorphism from $\ga(\al)$ to $\ga(\bt)$, and so $\al\cfn\bt$ in $\oin$ by \cite[Thm.~5.1]{Ko18}.
\end{proof}

Let $\al\in\oin$ with $\spa(\al) = \{x_1 < \ldots < x_k\}$, $k\geq 1$. Using Theorem~\ref{toin}, we can construct the $\frn$-conjugacy class $[\al]_{\frn}$ as follows:
\begin{enumerate}
	\item[(a)] begin with $[\al]_{\frn} = \emptyset$ and $\mathcal{Y}_k = \text{the set of all subchains }\{y_1 < \ldots < y_k\}$ of $X_n$;
	\item[(b)] select a subchain $\{y_1 < \ldots < y_k\}$ from $\mathcal{Y}_k$;
	\item[(c)] replace each $x_i$ in $\ga(\al)$ with $y_i$;
	\item[(d)] add $\bt$ to $[\al]_{\frn}$, where $\bt$ is the transformation represented by the digraph obtained in (c);
	\item[(e)] remove the subchain $\{y_1 < \ldots < y_k\}$ selected in (b) from $\mathcal{Y}_k$;
	\item[(f)] if $\mathcal{Y}_k\neq \emptyset$, return to (b); otherwise STOP.
\end{enumerate}

By the above algorithm and the fact that $[0]_\frn = \{0\}$ in any semigroup with zero, we have
\begin{center}
if $\al\in\oin$ with $|\spa(\al)|=k$, then $|[\al]_{\frn}|=\binom{n}{k}$
\end{center}
for every $k\in \{0,1,\ldots,n\}$.

Let $\emptyset\neq \al\in\oin$. If $\ga(\al)$ has $s + t$ components, where $\sig_1,\ldots,\sig_s$ are $1$-cycles and $\tau_1,\ldots,\tau_t$ are chains, then we will write $\al = \sig_1\jo\cdots\jo\sig_s\jo\tau_1\jo\cdots\jo\tau_t$, where each $\sig_i$ and $\tau_j$ is viewed as an element of $\oin$, and ``$\jo$'' (called the \emph{join}) is the union of functions viewed as sets.

\begin{example}\label{spraex1}
	Consider $\al = (1)\jo(4)\jo[3\,5\,7]\jo[10\,9\,8]\in\mathcal{OI}_{11}$, and note that we have
	\[
	\spa(\al) = \{1 < 3 < 4 < 5 < 7 < 8 < 9 < 10\}
	\]
	and $|\spa(\al)|=8$. Select any subchain of $X_{11}$ with~$8$ elements, say $\{2 < 3 < 5 < 6 < 7 < 8 < 10 < 11\}$. Now, replace
	each $x$ in $\al$, written as above, with the corresponding (according to the orderings) $y$ from that subchain.
	Then, $\bt=(2)\jo(5)\jo[3\,6\,7]\jo[11\,10\,8]$ is $\frn$-conjugate to $\al$.
\end{example}

We conclude this subsection by characterizing natural conjugacy in the class of centralizers of idempotents in $T(X)$. For a semigroup $S$ and an element $a\in S$, the \emph{centralizer} of $a$ in $S$ is the set $C(a) = \{x\in S : ax = xa\}$. It is clear that $C(a)$ is a subsemigroup of $S$. Centralizers in transformation semigroups have been studied extensively (see \cite{AK13} for references and motivation).

Here we will concentrate on the centralizers $C(\vep)$, where $\vep\in T(X)$ is an idempotent ($\vep\vep=\vep$), which were studied
in \cite{AK03,AK04,Ko02}. Let $\vep\in T(X)$ be an idempotent. Each connected component of the digraph $\ga(\vep)$ consists of a $1$-cycle with a vertex $y\in \ima(\vep)$ and the edges $(x,y)$ (perhaps none), where $x\neq\ y$ and $x\vep = y$ (see Figure  \ref{fig5} for the digraph of the idempotent $\vep =
\left(\begin{array}{@{}*{11}{c}@{}}
	1 & 2 & 3 & 4 & 5 & 6 & 7 & 8 & 9 & 10 & 11 \\
	4 & 4 & 4 & 4 & 5 & 5 & 7 & 7 & 7 &  7 & 11
\end{array}\right) \in T(X)$ where $X = \{1,\ldots,11\}$). Note that the set of vertices of a connected component of $\ga(\vep)$ is an element of the partition of $X$ induced by $\ker(\vep)$.

\begin{figure}[h]
	\[
	\xy
	(40,44)*{\bullet}="1";
	(40,36)*{\bullet}="4";
	{\ar@(ur,ul)};
	(32,36)*{\bullet}="4a";
	(48,36)*{\bullet}="4c";
	{\ar@{->} "4a";"1"};
	{\ar@{->} "4c";"1"};
	{\ar@{->} "4";"1"};
	(42,44)*{\text{\footnotesize{$4$}}};
	(42,36)*{\text{\footnotesize{$3$}}};
	(50,36)*{\text{\footnotesize{$2$}}};
	(30,36)*{\text{\footnotesize{$1$}}};
	(70,44)*{\bullet}="1";
	(70,36)*{\bullet}="4";
	{\ar@(ur,ul)};
	(62,36)*{\bullet}="4a";
	(78,36)*{\bullet}="4c";
	(90,44)*{\bullet}="7";
	(90,36)*{\bullet}="8";
	{\ar@{->} "4a";"1"};
	{\ar@{->} "4c";"1"};
	{\ar@{->} "4";"1"};
	{\ar@{->} "8";"7"};
	{\ar@(ur,ul)};
	(72,44)*{\text{\footnotesize{$7$}}};
	(72,36)*{\text{\footnotesize{$8$}}};
	(80,36)*{\text{\footnotesize{$9$}}};
	(59,36)*{\text{\footnotesize{$10$}}};
	(92,44)*{\text{\footnotesize{$5$}}};
	(92,36)*{\text{\footnotesize{$6$}}};
	(105,44)*{\bullet}="17";
	(105,44)*{\bullet}="18";
	{\ar@(ur,ul)};
	(108,44)*{\text{\footnotesize{$11$}}};
	\endxy
	\]
	\caption{The directed graph of an idempotent.}\label{fig5}
\end{figure}

Let $\rho$ be an equivalence relation on $X$, and let $R$ be a cross-section of the partition $X/\rho$ induced by $\rho$. Then the set $\txr$ of elements $\al\in T(X)$ that preserve both $\rho$ and $R$,
\[
\txr = \{\al\in T(X): R\al\subseteq R\text{ and }(x,y)\in \rho\implies (x\al,y\al)\in \rho\},
\]
is a subsemigroup of $T(X)$. The semigroups $\txr$ are exactly the same as the centralizers $C(\vep)$ of idempotents $\vep\in T(X)$
\cite[Thm.~2.3]{AK03}. More precisely, for every idempotent $\vep\in T(X)$, $C(\vep) = \txr$, where $\rho = \ker(\vep)$ and $R = \ima(\vep)$.
The next lemma follows immediately from this observation (see also \cite[Thm.~2.1]{Ko02}).

\begin{lemma}\label{lichar}
	Let $\vep,\al\in T(X)$, where $\vep\in T(X)$ is an idempotent. Then, $\al\in C(\vep)$ if and only if:
	\begin{enumerate}
		\item[\textup{(1)}] $(\ima(\vep))\al\subseteq\ima(\vep)$; and
		\item[\textup{(2)}] for all $y\in\ima(\vep)$ and $x\in X$, if $x\vep=y$, then $(x\al)\vep=y\al$.
	\end{enumerate}
\end{lemma}

\begin{lemma}\label{lcin}
	Let $\al\in C(\vep)$, where $\vep\in T(X)$ is an idempotent. For every $y\in \ima(\vep)$, if $y$ is an initial vertex of $\ga(\al)$, then every vertex $x$ of $\ga(\al)$ such that $x\vep=y$ is also initial.
\end{lemma}
\begin{proof}
	Let $y\in\ima(\vep)$ be initial in $\ga(\al)$, and let $x\vep = y$. Suppose to the contrary that $x$ is not initial in $\ga(\al)$.
	Then $x = z\al$ for some $z\in X$. Let $w = z\vep$. Then $w\in \ima(\vep)$ and, by Lemma~\ref{lichar}, $(z\al)\vep = w\al$,
	so $y = x\vep = (z\al)\vep = w\al$. Hence $y$ is not initial in $\ga(\al)$, which is a contradiction.
\end{proof}

\begin{defi}\label{dstar}
	Let $\al\in C(\vep)$, where $\vep\in T(X)$ is an idempotent. Suppose that $\Sig$ is a subdigraph of $\ga(\al)$. We denote by $\Sig_*$ the digraph with the vertices from $X\times X$, obtained by replacing each vertex $x$ of $\Sig$ by $(x,x\vep)$.
	
	For $x\in X$, it is convenient to think of $x\vep$ as the ``color'' of $x$, and of the vertex $(x,x\vep)$ in $\ga_*(\al)$ as the vertex $x$ ``colored'' by $x\vep$.
	
	Let $\Sig$ and $\up$ be subdigraphs of $\ga(\al)$. Let $\phi$ is a homomorphism from $\Sig$ to $\up$. We denote by $\phi_*$ the function from the set of vertices of $\Sig_*$ to the set of vertices of $\up_*$ defined by $(x,x\vep)\phi_* = (x\phi,(x\phi)\vep)$. Note $\phi_*$ is a homomorphism from $\Sig_*$ to $\up_*$.
	
	Finally, let $f$ be an isomorphism from $\Sig_*$ to $\up_*$. We say that $f$ is \emph{color preserving} if: (i) for all vertices $(x_1,y),(x_2,y)$ of $\Sig_*$, if $(x_1,y)f = (u_1,v_1)$ and $(x_2,y)f = (u_2,v_2)$, then $v_1 = v_2$; and (ii) for every vertex $(y,y)$ of $\Sig_*$, if $(y,y)f = (u,v)$, then $u = v$.
\end{defi}

\begin{defi}\label{dcpres}
	Let $\al,\bt\in C(\vep)$, where $\vep\in T(X)$ is an idempotent, and let $f$ be a color-preserving isomorphism from $\ga^p_*(\al)$ to $\ga^p_*(\bt)$. We say that $f$ \emph{respects colors} of the initial vertices of $\ga_*(\al)$ if it satisfies the following.
	Let $(y,y)$ be a vertex of $\ga_*(\al)$ (so $y\in \ima(\vep))$. Then:
	\begin{enumerate}
		\item[(a)] if $(y,y)$ is initial in $\ga_*(\al)$ and $(y,y)\to (z,z)$ is an edge in $\ga_*(\al)$, then there exists $v\in X$ such that
		\begin{itemize}
			\item[(i)] $(v,v)\to (z,z)f$ is an edge in $\ga_*(\bt)$,
			\item[(ii)] for every initial vertex $(x,y)$ of $\ga_*(\al)$, if $(x,y)\to (x_1,y_1)$ is an edge in $\ga_*(\al)$, then there exists
			a vertex $(u,v)$ in $\ga_*(\bt)$ such that $(u,v)\to (x_1,y_1)f$ is an edge in $\ga_*(\bt)$;
		\end{itemize}
		\item[(b)] if $(y,y)$ is not initial in $\ga_*(\al)$ and $(v,v) = (y,y)f$, then (ii) holds.
	\end{enumerate}
\end{defi}

\begin{lemma}\label{lresgt}
	Let $\al,\bt\in C(\vep)$, where $\vep\in T(X)$ is an idempotent.
	Suppose $\phi,\psi\in C(\vep)$ are $r$-homomorphisms from $\Gamma(\al)$ to $\Gamma(\bt)$ and from $\Gamma(\bt)$ to $\Gamma(\al)$, respectively, such that $x(\phi\psi)=x$ for every non-initial vertex $x$ of $\ga(\al)$, and $u(\psi\phi)=u$ for every non-initial vertex $u$ of $\ga(\bt)$. Let $A_{p*}$ and $B_{p*}$ be the sets of vertices of the digraphs $\ga^p_*(\al)$ and $\ga^p_*(\bt)$, respectively. Then $\phi_*|_{A_{p*}}$ is a color-preserving isomorphism from $\ga^p_*(\al)$ to $\ga^p_*(\bt)$
	that respects colors of the initial vertices of $\ga_*(\al)$,
	$(\phi_*|_{A_{p*}})^{-1} = \psi_*|_{B_{p*}}$ (so $\psi_*|_{B_{p*}}$ is an isomorphism from $\ga^p_*(\bt)$ to $\ga^p_*(\al)$), and $\psi_*|_{B_{p*}}$ is color preserving and it respects colors of the initial vertices of $\ga_*(\bt)$.
\end{lemma}
\begin{proof}
	Let $A_p$ and $B_p$ be the sets of vertices of $\ga^p(\al)$ and $\ga^p(\bt)$, respectively. By Lemma~\ref{lresphi}, $\phi|_{A_p}$ is an isomorphism from $\ga^p(\al)$ to $\ga^p(\bt)$ and $(\phi|_{A_p})^{-1}=\psi|_{B_p}$. Then by Definition~\ref{dstar}, $\phi_*|_{A_{p*}}$ is an isomorphism from $\ga^p_*(\al)$ to $\ga^p_*(\bt)$ and $(\phi_*|_{A_{p*}})^{-1}=\psi_*|_{B_{p*}}$.
	
	Let $f=\phi_*|_{A_{p*}}$. For vertices $(x_1,y),(x_2,y)$ of $\ga^p_*(\al)$, suppose that $(x_1,y)f = (u_1,v_1)$ and $(x_2,y)f = (u_2,v_2)$. Then, $x_1\vep = y = x_2\vep$, and so, since $\phi\in C(\al)$, we have $v_1 = u_1\vep = (x_1\phi)\vep = y\phi$. Similarly, $v_2 = y\phi$, and so $v_1 = v_2$. Let $(y,y)$ be a vertex of $\ga^p_*(\al)$, and suppose that $(y,y)f = (u,v)$. Then $y = y\vep\in\ima(\vep)$, and so $u = y\phi\in\ima(\vep)$. Hence $v = u\vep = u$. Therefore $\phi_*|_{A_{p*}}$ is color preserving.
	
	Let $(y,y)$ be an initial vertex of $\ga_*(\al)$, and let $(y,y)\to (z,z)$ be an edge in $\ga_*(\al)$. Set $v = y\phi$. Since $y\in\ima(\vep)$ and $\phi\in C(\vep)$, we have $v\in\ima(\vep)$. Thus $(y,y)\phi_* = (v,v)$ by the definition of $\phi_*$ (see Definition~\ref{dstar}). Since $\phi_*$ is a homomorphism from $\ga_*(\al)$ to $\ga_*(\bt)$, $(v,v) = (y,y)\phi_*\to (z,z)\phi_*=(z,z)f$ is an edge in $\ga_*(\bt)$. Let $(x,y)$ be an initial vertex of $\ga_*(\al)$, $(x,y)\to (x_1,y_1)$ be an edge in $\ga_*(\al)$, and let $(u,w) = (x,y)\phi_*$. Then $x\phi = u$ and so, since $\phi\in C(\vep)$ and $x\vep = y$, we have $w = u\vep = (x\phi)\vep = y\phi = v$. Further, since $\phi_*$ is a homomorphism from $\ga_*(\al)$ to $\ga_*(\bt)$, $(u,v) = (u,w) = (x,y)\phi_*\to (x_1,y_1)\phi_* = (x_1,y_1)f$ is an edge in $\ga_*(\bt)$.
	
	We have proved that $f = \phi_*|_{A_{p*}}$ satisfies (a) of Definition~\ref{dcpres}. The proof of $f = \phi_*|_{A_{p*}}$ satisfying (b) of Definition~\ref{dcpres} is similar. Thus $\phi_*|_{A_{p*}}$ respects colors of the initial vertices of $\ga_*(\al)$. By symmetry, $\psi_*|_{B_{p*}}$ is color preserving and it respects colors of the initial vertices of $\ga_*(\bt)$.
\end{proof}

\begin{theorem}\label{tcechar}
	Let $\al,\bt\in C(\vep)$ where $\vep\in T(X)$ is an idempotent. Then $\al\cfn\bt$ in $C(\vep)$ if and only if there exists an isomorphism $f$ from $\ga^p_*(\al)$ to $\ga^p_*(\bt)$ such that $f$ and $f^{-1}$ are color preserving and they respect colors of the initial vertices (of $\ga_*(\al)$ and $\ga_*(\bt)$, respectively).
\end{theorem}
\begin{proof}
	We first note that the statements about $f^{-1}$ being color preserving and respecting colors of the initial vertices make sense since if $f$ is an isomorphism from $\ga^p_*(\al)$ to $\ga^p_*(\bt)$, then $f^{-1}$ is an isomorphism from $\ga^p_*(\bt)$ to $\ga^p_*(\al)$. Suppose that $\al\cfn\bt$ in $C(\vep)$. Then by Theorem~\ref{tstn}, there exist $r$-homomorphisms $\phi,\psi\in C(\vep)$ as in Lemma~\ref{lresgt}. Thus by Lemma~\ref{lresgt}, $f = \phi_*|_{A_{p*}}$ is an isomorphism from $\ga^p_*(\al)$ to $\ga^p_*(\bt)$ that satisfies the desired properties.
	
	Conversely, suppose that there exists an isomorphism $f$ from $\ga^p_*(\al)$ to $\ga^p_*(\bt)$ that satisfies the given properties.
	To prove that $\al\cfn\bt$, we will define suitable $r$-homomorphisms $\phi$ (from $\ga(\al)$ to $\ga(\bt)$) and $\psi$ (from $\ga(\bt)$ to $\ga(\al)$). Define $\phi:X\to X$ as follows. Let $x\in X$. If $x$ is not initial in $\ga(\al)$, then we set $x\phi = u$, where $u$ is such that $(x,x\vep)f = (u,u\vep)$.
	Suppose that $x$ is initial in $\ga(\al)$. Then for $y = x\vep$, $(x,y)$ is initial in $\ga_*(\al)$. Let $v$ a vertex of $\ga(\al)$ as in Definition~\ref{dcpres}. Fix a vertex $u_x$ of $\ga(\al)$
	such that $u_x = u$, where $u$ is a vertex of $\ga(\al)$ as in Definition~\ref{dcpres}. (There may be more that one such a $u$ for $x$, and we fix one of them.) Moreover, if $x = y$ (so $(y,y)$ is initial in $\ga_*(\al)$), then $v$ is one of the vertices $u$
	as in Definition~\ref{dcpres}, and we may assume that $u_x = v$. We now set $x\phi = u_x$.
	
	To show that $\phi$ is a homomorphism from $\ga(\al)$ to $\ga(\bt)$, let $x\to x_1$ be an edge in $\ga(\al)$. Note that $x_1$ is not initial in $\ga(\al)$. Suppose $x$ is not initial in $\ga(\al)$. Let $u = x\phi$ and $u_1 = x_1\phi$. Then $(x,x\vep)f = (u,u\vep)$ and $(x_1,x_1\vep)f = (u_1,u_1\vep)$. Further, $(x,x\vep)\to (x_1,x_1\vep)$ is an edge in $\ga^p_*(\al)$, and so $(u,u\vep) = (x,x\vep)f\to (x_1,x_1\vep)f = (u_1,u_1\vep)$ is an edge in $\ga^p_*(\bt)$. Thus $x\phi = u\to u_1 = x_1\phi$ is an edge in $\ga^p(\bt)$, and so an edge in $\ga(\bt)$. Suppose $x$ is initial in $\ga(\al)$, and let $y = x\vep$ and $y_1 = x_1\vep$. Then
	$(x,y)$ is initial in $\ga_*(\al)$ and $(x,y)\to (x_1,y_1)$ is an edge in $\ga_*(\al)$. Thus the definition of $\phi$ implies that
	there exists $v\in X$ such that $(u_x,v)$ is a vertex of $\ga_*(\bt)$ and $(u_x,v)\to (x_1,y_1)f=(u_1,u_1\vep)$ is an edge in $\ga_*(\bt)$ (see Definition~\ref{dcpres}). Hence $x\phi = u_x\to u_1 = x_1\phi$ is an edge in $\ga(\bt)$.
	
	To show that $\phi$ is an $r$-homomorphism (see Definition~\ref{drh}), we first note that $\ga(\al)$ does not have any terminal vertices since $\al\in T(X)$. Let $x$ be a bottom initial vertex of $\ga(\al)$ and let $x\to x_1$ be an edge in $\ga(\al)$. Then $x_1$ is an initial vertex of $\ga^p(\al)$ and $x\phi\to x_1\phi$ is an edge in $\ga(\bt)$ (since $\phi$ is a homomorphism). Thus $(x_1,x_1\vep)$ is an initial vertex of $\ga_*^p(\al)$, and so, since $f$ is an isomorphism from $\ga_*^p(\al)$ to $\ga_*^p(\bt)$, the vertex $(x_1,x_1\vep)f = (x_1\phi,(x_1\phi)\vep)$ is initial in $\ga_*^p(\bt)$. Hence $x_1\phi$ is initial in $\ga^p(\bt)$ and so $x\phi$ is initial (even bottom initial) in $\ga(\bt)$.
	
	We will now show that $\phi\in C(\vep)$. Let $y\in \ima(\vep)$. Suppose $y$ is not initial in $\ga(\al)$. Then $(y,y)$ is a vertex in $\ga_*^p(\al)$, and so, since $f$ is color preserving, $(y,y)f = (v,v)$ for some $v\in X$. Thus $y\phi = v$ and $v\in \ima(\vep)$. Suppose $y$ is initial in $\ga(\al)$. Then, setting $x = y$, we have $y\phi = x\phi = u_x = v$, where $v$ is as in Definition~\ref{dcpres}, so $v\in \ima(\vep)$. We have proved that $(\ima(\vep))\phi\subseteq \ima(\vep)$.
	
	Let $x\in X$, and let $y = x\vep$, $u = x\phi$, and $v = y\phi$. Consider two possible cases.
	\vskip 1mm
	\noindent\textbf{Case 1.} $y$ is not initial in $\ga(\al)$.
	\vskip 1mm
	Then, by the definition of $\phi$, $(y,y\vep)f=(v,v\vep)$, and so $(y,y)f = (v,v)$. If $x$ is not initial in $\ga(\al)$, then
	$(x,y)f = (x,x\vep)f = (u,u\vep)$, and so, since $f$ is color preserving, $(x\phi)\vep = u\vep = v\vep = v = y\phi$. Suppose
	$x$ is initial in $\ga(\al)$. Then $u$ is as in Definition~\ref{dcpres}, so $(u,v)$ is a vertex of $\ga_*(\bt)$.
	Thus $u\vep = v$, and so $(x\phi)\vep = y\phi$.
	\vskip 1mm
	\noindent\textbf{Case 2.} $y$ is initial in $\ga(\al)$.
	\vskip 1mm
	Then by the definition of $\phi$, $v$ is as in Definition~\ref{dcpres}. Further, by Lemma~\ref{lcin}, $x$ is
	initial in $\ga(\al)$, and we obtain $(x\phi)\vep = y\phi$ as in the proof of Case~1.
	
	Hence $\phi\in C(\vep)$. We define an $r$-homomorphism $\psi\in C(\bt)$ from $\ga(\bt)$ to $\ga(\al)$ in the same way using $f^{-1}$.
	Then for every non-initial vertex $x$ of $\ga(\al)$, we have
	$(x,x\vep)(ff^{-1}) = (x,x\vep)$, and so $x(\phi\psi) = x$.
	Similarly, $u(\psi\phi) = u$ for every non-initial vertex $u$ of $\ga(\bt)$. Therefore $\al\cfn\bt$ in $C(\vep)$ by Theorem~\ref{tstn}.
\end{proof}

The following example shows that if $f$ from the statement of Theorem~\ref{tcechar} is a color preserving isomorphism, then it does not follow that $f^{-1}$ is also color preserving.

\begin{example}\label{ecechar}
	Let $X=\{1,\ldots,9\}$. Consider the idempotent $\vep\in T(X)$ whose digraph is given in Figure~\ref{fig6}.
	\begin{figure}[h]
		\[
		\xy
		(44,44)*{\bullet}="1";
		(39,36)*{\bullet}="4a";
		{\ar@(ur,ul)};
		(49,36)*{\bullet}="4c";
		{\ar@{->} "4a";"1"};
		{\ar@{->} "4c";"1"};
		(46,44)*{\text{\footnotesize{$1$}}};
		(51,36)*{\text{\footnotesize{$6$}}};
		(37,36)*{\text{\footnotesize{$5$}}};
		(70,44)*{\bullet}="1";
		(65,36)*{\bullet}="4a";
		{\ar@(ur,ul)};
		(75,36)*{\bullet}="4c";
		{\ar@{->} "4a";"1"};
		{\ar@{->} "4c";"1"};
		(72,44)*{\text{\footnotesize{$2$}}};
		(77,36)*{\text{\footnotesize{$8$}}};
		(63,36)*{\text{\footnotesize{$7$}}};
		(90,44)*{\bullet}="7";
		(90,36)*{\bullet}="8";
		{\ar@(ur,ul)};
		{\ar@{->} "8";"7"};
		(92,44)*{\text{\footnotesize{$3$}}};
		(92,36)*{\text{\footnotesize{$9$}}};
		(105,44)*{\bullet}="17";
		(105,44)*{\bullet}="18";
		{\ar@(ur,ul)};
		(108,44)*{\text{\footnotesize{$4$}}};
		\endxy
		\]
		\caption{The digraph of $\vep$ from Example~\ref{ecechar}.}\label{fig6}
	\end{figure}
	Define $\al,\bt\in T(X)$ by
	\[
	\al=\begin{pmatrix}1&2&3&4&5&6&7&8&9\\1&1&2&1&1&1&1&5&7\end{pmatrix}
	\mbox{ and }
	\bt=\begin{pmatrix}1&2&3&4&5&6&7&8&9\\1&1&1&2&1&1&5&6&1\end{pmatrix}.
	\]
	Then, $\al,\bt\in C(\vep)$, and $\ga^p_*(\al)$ and $\ga^p_*(\bt)$ are as in Figure~\ref{fig7}.
	\begin{figure}[h]
		\[
		\xy
		(40,44)*{\bullet}="1";
		(40,36)*{\bullet}="4";
		{\ar@(ur,ul)};
		(32,36)*{\bullet}="4a";
		(48,36)*{\bullet}="4c";
		{\ar@{->} "4a";"1"};
		{\ar@{->} "4c";"1"};
		{\ar@{->} "4";"1"};
		(45,44)*{\text{\footnotesize{$(1,1)$}}};
		(40,33)*{\text{\footnotesize{$(5,1)$}}};
		(48,33)*{\text{\footnotesize{$(7,2)$}}};
		(32,33)*{\text{\footnotesize{$(2,2)$}}};
		(76,44)*{\bullet}="1";
		(76,36)*{\bullet}="4";
		{\ar@(ur,ul)};
		(68,36)*{\bullet}="4a";
		(84,36)*{\bullet}="4c";
		{\ar@{->} "4a";"1"};
		{\ar@{->} "4c";"1"};
		{\ar@{->} "4";"1"};
		(81,44)*{\text{\footnotesize{$(1,1)$}}};
		(76,33)*{\text{\footnotesize{$(5,1)$}}};
		(84,33)*{\text{\footnotesize{$(6,1)$}}};
		(68,33)*{\text{\footnotesize{$(2,2)$}}};
		\endxy
		\]
		\caption{The digraphs of $\ga^p_*(\al)$ (left) and $\ga^p_*(\bt)$ (right) from Example~\ref{ecechar}.}\label{fig7}
	\end{figure}
	Now $f = \begin{pmatrix}(1,1)&(2,2)&(5,1)&(7,2)\\(1,1)&(2,2)&(5,1)&(6,1)\end{pmatrix}$ is a color preserving isomorphism from 
	$\ga^p_*(\al)$ to $\ga^p_*(\bt)$, but $f^{-1} = \begin{pmatrix}(1,1)&(2,2)&(5,1)&(6,1)\\(1,1)&(2,2)&(5,1)&(7,2)\end{pmatrix}$ is not color preserving. Actually, there is no color preserving isomorphism from $\ga^p_*(\al)$ to $\ga^p_*(\bt)$ such that its inverse is also color preserving. Hence $\al\not\cfn\bt$ in $C(\vep)$ by Theorem~\ref{tcechar}.
\end{example}

\subsection{Conjugacy by permutation in transformation semigroups}\label{subbp}

In this subsection, we compare conjugacy by permutation $\cbp$ with natural conjugacy $\cfn$ in the semigroups of transformations considered in \S\ref{subnat}.

In the symmetric group $\sym(X)$, $\cbp$ is the usual group conjugacy; 
therefore it is determined by the \emph{form} of a permutation $\al$. The form of $\al$ is obtained by replacing each element of $X$ in the cycle decomposition of $\al$ by some generic symbol, say $*$.
For example, the form of $\al=(1)(4)(2\,5)(6\,8)(3\,7\,9)\in S_9$ is $(*)(*)(*\,*)(*\,*)(*\,*\,*)$ ($1$-cycles can be omitted when $X$ is finite).

A similar statement is true for $\cbp$ in any semigroup $S$ of transformations on $X$, where here the \emph{form} of $\al\in S$ is obtained by replacing each element of $X$ in the digraph that represents $\al$ by, say, $*$. However, if $X$ is infinite, we have to extend the digraph $\ga(\al)$ ($\al\in P(X)$) to the digraph $\gae(\al) = (X,E)$, where for all $x,y\in X$, $(x,y)\in E$ if and only if $x\in \dom(\al)$ and $x\al = y$. In other words, $\gae(\al)$ is obtained from $\ga(\al)$ by adding all elements $x\in X\setminus\spa(\al)$ (if any) as isolated vertices.

\begin{lemma}\label{lgae}
	Let $\al,\bt\in P(X)$. Then:
	\begin{enumerate}
		\item[\textup{(1)}] if $\al\in T(X)$, then $\gae(\al)=\ga(\al)$;
		\item[\textup{(2)}] if $\gae(\al)\cong\gae(\bt)$, then $\ga(\al)\cong\ga(\bt)$;
		\item[\textup{(3)}] if $\al,\bt\in T(X)$ or $X$ is finite, then $\gae(\al)\cong\gae(\bt)$ if and only if $\ga(\al)\cong\ga(\bt)$.
	\end{enumerate}
\end{lemma}
\begin{proof}
	Statement (1) is true since for every $\al\in T(X)$, $\spa(\al)=X$; and (2) follows immediately from the definitions of $\gae(\al)$ and $\ga(\al)$. If $\al,\bt\in T(X)$, then the conclusion of (3) follows from (1). If $X$ is finite, then the conclusion of (3) is true since in this case, $|X\setminus \spa(\al)| = |X\setminus \spa(\bt)|$.
\end{proof}

Now, it is obvious that for all $\al,\bt\in P(X)$,
\begin{equation}\label{ecbp}
	\al\cbp \bt\text{ if and only if }\gae(\al)\cong \gae(\bt).
\end{equation}

The next corollary follows from the results stated in \S\ref{subnat}, Lemma~\ref{lgae}, and \eqref{ecbp}.

\begin{cor}\label{ccbpcfn}
	Let $X$ be a nonempty set, $\emptyset\neq Y\subseteq X$, and $n\geq 1$. The following statements are true:
	\begin{enumerate}
		\item[\textup{(1)}] if $|X|\geq3$, then in $P(X)$, $\cbp\,\subset\,\cfn$;
		\item[\textup{(2)}] if $|X|\geq4$, then in $T(X)$, $\cbp\,\subset\,\cfn$;
		\item[\textup{(3)}] if $X$ is infinite, then in $\mi(X)$, $\cbp\,\subset\,\cfn$;
		\item[\textup{(4)}] if $X$ is finite, then in $\mi(X)$, $\cbp\,=\,\cfn$;
		\item[\textup{(5)}] in $\ox$ and in $\mj(X)$, $\cbp\,=\,\cfn$;
		\item[\textup{(6)}] if $|X|\geq4$, $|Y|\geq2$, and $Y\ne X$, then in $T(X,Y)$, $\cbp$ and $\cfn$ are not comparable with respect to inclusion;
		\item[\textup{(7)}] if $n\geq4$, then in $\on$, $\cbp$ and $\cfn$ are not comparable with respect to inclusion;
		\item[\textup{(8)}] if $n\geq3$, then in $\oin$, $\cfn\,\subset\,\cbp$.
	\end{enumerate}
\end{cor}

\subsection{Natural conjugacy in the endomorphism monoid of a finite abelian $G$-set}\label{rightleft1}
\label{ssec:conjGset}
Let $G$ be a group with identity $e$. A (left) $G$-set is a set $X$ together with an action $\cdot\colon G\times X\to X$ such that for all $k,l\in G$, $x\in X$, we have $(kl)\cdot x = k\cdot (l\cdot x)$ and $e\cdot x = x$. Throughout this subsection, we assume that $X$ and $G$ are finite and that $G$ is abelian.

A \emph{$G$-endomorphism} of $X$ is a function $f\in T(X)$ such that $f(k\cdot x)= k\cdot f(x)$ for all $k\in G, x\in X$. Exceptionally, in this subsection, we will compose transformations from right to left, that is, $(f\circ g)(x) = f(g(x))$. This change is for compatibility with the common practice of using left instead of right $G$-sets. With this convention, the set $\End_G(X)$ of all $G$-endomorphisms of $X$ is a submonoid of $T(X)$.

If $x,y\in X$ lie in the same $G$-orbit $O$, then it follows from the commutativity of $G$ that $x$ and $y$ have the same point stabilizer $G_x = G_y$. We set $G_O = G_x$ for any $x\in O$.

Let $f\in \End_G(X)$. The following facts about $f$ are easily checked:
\begin{itemize}
	\item $f$ maps $G$-orbits to $G$-orbits;
	\item if $f(x) = y$ then $G_x \le G_y$;
	\item if $x\in X$ lies in the $G$-orbit $O$, then $f(x)$ determines $f(y)$ for all $y\in O$ as $f(k\cdot x) = k \cdot f(x)$.
\end{itemize}
For every $f\in \End_G(X)$, we let $\Gamma(f)$ be the functional digraph of $f$. Let $K$ be the set of $G$-orbits on~$X$. As $f$ preserves orbits, $\Gamma(f)$ induces a digraph on $K$, which we denote by $K(f)$. 
For $O\in K$, we set $S_f(O) = \{O'\in K: f(O) = f(O')\}$.

For every $f\in \End_G(X)$, a \emph{$G$-trim} of $f$ is a digraph obtained from $K(f)$ by removing vertices according to the following procedure.
\begin{enumerate}
	\item Remove all initial  vertices $O\in K(f)$, for which there is an $O'\in S_f(O)$ satisfying $G_O < G_{O'}$.
	\item After step 1, remove all inital $O\in K(f)$ for which there is a non-initial $O'\in S_f(O)$ satisfying $G_O = G_{O'}$.
	\item After step 2, for each remaining initial orbit $O$ consider the set
	\[
	U_O = \{O'\in S_f(O)\,:\,G_O = G_{O'}\}.
	\]
	Clearly, all orbits in $U_O$ are initial, and if $O'\in U_O$, then $U_{O'} = U_O$. Delete all but one orbit in $U_O$.
\end{enumerate}
It is easy to see that all $G$-trims of $f$ are isomorphic as digraphs by an isomorphism that preserves the point stabilizers of the orbits.  Hence, up to such an isomorphism, we may speak of the $G$-trim of $K(f)$, and use the notation $K^{t}(f)$. We will occasionally need to refer to $G$-trims without this equivalence up to isomorphism, in which case we will speak of a concrete realization of the $G$-trim.

We now extend the digraphs $K(f)$ and $K^t(f)$ by adding information about the cycles in these digraphs. Note that all orbits that are vertices of the same cycle necessarily have the same stabilizer.

Let $O\in K$ be such that $O$ lies in a cycle of $K(f)$ with $n$ vertices. Then $f^n$ maps $O$ to itself. Pick $x\in O$, then $f^n(x) = k\cdot x$ for some $k\in G$. This $k$ is in general not unique, as any element of the coset $kG_O$ can replace $k$.

Using the fact that $G$ is abelian, we obtain $f^n(j\cdot x) = (jk)\cdot x = (kj)\cdot x = k\cdot (j\cdot x)$. It follows that $kG_O$ does not depend on the choice of $x\in O$, and hence we may label the vertex $O$ with $kG_O$. Moreover, it is straightforward to check that all vertices in the cycle containing $O$ will be labeled equally. Because the trim procedure does not remove vertices in cycles, we may carry these labels from $K(f)$ to $K^t(f)$. By abuse of notation, we will call the resulting partially labeled graphs $K(f)$ and $K^t(f)$, as well.

We remark that the unlabeled version of the $G$-trim of $f$ is not necessarily, up to isomorphism, obtained from its trim $\Gamma^t(f)$ by identifying vertices from the same $G$-orbit. Compared to such a construction, the $G$-trim will in general contain additional initial vertices.

\begin{theorem}
	Suppose that $a,b\in \End_G(X)$, and that the $G$-trims $K^{t}(a)$ and $K^{t}(b)$ are isomorphic by a graph isomorphism $g'$ that preserves stabilizers of orbits and the partial labels. Then $a\cfn b$ in $\End_G(X)$.
\end{theorem}
\begin{proof}
	Let $h'$ be the inverse of $g'$. We will transform $g'$ and $h'$ into commutators $g,h\in \End_G(X)$.
	
	We will first define $g$ on the union of those orbits that lie on individual cycles of $K^t(f)$. Given an individual such cycle with $n$ vertices, pick a vertex $O_1$ and label the remaining vertices $O_i = a(O_{i-1})$ for $i = 2,\ldots,n$. Moreover, pick $x_1\in O_1$, and define $x_i = a(x_{i-1})\in O_i$ for $i = 2,\ldots,n$. Finally, pick $y_1\in g'(O_1)$, and set $y_i = b(y_{i-1})$ for $i = 2,\ldots,n$.
	
	For $x\in \bigcup O_i $ with $x = k\cdot x_i$, set $g(x) = k\cdot y_i$. Clearly $g$ is a $G$-isomorphism from $\bigcup O_i$ to $\bigcup g'(O_i)$. Moreover, we claim that for $x\in \bigcup O_i$, we have that $ga(x) = bg(x)$. If $x\notin O_n$, this is straightforward to check. Otherwise, assume that $x = l\cdot x_n$, and that $kG_{O_1}$ was the (common) label of the $O_i$.
	Then
	\begin{align*}
		ga(x) &= ga(l\cdot x_n) = l\cdot (ga(x_n))= l\cdot (ga(a^{n-1}(x_1))) = l\cdot (ga^n(x_1))= l\cdot (g(k\cdot x_1)) \\
		&= (lk)\cdot g(x_1) = (lk)\cdot y_1 = l\cdot (k\cdot y_1) = l\cdot b^n(y_1) = l\cdot b(b^{n-1}(y_1)) \\
		&= l\cdot b(y_{n}) = l\cdot b(g(x_{n})) = bg(l\cdot x_n) = bg(x),
	\end{align*}
	where we used the fact that $g'$ preserves the label $kG_{O_1}$.
	
	We now extend the definition of $g$ to $\bigcup_{O \in K^t(f)} O$ recursively. Suppose that $O\in K^t(a)$ is an orbit such that $g$ is not defined on $O$, but already defined on $a(O)$ (initially, this will only be the case when $a$ maps $O$ into a cycle of $K^t(f)$). Let $O' = g'(O)$ and note that $b(O') = g'a(O)$, as $g'$ is a graph isomorphism.
	
	Pick $x\in O$ and consider $ga(x)\in b(O')$. As $b$ maps $O$ surjectively onto $ga(O)$, there exists a $y\in O'$ with $b(y) = ga(x)$. Now define $g(l\cdot x) = l\cdot y$. Note that $g$ induces a $G$-isomorphism from $O$ to $O'$. Moreover, it satisfies $ga(l\cdot x) = bg(l\cdot x)$ for all $l\in G$ by the construction.
	
	By applying the above construction recursively, we extend $g$ to all of $\bigcup_{O\in K^t(a)} O$. In this way, we obtain a $G$-isomorphism from $\bigcup_{O\in K^t(a)} O$ to its image $\bigcup_{O\in K^t(a)}g'(O)$, which satisfies $ga(x) = bg(x)$, whenever defined.
	
	Define $h$ on $\bigcup_{O\in K^t(f)}g'(O)$ as the inverse function of $g$. As our last step, we extend $g$ and $h$ to $G$-endomorphisms on $X$ that witness the conjugacy of $a$ and $b$.
	
	Let now $O\in K(a)$ be an orbit that is not in $K^t(a)$. It follows from the trim procedure that there exists an orbit $O'\in K^t(a)$ such that $O'\in S_a(O)$ and $G_O\subseteq G_{O'}$. Pick an arbitrary element $x\in O$ and an element $x'\in O'$ such that $a(x) = a(x')$, and set $g(l\cdot x) = l\cdot g(x')$. This is well defined, as $G_O\subseteq G_{O'}\subseteq G_{g(O')}$, and is clearly $G$-compatible. Moreover, we have that
	\[
	ga(l\cdot x) = l\cdot (ga(x)) = l\cdot (ga(x')) = l\cdot (bg(x')) = bg(l\cdot x)\,,
	\]
	for all $l\in G$. Applying this construction to all remaining orbits, we obtain a $G$-endomorphism $g$ that satisfies $ga=bg$.
	
	By a corresponding construction, if $O\in K(b)$, but $O\notin K^t(b)$, we pick an orbit $O'\in K(b)$ with $O'\in S_b(O)$ and $G_O\subseteq G_{O'}$, elements $x\in O$, $x'\in O'$ with $b(x) = b(x')$, and set $h(l\cdot x) = l\cdot h(x')$. This results in a $G$-endomorphism $h$.
	
	To show that $a\cfn b$, it suffices to show $ga = bg$, $ahg = a$, and $ghb = b$ (Proposition \ref{Prp:alternatives}). Note that compared to the notation of the proposition, we swapped the roles of $g$ and $h$ (this is an indirect effect of our choice to compose from right to left). We have already seen that $g$ satisfies $ga = bg$.
	
	Consider the identity $ahg = a$, and let $y\in X$. If $y$ lies in an orbit from $K^t(a)$, then $ahg(y) = ag^{-1}g(y) = a(y)$, where the inverese is considered with respect to the (bijective) restriction of $g$ to $\bigcup_{O\in K^t(a)} O$. Assume instead that $y$ lies in an orbit $O$ with $O\notin K^t(a)$. Let $O',x,x'$ be as in the definition of $g$ on $O$, and let $y = l\cdot x$ for $l\in G$. Then
	\[
	ahg(y) = ahg (l\cdot x) = l\cdot ahg(x) = l\cdot ahg(x') = l\cdot  ag^{-1}g(x') = l\cdot a(x') = l\cdot a(x) = a(l\cdot x) = a(y)\,.
	\]
	It follows that $ahg = a$.
	
	Finally, we note that restricted to the image of $b$, $gh$ acts as the identity and so $ghb = b$. Hence $a\cfn b$, as required.
\end{proof}

\begin{theorem} Suppose that $a\cfn b$ in $\End_G(X)$. Then the $G$-trims $K^{t}(a)$ and $K^{t}(b)$ are isomorphic by a graph isomorphism $g'$ that preserves stabilizers of orbits and partial labels.
\end{theorem}
\begin{proof}
	Let $h,g\in \End_G(X)$ witness $a\cfn b$, and let $h', g'$ be their induced actions on the $G$-orbits of $X$. Note that $h,g$ satisfy all of the equations listed before Lemma \ref{lem:alternatives}. We remark that compared to the notation in the lemma, we exchanged the roles of $h$ and $g$. By replacing $g$ with $(gh)^ig$ for suitable $i > 0$, we may assume without loss of generality that $gh$ and $hg$ are idempotent.
	
	As $ag = gb$ , and $bh = ha$, we see that $g$ and $h$ (and hence also $g'$ and $h'$) are digraph homomorphisms. Moreover, from $ghb = b$ and $hga = a$, we obtain that $g'$ and $h'$ are inverse to each other when restricted to non-initial orbits of $K(a)$ and $K(b)$. This implies that $g'$ and $h'$ preserves the stabilizers of non-initial orbits.
	
	Now consider an initial orbit $O\in K(a)$ such that the set $S_a(O)$ does not contain any orbit $O'$ with $G_O < G_{O'}$ nor a non-initial orbit $O'$ with $G_O = G_{O'}$. Consider the set $T_a(O) = \{O' \in S_a(O)\,:\, G_{O'}=G_O\}$. Note that any concrete choice for the $G$-trim $K^t(a)$ must contain exactly one orbit from each such set $T_a(O)$.
	
	The digraph homomorphism $g'$ maps each $O' \in T_a(O)$ to an orbit  in $S_b(g'(O))$. We claim that $G_{O'} = G_{g'(O')}$. Assuming otherwise that $G_{O'} < G_{g'(O')}$, we obtain that $G_{h'g'(O')} \leq G_{g'(O')} < G_{O'}$, contradicting our choice of $O$, where we note that $h'g'(O')\in S_a(O') = S_a(O)$.
	
	Next we claim that $g'(O')$ is an initial orbit. Assuming otherwise, it follows that $h'g'(O')$ is not an inital orbit either. However, $h'g'(O')\in S_a(O)$, and by essentially the argument from the previous paragraph applied twice, we have that $G_{h'g'(O')} = G_{g'(O')} = G_{O'}=G_{O}$. Once again, this contradicts our choice of $O$.
	
	It follows that $g'$ maps $T_a(O)$ to the correspondingly defined set $T_b(g'(O))$, where $g'(O)$ is an inital orbit such that the set $S_b(g'(O))$ does not contain any orbit $O''\in K(b)$ with $G_{g'(O)} < G_{O''}$ nor a non-initial orbit $O''$ with $G_{g'(O))} = G_{O''}$.
	
	As we assumed without loss of generality that $h'g'$ is idempotent, it follows that there is at least one orbit $O'\in T_a(O)$ such that $O'=h'g'(O')$. Choosing such an orbit for each set of the form $T_a(O)$ (with $O$ constrained as above) and adding it to the non-initial orbits of $K(a)$, we obtain a concrete realization of the $G$-trim $K^t(a)$. It is straightforward to check that the image of this realization under $g'$ is a concrete realization of the $G$-trim of $K(b)$, and that $g'$ and $h'$ are inverse graph isomorphisms between these realizations that moreover preserve stabilizers of orbits.
	
	It remains to show that $g'$ preserves the labels of the vertices lying in cycles of $K^t(a)$. Let $O$ be such an orbit contained in a cycle with $n$ elements and $x\in O$, such that $a^n(x) = l\cdot x$ for some $l\in G$, so that $O$ is labeled $lG_{O}$. Then
	\[
	b^ng(x) = ga^n(x) = g(l\cdot x) = l\cdot g(x)\,,
	\]
	showing that the cycles containing $g'(O)$ has label $lG_{g'(O)} = lG_O$ as well. The result follows.
\end{proof}

\begin{theorem}
	In $\End_G(X)$, two $G$-endomorphisms $a$ and $b$ are naturally conjugate if and only if their $G$-trims $K^t(a)$ and $K^t(b)$ are isomorphic by an isomorphism that preserves stabilizers and partial labels.
\end{theorem}

\section{Conjugacies $\cfn$, $\ctr$, $\cpns$, $\coon$, and $\con$ in finite partition monoids}
\label{Sec:par}
\setcounter{figure}{0}

The \emph{partition monoid} $\Pax$ on a set $X$ has the set of all partitions of $X\cup X'$ as its underlying set, where $X'$ is a disjoint copy of $X$. These monoids originally arose in the study of partition algebras (see, for example, \cite{HaRa05,Ma96}) and subsequently attracted the attention of mathematicians working in semigroup theory (see, for example, \cite{E16,FiLa11,EaGr21}). One reason for the attention is that $\Pax$ contains some important semigroups as subsemigroups, such as $T(X)$ and $\mi(X)$ (see \S\ref{Sec:tra}), as well as the symmetric group $\sym(X)$ on $X$ \cite{EaGr21}.

In this section, we are interested in the finite partition monoid
$\Pan = \Pax$ where $X = \{1,\ldots,n\}$, and in the submonoids $\PBan$ and $\Ban$ of $\Pan$, called the partial Brauer monoids and Brauer monoids, respectively. Our goal is to characterize the conjugacies $\cfn$, $\ctr$, $\cpn$, $\coon$, and $\con$ in these monoids.

Throughout this section, we will identify an equivalence relation $\sim$ on a set $Y$ with the partition $Y/\sim$ induced by $\sim$. It will always be clear from the context how we view $\sim$.

Using the notation from \cite{E16}, we let $\bn=\{1,\ldots,n\}$ and $\bn'=\{1',\ldots,n'\}$. Symbols $x,y,z,k,l,m,\ldots$ will always refer to elements in $\bn$, and $x',y',z',k',l',m',\ldots$ to the corresponding elements in $\bn'$. If $A\subseteq \bn$, then $A'=\{x'\,:\,x\in A\}\subseteq \bn'$.

As is customary, we represent an element $a\in\Pan$ (a partition of $\bn\cup\bn'$) as a simple graph with vertices $1,\ldots,n$ in a row, vertices $1',\ldots,n'$ directly below, and edges drawn so that the connected components of the graph correspond to the blocks of the partition $a$. Such a graph is not unique, so we identify two graphs that have the same connected components. For example, the graph
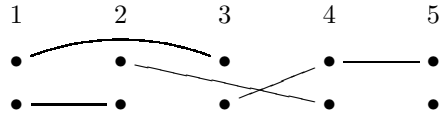
\[
\xymatrix @R=.2cm @C=1cm {1&2&3&4&5\\
	\bullet \ar@{-}@/^.7pc/[rr]& \bullet \ar@{-}[drr] & \bullet & \bullet\ar@{-}[r]\ar@{-}[ld] & \bullet \\
	\bullet \ar@{-}[r]& \bullet &  \bullet& \bullet & \bullet\rlap{\,} }
\]
represents the element $a\in\mathcal{P}_5$ whose blocks are:
$\{1,3\},\{2,4'\},\{1',2'\},\{3',4,5\},\{5'\}$. For $x\in\bn$, $[x]_a$ will denote the block of $a$ containing $x$. Similarly we write $[x']_a$ for the block containing $x'\in\bn'$.

We multiply elements of $\Pan$ as follows. If $a$ is as above and $b$ is represented by the graph
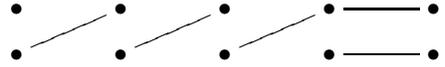
\[
\xymatrix @R=.2cm @C=1cm{ \bullet &\bullet &\bullet &\bullet \ar@{-}[r]&\bullet\\
	\bullet \ar@{-}[ur]&\bullet\ar@{-}[ur] &\bullet\ar@{-}[ur] &\bullet\ar@{-}[r] &\bullet\rlap{\,,} }
\]
then to obtain the product $ab$, we first draw $a$ over $b$:
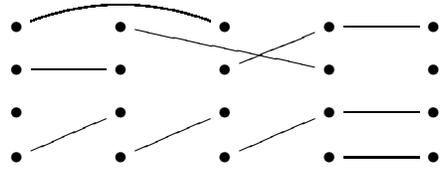
\[
\xymatrix @R=.2cm @C=1cm{ \bullet \ar@{-}@/^.7pc/[rr]& \bullet \ar@{-}[drr] & \bullet & \bullet\ar@{-}[r]\ar@{-}[ld] & \bullet \\
	\bullet \ar@{-}[r]& \bullet &  \bullet& \bullet & \bullet\\
	\bullet &\bullet &\bullet &\bullet \ar@{-}[r]&\bullet\\
	\bullet \ar@{-}[ur]&\bullet\ar@{-}[ur] &\bullet\ar@{-}[ur] &\bullet\ar@{-}[r] &\bullet\rlap{\,,} }
\]
then glue the two middle rows:
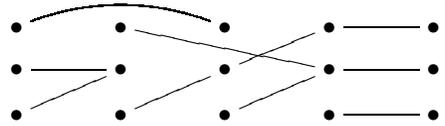
\[
\xymatrix @R=.2cm @C=1cm{ \bullet \ar@{-}@/^.7pc/[rr]& \bullet \ar@{-}[drr] & \bullet & \bullet\ar@{-}[r]\ar@{-}[ld] & \bullet \\
	\bullet \ar@{-}[r] &\bullet &\bullet &\bullet \ar@{-}[r]&\bullet\\
	\bullet \ar@{-}[ur]&\bullet\ar@{-}[ur] &\bullet\ar@{-}[ur] &\bullet\ar@{-}[r] &\bullet\rlap{\,,} }
\]
and finally remove the middle row, keeping in the same block those elements of $X\cup X'$ such that there is a path between these elements in the graph with the middle row:
\[
\xymatrix @R=.2cm @C=1cm {
	\bullet \ar@{-}@/^.7pc/[rr]& \bullet \ar@{-}[dr] & \bullet & \bullet\ar@{-}[r]\ar@{-}[lld] & \bullet \\
	\bullet & \bullet &  \bullet& \bullet & \bullet\ar@{-}[l]\rlap{\,.} }
\]
(See \cite[\S4.1]{EaGr21}.)

Let $a\in\Pan$. Throughout this section, we will need the following definitions:
\begin{align*}
	\ker(a)&=\{[x]_a\cap\bn\,:\,x\in\bn\},\\
	\coker(a)&=\{[x']_a\cap\bn'\,:\,x'\in\bn'\},\\
	\dom(a)&=\{x\in X\,:\,x\text{ belongs to a transversal block of }a\},\\
	\codom^{\wedge}(a)&=\{x\in X\,:\,x'\text{ belongs to a transversal block of }a\},\\
	\coker^{\wedge}(a)&=\{A\subseteq \bn \,:\,A'\in\coker(a)\},\\
	\rank(a)&=\text{the number of transversal blocks of $a$}.
\end{align*}
(We follow \cite[\S2]{DEAFHHL15} and \cite[\S4.2]{EaGr21}, with some changes in names and notation to make our arguments clearer.) We will also need the restriction of $\ker(a)$ and $\coker^{\wedge}(a)$ to $\dom(a)$ and $\codom^{\wedge}(a)$, respectively. For $a\in\Pan$, we define
\begin{equation}\label{ekert}
	\ker^t(a)=\{A\in\ker(a)\,:\,A\subseteq\dom(a)\}\,\text{ and }\,\coker^t(a)=\{B\in\coker^{\wedge}(a)\,:\,B\subseteq\codom^{\wedge}(a)\}.
\end{equation}
Note that for every $A\in\ker^t(a)$, there exists a unique $B\in\coker^t(a)$ such that $A\cup B'$ is a transversal block of $a$; and that $\rank(a)=|\ker^t(a)|=|\coker^t(a)|$.

Now we define the following subsets of $\Pan$:
\begin{align*}
	\PBan&=\{a\in\Pan\,:\,\text{each block of $a$ has size at most }2\},\\
	\Ban&=\{a\in\Pan\,:\,\text{each block of $a$ has size }2\}.
\end{align*}
The subsets $\PBan$ and $\Ban$ are submonoids of $\Pan$ \cite[\S2]{DEAFHHL15},
called \emph{partial Brauer monoids} and \emph{Brauer monoids}, respectively.

\subsection{Conjugacy $\cfn$ in $\Pan$, $\mathcal{BP}_n$, and $\mathcal{B}_n$}
\label{subcfn}

Let $b\in\Pan$. As in previous work on $\Pan$, a special role is played by the equivalence relation $\ker(b)\vee \coker^{\wedge}(b)$.
We say that $b$ is \emph{connected} if $\ker(b)\vee \coker^{\wedge}(b)$ is the universal relation on $\{1, \dots, n\}$. Let $s$ be a block of $b$. We say that $s$ is \emph{transversal} if $s\cap\bn\ne\emptyset$ and $s\cap\bn'\ne\emptyset$. If $b$ does not have any transversal blocks, it is called \emph{transversal free}; if it has exactly one transversal block, it is called $1$-\emph{transversal}.

Let $A\subseteq \bn$ be nonempty. For $b\in \Pan$, let $b_A$ denote the partition of $A\cup A'$ (that is, an element of $\mathcal{P}_A$)
with $[x]_{b_A}=[x]_b \cap (A \cup A')$ and $[x']_{b_A}=[x']_b \cap (A \cup A')$, for all $x \in A$. We call $b_A$ the \emph{subpartition} of $b$ induced by $A$. In this context, for a block $s$ of $b$, we use the notation $s_A = s \cap (A \cup A')$, where this usage is meant to imply that $s$ is a block of $b$ intersecting $A\cup A'$.

A subpartition $b_A$ is called \emph{trivial} if $|A|=1$. The definitions of $b_A$ being connected, transversal free, and $1$-transversal are obtained by adjusting their definitions for $b$ to the index set $A$ in the obvious way. Similarly we extend the definitions of $\ker, \coker, \ker^{\wedge}$, and $\coker^{\wedge}$ to $b_A$.

For the following results, it will be useful to represent an intermediate step in the calculation of a partition product.
Let $\bn^* =\{1^*,\dots, n^*\}$. For partitions $a,b \in \Pan$,
let $(a,b)^*$ denote the partition of the set $\bn \cup \bn^* \cup \bn'$
corresponding to the situation before the final deletion of the middle row, where $\bn, \bn^*, \bn'$ represent the top, middle, and bottom row, respectively. When $a,b$ are represented by specific graphs, we represent $(a,b)^*$ as the graph obtained by identifying corresponding vertices in the lower row of $a$ with those in the upper row of $b$, followed by the merging of all double edges.

Recall that we are identifying partitions with their corresponding equivalence relations. For example, we might write $(x,y)\in b$ instead of $y\in [x]_b$.

\begin{lemma}\label{l:2-bridge}
	Let $b\in \Pan$ such that $b_A$ is connected and transversal-free, it contains blocks $s_A \subseteq A$ and
	$t_A \subseteq A'$, and for every block $r_A\notin\{s_A,t_A\}$, $r_A = r$. Fix $y \in A$ and define $c\in\Pan$ as follows:
	\begin{itemize}
		\item $[y]_c = (s\setminus A)\cup \{y\}$ and $[y']_c=(t\setminus A')\cup \{y'\}$;
		\item $[x]_c = \{x\}$ and $[x']_c = \{x'\}$, for all $x\in A\setminus \{y\}$;
		\item $[x]_c = [x]_b$ if $[x]_b$ does not intersect $A\cup A'$, and $[x']_c = [x']_b$ if $[x']_b$ does not intersect $A\cup A'$.
	\end{itemize}
	Then $b\cfn c$.
\end{lemma}
\begin{proof}
	Define $g\in \Pan$ by $[x]_g = [x]_b$ for $x\in A\setminus s$, $[x]_g = s_A\cup \{y'\}$ for $x\in s_A$,
	$[x']_g = \{x'\}$ for $x'\in A'\setminus \{y'\}$, and $[x]_g = [x']_g = \{x,x'\}$ for $x\notin A$.
	
	Define $h\in \Pan$ by  $[x']_h = [x']_b$ for $x\in A'\setminus t$, $[x']_h = t_A\cup\{y\}$ for $x'\in t_A$,
	$[x]_h = \{x\}$ for $x\in A\setminus \{y\}$, and $[x]_h = [x']_h = \{x,x'\}$ for $x\notin A$.
	
	It is easy to see that $(gh)_A$ is obtained from $b_A$ by merging the upper block $s_A$ with the lower block $t_A$, while outside of $A \cup A'$, $gh$ acts as the identity. Hence, since $b_A$ is connected, $A^*$ is included in a single block of $(gh,b)^*$. Note that $y^*\in A^*$ and that, by the definition of $g$, $(z,y^*)\in(gh,b)^*$ for every $z\in s_A$.
	
	We claim that $ghb = b$. For any $b$-block other than $s$, it is straightforward to check that it is also a $ghb$-block (using the hypothesis that $r_A = r$ for every block $r_A\ne s_A,t_A$). Regarding the block $s$, select any $z\in s_A$. We want to prove that $[z]_{ghb} = s$. Let $x\in s$. If $x\in s_A$, then $x\in [z]_{ghb}$ since $s_A\subseteq [z]_{ghb}$. Suppose $x\in s\setminus s_A$. Then, $(z,y^*)$, $(y^*,z^*)$, and $(z^*,x^*)$
	are in $((gh),b)^*$. Since $(x,x')\in gh$, we also have $(x^*,x)\in(gh,b)^*$. Thus, by the definition of the product in $\Pan$, $(z,x)\in ghb$. Finally, let $x'\in s$. Then, $(z,y^*)$, $(y^*,z^*)$, and $(z^*,x')$ are in $(gh,b)^*$, and so $(z,x')\in ghb$. We have proved that $s\subseteq [z]_{ghb}$, and equality $s = [z]_{ghb}$ follows as all other blocks of $b$ are also blocks of $ghb$. Hence $ghb = b$.
	
	A similar argument shows that $b = bgh$. We now have $g(hbg) = (ghb)g = bg$, $h(b)g = hbg$, and $g(hbg)h = (gh)(bgh) = ghb = b$. Thus $hgb$ and $b$ satisfy (i), (iii), and (iv) of Proposition~\ref{Prp:alternatives}, and so $hbg\cfn b$.
	A straightforward calculation now shows that $hbg = c$, and so $b\cfn c$.
\end{proof}

The following result is similar to Lemma \ref{l:2-bridge}, except that the blocks $s_A$ and $t_A$ are merged.

\begin{lemma}\label{l:1-bridge}
	Let $b\in \Pan$ such that $b_A$ is connected, it has exactly one transversal block $s_A$, and for every block $r_A\ne s_A$, $r_A = r$.
	Fix $y\in A$ and define $c\in \Pan$ as follows:
	\begin{itemize}
		\item $[y]_c = (s\setminus (A\cup A'))\cup\{y, y'\}$;
		\item $[x]_c = \{x\}$ and $[x']_c = \{x'\}$, for all $x\in A\setminus \{y\}$;
		\item $[x]_c = [x]_b$ if $[x]_b$ does not intersect $A\cup A'$, and $[x']_c = [x']_b$ if $[x']_b$ does not intersect $A\cup A'$.
	\end{itemize}
	Then $b\cfn c$.
\end{lemma}
\begin{proof}
	Define $g\in \Pan$ by $[x]_g = [x]_b$ for $x\in A\setminus s$, $[x]_g=(s_A\cap A)\cup\{y'\}$ for $x \in (s_A \cap A)$,
	$[x']_g=\{x'\}$ for $x \in A' \setminus \{y'\}$, and $[x]_g=[x']_g=\{x,x'\}$ for $x \notin A$.
	
	Define $h \in \Pan$ by  $[x']_h=[x']_b$ for $x \in A'\setminus s$, $[x']_h = (s_A\cap A')\cup \{y\}$ for $x'\in (s_A\cap A')$,
	$[x]_h = \{x\}$ for $x\in A\setminus \{y\}$, and $[x]_h = [x']_h = \{x,x'\}$ for $x\notin A$.
	
	Then, as in the proof of Lemma~\ref{l:2-bridge}, we can show that $b=ghb=bgh$ and $c=hbg$. Hence $b\cfn c$.
\end{proof}

\begin{defi}\label{dnor}
	Let $b\in \Pan$. We say that $b$ is in $\frn$-\emph{normal form} if the following conditions hold:
	\begin{enumerate}
		\item in every non-trivial, connected, transversal-free subpartition $b_A$ of $b$, there exist distinct blocks $s_A, t_A$ with $s_A\neq s$ and $t_A\neq t$, such that either $s_A, t_A\subseteq A$ or $s_A, t_A\subseteq A'$;
		\item in every non-trivial, connected, $1$-transversal subpartition $b_A$ of $b$ with transversal $s_A$, there exists
		a block $t_A\neq s_A$ such that $t\neq t_A$.
	\end{enumerate}
\end{defi}

\begin{rem}\label{rnor}
	Applying Lemmas~\ref{l:2-bridge} and~\ref{l:1-bridge} to nontrivial connected sets $A$ results in a partition with an increased number of singleton blocks. It follows that this process must stop, and hence every $\frn$-conjugacy class contains an element in normal form.
\end{rem}

We next show that in each $\frn$-conjugacy class, any partitions
$a$ and $b$ in normal form can be obtained from each other by a permutation of the underlying set $\bn$.

\begin{lemma}\label{l:idem-p}
	Let $a,p\in \mathcal{P}_A$ such that $ap = pa = a$ and $p$ is an idempotent. Suppose that there are $k,l\in \bn$ with $(k,l')\in p$.
	Then $(k,k^*) \in (p,a)^*$ and $(l^*,l') \in (a,p)^*$.
\end{lemma}
\begin{proof}
	Suppose that $p$ is represented by the simple graph with the largest possible number of edges. Since $p = p^2$, $(k,l')$ is in $pp$, and hence it is also in $(p,p)^*$. Since $(k,l')\in p$, we have
	$(l', k^*)\in (p,p)^*$. Hence $(k,k^*)\in (p,p)^*$.
	
	Let $k-\cdots-k^*$ be a shortest path from $k$ to $k^*$ in the graph representing $(p,p)^*$, as obtained from the maximal graph representing $p$. Suppose to the contrary that this path contains a vertex $j'\in A'$. Then, the path has a subpath
	$i_1^* - j'_1 - \cdots - j'_t - i^*_2$, where $t\geq1$. But $t$ must be $1$ since $j'_1-i^*_2$ (by the fact that $p$ is represented by the graph with the largest number of edges) and $k - \cdots-k^*$ is a shortest path from $k$ to $k^*$. We then have $i_1^* - j'_1 - i^*_2$, which implies $(i_1,j'_1),(j'_1,i_2)\in p$. Hence $(i_1,i_2)\in p$, and so $(i^*_1,i^*_2)\in(p,p)^*$. This a contradiction since we can replace $i_1^*-j'_1-i^*_2$
	with $i_1^*-i^*_2$ obtaining a shorter path from $k$ to $k^*$.
	
	Now let $a$ also be represented by the graph with the maximal number of edges. Then because $a = pa$,  every edge in the graph for $(p,p)^*$ with no vertex from $A'$ is also an edge in the graph for $(p,a)^*$. Thus, the path
	$k-\cdots-k^*$ above is also a path in the graph for $(p,a)^*$. Hence $(k,k^*)\in (p,a)^*$.
	
	Dually, we obtain $(l^*,l')\in(a,p)^*$.
\end{proof}


\begin{lemma}\label{l:mirrB} Let $a,p \in \Pan$ such that $pa=ap=a$ and $p$ is an idempotent.
	Let $A$ be a non-empty subset of~$\bn$ such that $a_A$ is connected,
	$\ker(a_A)=\ker(p_A)$, and $\coker(a_A)=\coker(p_A)$. Then:
	\begin{itemize}
		\item [\rm(1)] there is at most one $a$-block $s$ intersecting $A$ such that $s$ is transversal or $s$ is not a block of $p$;
		\item [\rm(2)] there is at most one $a$-block $v$ intersecting $A'$ such that $v$ is transversal or $v$ is not a block of $p$.
	\end{itemize}
\end{lemma}
\begin{proof}
	Since $a_A$ is connected and $\coker(p_A)=\coker(a_A)$, the set $A^*$ is included in a single block of $(p,a)^*$. Suppose to the contrary that (1) is false. Then there are three possible cases.
	\vskip 1mm
	\noindent\textbf{Case 1.} There are distinct transversal $a$-blocks $s$ and $t$ intersecting $A$.
	\vskip 1mm
	We then have $g,k'\in s$ and $h,l'\in t$, where $g,h\in A$. Thus $(g^*,k'),(h^*,l')\in (p,a)^*$,
	and so $[k']_{(p,a)^*} = [l']_{(p,a)^*}$ (as $A^*$ lies within one block). It follows that  $(k',l')\in pa$, and so $(k',l')\in a$
	since $pa = a$. This is a contradiction since $s\ne t$.
	\vskip 1mm
	\noindent\textbf{Case 2.} There are $a$-blocks $s$ and $t$ intersecting $A$ such that $s$ is transversal, $t$ is not transversal, and $t$ is not a $p$-block.
	\vskip 1mm
	As in Case~1, we have $g,k'\in s$, where $g\in A$. Select $h\in t\cap A$. Now, $[h]_p$ needs to be a transversal block, for otherwise $[h]_p = [h]_{pa} = [h]_a = t$ and $t$ is not a $p$-block.
	Hence, by Lemma \ref{l:idem-p}, $(h,h^*) \in (p,a)^*$. We now have $(g^*,k'),(h^*,h)\in (p,a)^*$,
	which implies $(h,k')\in pa$, and so $(h,k')\in a$. This is a contradiction since $t$ is not transversal.
	\vskip 1mm
	\noindent\textbf{Case 3.} There are distinct non-transversal $a$-blocks $s$ and $t$ intersecting $A$ that are not $p$-blocks.
	\vskip 1mm
	Select $g\in s\cap A$ and $h\in t\cap A$. As in Case~2, we obtain $(g,g^*),(h,h^*)\in (p,a)^*$, leading to the contradiction $(g,h)\in a$.
	
	We have proved~(1). Statement~(2) follows by a dual argument.
\end{proof}

The following result is crucial for our characterization of $\cfn$ in $\Pan$.

\begin{prop}\label{t:KerCoker}
	Let $a\in \Pan$ be in normal form, and let $p\in \Pan$ be such that $pa = a = ap$. Then the kernel and cokernel of $p$ consist of singletons.
\end{prop}
\begin{proof}
	Suppose by way of contradiction that there exist distinct $k,l\in \bn$ such that $(k,l)\in p$ or $(k',l')\in p$. By replacing $p$ with its idempotent power, we may assume that $p$ is an idempotent.
	
	Suppose $(k,l)\in p$. Then, since $pa = a$, we have $(k,l)\in a$. Since $a$ is in normal form, it follows that $(k',l')\notin a$. Thus,
	$(k',l')\notin p$ since $ap = a$. It follows that $\ker(a_{\{k,l\}}) = \ker(p_{\{k,l\}})$ and $\coker(a_{\{k,l\}}) = \coker(p_{\{k,l\}})$. By a dual argument, these equalities also hold if $(k',l')\in p$.
	
	Let $A$ be a subset of $\bn$ of maximum size such that $a_A$ is connected and it satisfies $\ker(a_A) = \ker(p_A)$, $\coker(a_A) = \coker(p_A) $. We have $|A|\geq |\{k,l\}| = 2$, so $a_A$ is nontrivial.
	
	By Lemma \ref{l:mirrB}, $a_A$ has at most one transversal block,
	there exists at most one $a$-block $s$ intersecting $A$ such that $s$
	is transversal or $s$ is not a block of $p$, and there exists at most one $a$-block $v$ intersecting $A'$ such that $v$ is transversal or $v$ is not a block of $p$.
	
	Consider the set $H = \{h\in \bn\setminus A: [h]_a\cap A\neq \emptyset, [h]_a\neq s\}$ (here and in the following, we ignore conditions of the form $[h]_a\neq s$ if no exceptional block $s$ exists). We claim that for each $h\in H$, there exists $l_h\in A$ such that $(h', l'_h)\in a$.
	
	For $h\in H$, let $t = [h]_a$. Then $t$ intersects $A$. Since $t\neq s$, $t$ is also a block of $p$, and hence $\ker(a_{A\cup \{h\}}) = \ker(p_{A\cup \{h\}})$. Moreover, $a_{A\cup \{h\}}$ is connected, and hence by the maximality of the size of $A$, we conclude that $\coker(a_{A\cup \{h\}})\neq \coker(p_{A\cup \{h\}})$. This implies that there is an $l_h\in A$ such that $(l'_h,h')\in a$, $(l'_h,h')\notin p$.
	(Note that $\coker(p_{A\cup \{h\}})\subseteq \coker(a_{A\cup \{h\}})$ since $ap = a$.)
	
	Consider the set
	\[
	B = \{x\in \bn\cap s: [x']_a\cap A'\neq \emptyset \}\cup \bigcup\{u : u \mbox{ is an $a$-block with } u\cap A\neq \emptyset, u\neq s\}.
	\]
	(If no exceptional block $s$ exists, interpret the first set as $\emptyset$, and ignore the condition $u\neq s$). By the definition
	of $B$, we have $A\subseteq B$ (so $a_B$ is not trivial), $a_B$ is connected, and every $a$-block intersecting $B$ also intersects~$A$.
	Hence, by Lemma~\ref{l:mirrB}, $s$ is the only $a$-block intersecting $B$ such that $s$ is transversal or $s$ is not a block of $p$. In particular, $a_B$ has at most one transversal block, which, if it exists, equals $s_B$.
	
	Moreover, every $a$-block intersecting $B'$ also intersects $A'$.
	Indeed, let $r$ be an $a$-block intersecting $B'$, say $g'$ is in the intersection. If $g$ lies in the first set from the definition of $B$, then $r$ intersects $A'$ by the definition of $B$.
	Suppose $g\in u$, where $u$ is an $a$-block included in the second set of the definition of $B$. If $g\in A$, then $g'\in r\cap A'$.
	Otherwise, $g\in u\setminus A$. Since $u\neq s$ and $u\cap A\neq \emptyset$, $g\in H$. Hence $(l'_g, g')\in a$, with $l'_g\in A'$, and so $r$ intersects $A'$.
	
	By Lemma~\ref{l:mirrB} and the fact that every $a$-block
	intersecting $B'$ also intersects $A'$, $v$, if it exists, is the only $a$-block intersecting $B'$ such that $v$ is transversal or $v$ is not a block of $p$.
	
	Suppose $a_B$ has a transversal block, which must be equal to both $s_B$ and $v_B$. Then $s = v$ and, since $a$ is normal, there is
	an $a$-block $w$ such that $w\neq s$ (so $w\neq v$), $w$ intersects $B\cup B'$, and $w\neq w_B$. The block $w$ cannot intersect $B$ (by the definition of $B$), so it intersects $B'$. Suppose $a_B$ is transversal free. Then we have either two distinct $a$-blocks intersecting $B$ and extending beyond $B\cup B'$, or two blocks intersecting $B'$ and extending beyond $B\cup B'$. The former is not possible, because only $s$ can extend beyond $B\cup B'$ (by the definition of $B$). In the second case, one of these blocks, say $w$, must differ from $v$.
	
	In either case, we have an $a$-block $w$ such that $w\neq v$, $w$ intersects $B'$, and $w\neq w_B$. Since $v$ is the only $a$-block intersecting $B'$ such that $v$ is transversal or $v$ is not a block of $p$, $w\subseteq \bn'$ and $w$ is a block of $p$. Since $w\neq w_B$, there is $m'\in w\setminus B'$.
	
	Consider the set $A\cup \{m\}$. Because $w$ is also a block of $p$ and it intersects $A'$, we have $\coker(a_{A\cup \{m\}}) = \coker(p_{A\cup \{m\}})$. Thus by the maximality of the size of $A$,
	$\ker(a_{A\cup \{m\}})\neq \ker(p_{A\cup \{m\}})$. However, our construction of $B$ shows that $[m]_a$ does not intersect $B$, and hence it does not intersect $A$. Because $pa = a$, this also holds for $[m]_p$, which implies $\ker(a_{A\cup \{m\}}) = \ker(p_{A\cup \{m\}})$. This is a contradiction, which completes the proof.
\end{proof}

Let $S_n$ be the symmetric group of permutations on $\bn = \{1,\ldots,n\}$. Then $S_n$ acts on $\Pan$ by $a^\sig$ ($a\in \Pan$, $\sig\in S_n$), where $a^\sig$ is obtained by replacing $x$ by $x\sig$ and $y'$ by $(y\sig)'$ in each block of $a$.

For $\sig\in S_n$, define $\lam_\sig = \{\{x,(x\sig)'\} : x\in \bn\}\in \Pan$. Then $\psn = \{\lam_\sig\in \Pan : \sig\in S_n\}$ is the group of units of $\Pan$, which is isomorphic to $S_n$. The mapping $\sig\to \lam_\sig$ is an isomorphism for $S_n$ to $\psn$. Note that for all $a\in \Pan$ and $\sig\in S_n$, $a^\sig = \lam^{-1}_\sig a \lam_\sig$.

We can now characterize the natural conjugacy $\cfn$ in $\Pan$.

\begin{theorem}\label{tcfn}
	In the partition monoid $\Pan$, every $\frn$-conjugacy class contains an element in normal form. Moreover, if $a,b\in \Pan$ are in normal form, then $a\cfn b$ if and only if $b = a^\sig$ for some permutation $\sig\in S_n$.
\end{theorem}
\begin{proof}
	The first statement follows by repeated applications of Lemmas \ref{l:2-bridge} and \ref{l:1-bridge}. To simplify the notation in the proof of the second statement, we will identify any $\sig\in S_n$ with $\lam_\sig\in \psn$. In particular, when we write $\sig^{-1}a\sig$, where $a\in \Pan$, we will mean $\lam_{\sig}^{-1}a\lam_{\sig}$. Let $a,b\in \Pan$ be in normal form. It is clear that if $b = a^\sig$ for some $\sig\in S_n$, then $a\cfn b$.
	
	For the converse, suppose that $a\cfn b$ and let $g,h\in \Pan$ be $\cfn$-conjugators for $a$ and $b$. Let $g_1 = (gh)^ig$, where $i\geq0$ is an integer such that $g_1h$ is an idempotent. It is straightforward to check that $g_1$ and $h$ are also conjugators for $a$ and $b$. Now, let $h_1 = (hg_1)^jh$, where $j\geq 0$ is an integer such that $h_1 g_1$ is an idempotent. Again, we can check that $g_1$ and $h_1$ are conjugators for $a$ and~$b$. By a routine calculation, we can show that $g_1 h_1$ is also an idempotent.
	Therefore, we may assume that $gh$ and $hg$ are idempotents.
	
	By Proposition~\ref{t:KerCoker}, the kernel and cokernel of $gh$ and of $hg$ both consist of singletons. It follows that the same statement holds for $g$ and $h$. Hence, for every $x\in \bn$,
	$[x]_g = \{x,y'\}$ or $[x]_g = \{x\}$, and $[x']_g = \{x',y\}$ or $[x']_g = \{x'\}$, for some $y\in \bn$. The same statement is true for $h$. Since $gh$ is an idempotent, for every $x\in \bn$, either $[x]_{gh} = \{x,x'\}$ or $[x]_{gh} = \{x\}$ and $[x']_{gh} = \{x'\}$. The same statement is true for $hg$.
	
	Define $\sig:\bn\to \bn$ by
	\[
	x\sig = \begin{cases}
		y & \text{if }[x]_g = \{x,y'\}\text{ or }[x']_h = \{x',y\},\\
		x & \text{if }[x]_g = \{x\}\text{ and }[x']_h = \{x'\}.
		\end{cases}
	\]
	By the properties of $g$, $h$, $gh$, and $hg$ stated above, $\sig$ is well defined and $\sig\in S_n$. By the definition of $\sig$, we have $g\subseteq \sig$ and $h\subseteq \sig^{-1}$.
	To conclude the proof, it suffices to show that $\sig b\sig^{-1}=a$.
	
	Since $g\subseteq \sig$ and $h\subseteq \sig^{-1}$, we have
	$a = gbh\subseteq \sig b\sig^{-1}$. For the reverse inclusion, let $x\in \bn$. We will prove that $[x]_{\sig b\sig^{-1}}\subseteq [x]_a$ and $[x']_{\sig b\sig^{-1}}\subseteq [x']_a$.
	
	Suppose $z\in[x]_{\sig b\sig^{-1}}$. If $z=x$, then $z\in[x]_a$. Suppose $z\neq x$. Then, $z\in [x]_{\sig b\sig^{-1}}$
	can only happen when $x\sig = y_1$, $(y_1,y_2)\in b$, and $z\sig = y_2$, for some $y_1,y_2\in \bn$. Note that $y_1\neq y_2$.
	We have $[y_1]_{hg} = \{y_1,y_1'\}$ or $[y_1]_{hg} = \{y_1\}$. The latter is impossible since we would have $[y_1]_{hgb} = \{y_1\}$, but $hgb = b$ and $y_2\in [y_1]_b$. Thus $[y_1]_{hg} = \{y_1,y_1'\}$, so there is $l\in \bn$ such that $(y_1,l')\in h$ and $(l,y_1')\in g$. Hence $l\sig = y_1$, which implies $l = x$ (since $x\sig = y_1$), and so $(x,y_1')\in g$. By symmetry, $(z,y_2')\in g$. We now have $(x,y_1')\in g$, $(y_1,y_2)\in b$, and $(z,y_2')\in g$, which implies $z\in [x]_{gbh}$, and so $z\in [x]_a$.
	
	Suppose $z'\in [x]_{\sig b\sig^{-1}}$. Then, $x\sig = y$, $(y,k')\in b$, and $k\sig^{-1} = z$ (that is, $z\sig = k$), for some $y,k\in \bn$. We have $[y]_{hg} = \{y,y'\}$ or $[y]_{hg} = \{y\}$. The latter is impossible since we would have $[y]_{hgb} = \{y\}$, but $hgb = b$ and $k'\in [y]_b$. Thus $[y]_{hg} = \{y,y'\}$, so there is $l\in \bn$ such that $(y,l')\in h$ and $(l,y')\in g$. Hence $l\sig=y$, which implies $l = x$ (since $x\sig = y$), and so $(x,y')\in g$. Further, we have $[k']_{hg} = \{k,k'\}$ or $[k']_{hg} = \{k'\}$. The latter is impossible since we would have
	$[k']_{bhg} = \{k'\}$, but $bhg = b$ and $y\in [k']_b$. Thus $[k']_{hg} = \{k,k'\} = [k]_{hg}$, so there is $m\in \bn$
	such that $(k,m')\in h$ and $(m,k')\in g$. Hence $m\sig = k$, which implies $m = z$ (since $z\sig = k$), and so $(k,z')\in h$.
	We now have $(x,y')\in g$, $(y,k')\in b$, and $(k,z')\in h$, which implies	$z'\in [x]_{gbh}$, and so $z'\in [x]_a$.
	
	We have proved that $[x]_{\sig b\sig^{-1}}\subseteq [x]_a$. By a dual argument, we obtain $[x']_{\sig b\sig^{-1}}\subseteq [x']_a$.
	It follows that $\sig b\sig^{-1} = a$, and so $b = \sig^{-1}a\sig$, that is, $b = a^\sig$.
\end{proof}


We next prove some consequences of our classification. Recall that $\cfn\subseteq \mathcal{D}$ (Proposition \ref{Prp:D}). In $\Pan$, the $\mathcal{D}$-classes correspond to partitions of the same rank. The following characterizes $\cfn$ on partitions of small rank.

\begin{cor}\label{c:Pnrank0}
	In $\Pan$ the partitions of rank $0$ form a single $\cfn$-class.
\end{cor}
\begin{proof}
	The singleton partition is clearly in $\cfn$-normal form. We claim that it is the only such partition of rank $0$.
	
	If $b$ is any other rank $0$ partition, it contains a nontrivial connected subset. Consider a maximal such subset $A$. Then any block $B$ in $b_A$ must be a block of $b$ for otherwise $b$ would be a transversal by the maximality of $A$. However, this is impossible as $b$ has rank $0$. The set $A$ now witnesses that $b$ is not in normal form, as required.
\end{proof}

\begin{cor}\label{c:Pnrank1}
	In $\Pan$, the partitions of rank $1$ form two distinct $\cfn$-classes, if $n\geq 2$, and a single $\cfn$-class, if $n=1$.
\end{cor}
\begin{proof}
	Let $n\geq 2$. Consider the set $T$ of partitions $b_{x,y'}$ that contain a single $2$-element transversal $\{x,y'\}$ and consists of singletons otherwise. Clearly the elements of $T$ are $\cfn$-normal. By Theorem \ref{tcfn}, the elements of $T$ lie in two different $\cfn$-classes depending on whether $x=y$ or not.
	
	If $b$ is any other rank $1$ transformation, it contains a non-trivial connected subset, and hence a maximal such subset $A$. Similar to Corollary \ref{c:Pnrank0}, we see that $b_A$ can contain at most one block that is not a block of $b$. Moreover, this must be the transversal block of $b_A$, if one is present. It follows that $A$ witnesses that $b$ is not in normal form, as required.

	The result for $n=1$ is trivial.
\end{proof}

We remark that the classes of the corollary can be characterized by the existence or absence of a $1$-transversal connected subpartition.

\begin{cor}\label{c:Pnrankinf}
	As $n \to \infty$, the number of $\cfn$-classes of $\Pan$ consisting of rank $2$ partitions is not bounded.
\end{cor}
\begin{proof}
	In $\Pan$, consider all partitions consisting of singletons and a subpartition from the following list and its infinite generalization:
	
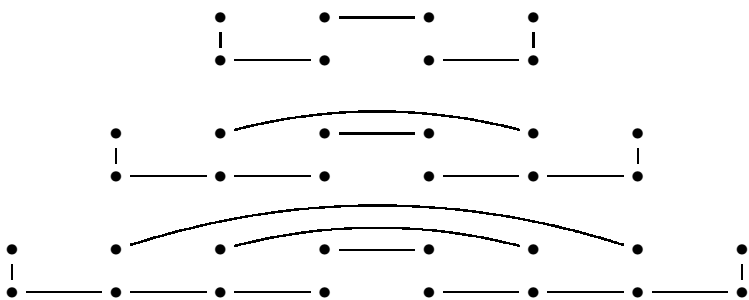
\[ \xymatrix @R=.2cm @C=1cm {
		&&  \bullet\ar@{-}[d]& \bullet\ar@{-}[r] & \bullet & \bullet\ar@{-}[d] & \\
		&&	  \bullet\ar@{-}[r]  & \bullet &  \bullet\ar@{-}[r] & \bullet & \\
		&&&&&\\
		& \bullet\ar@{-}[d] &\bullet\ar@{-}@/^.7pc/[rrr]& \bullet\ar@{-}[r] & \bullet  &\bullet&\bullet\ar@{-}[d]  \\
		& \bullet\ar@{-}[r]&	  \bullet\ar@{-}[r]  & \bullet &  \bullet\ar@{-}[r] &\bullet\ar@{-}[r]& \bullet  \\
		\\
		\bullet\ar@{-}[d] &\bullet\ar@{-}@/^1.4pc/[rrrrr]&\bullet\ar@{-}@/^.7pc/[rrr]& \bullet\ar@{-}[r] & \bullet  &\bullet&\bullet &\bullet\ar@{-}[d]  \\
		\bullet\ar@{-}[r]&	  \bullet\ar@{-}[r]  &	  \bullet\ar@{-}[r] & \bullet &  \bullet\ar@{-}[r] &\bullet\ar@{-}[r]&	  \bullet\ar@{-}[r] & \bullet  \\
	}
	\]
	It is straightforward to check that all such partitions are in normal form, and that a pair $(s,t)$ of them is not in $\cfn$ whenever $s$ and $t$ differ in their nontrivial subpartition. The result follows.
\end{proof}

The above results explains why it is likely not possible to give a more explicit description of the $\cfn$-classes of $\Pan$. If $d\geq 2$, we can construct increasingly complex connected, $\cfn$-normal, and nonconjugate partitions with rank $d$.

For checking practical examples, our results imply which connected subpartitions $A$ of a given size can appear in a $\cfn$-normal partition (together with information about which blocks $t$ satisfy $t_A \neq t$). Without proof, all such subpartitions of size $2$ and $3$ are given below, up to vertical and horizontal permutation. For this list only, a pointed arrow indicates that the corresponding block $t$ satisfies $t_A\neq t$, while the absence of such an arrow allows both $t_A = t$ and $t_A\neq t$.
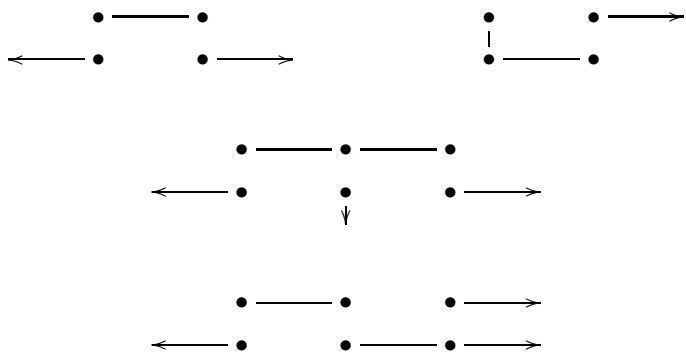
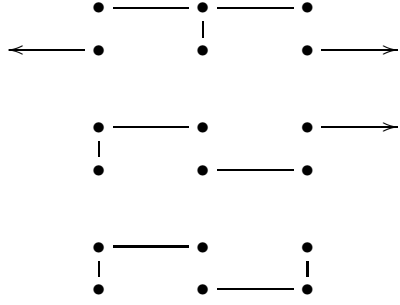
\[
\xymatrix @R=.2cm @C=1cm{ &  \bullet \ar@{-}[r]& \bullet & & & \bullet\ar@{-}[d] &\bullet\ar@{->}[r]&\\
	&\bullet \ar@{->}[l]& \bullet\ar@{->}[r]& & & \bullet\ar@{-}[r] & \bullet\\
	\rlap{\,} }
\]
\[
\xymatrix @R=.2cm @C=1cm{ &  \bullet \ar@{-}[r]& \bullet \ar@{-}[r] &\bullet & \\
	&\bullet \ar@{->}[l]&\bullet\ar@{->}[d] & \bullet\ar@{->}[r]& \\
	& & & & \\
	\rlap{\,} }
\]
\[
\xymatrix @R=.2cm @C=1cm{ &  \bullet \ar@{-}[r]& \bullet  &\bullet \ar@{->}[r]& \\
	&\bullet \ar@{->}[l]&\bullet\ar@{-}[r] & \bullet \ar@{->}[r]& \\
	\rlap{\,} }
\]
\[
\xymatrix @R=.2cm @C=1cm{ &  \bullet \ar@{-}[r]& \bullet \ar@{-}[r] &\bullet & \\
	&\bullet \ar@{->}[l]&\bullet\ar@{-}[u] & \bullet\ar@{->}[r]& \\
	\rlap{\,} }
\]
\[
\xymatrix @R=.2cm @C=1cm{ &  \bullet \ar@{-}[r] \ar@{-}[d]& \bullet  &\bullet\ar@{->}[r] & \\
	&\bullet &\bullet\ar@{-}[r] & \bullet & \\
	\rlap{\,} }
\]
\[
\xymatrix @R=.2cm @C=1cm{ &  \bullet \ar@{-}[r] \ar@{-}[d]& \bullet  &\bullet \ar@{-}[d]& \\
	&\bullet &\bullet\ar@{-}[r] & \bullet & \\
	\rlap{\,} }
\]

We now extend our results to the Brauer monoid $\Ban$ and the partial Brauer monoid $\PBan$. When it is necessary for distinction, we write $\cfn^P$, $\cfn^{B}$ and $\cfn^{PB}$ for the natural conjugacy relation in $\Pan$, $\Ban$ and $\PBan$, respectively. Similarly, we will use expressions such as ``$\frn^{PB}$-normal form''. Clearly, $\cfn^B\subseteq \cfn^{PB}\subseteq \cfn^P$.

It is straightforward to check that in Lemmas \ref{l:2-bridge} and \ref{l:1-bridge}, if $b\in \PBan$, so are the conjugators $g,h$. As conjugacy by a unit is identical in $\PBan$ and $\Pan$, it follows that two partitions in $\PBan$ are conjugate if and only if they are conjugate in $\Pan$. We are moreover able to give a simpler description of our normal form in the case of $\PBan$.

\begin{defi}\label{d:PBnorm}
	Let $b\in \PBan$. We say that $b$ is in $\frn$-\emph{normal form} if the following conditions hold:
	\begin{enumerate}
		\item if $\{x,y\}$ is a block, then $x'$ and $y'$ lie in (necessarily distinct) transversal blocks;
		\item if $\{x',y'\}$ is a block, then $x$ and $y$ lie in (necessarily distinct) transversal blocks.
	\end{enumerate}
\end{defi}

\begin{theorem}\label{tcfnPB}
	In the partial Brauer monoid $\PBan$, every $\frn$-conjugacy class contains an element in normal form. Moreover, if $a,b\in\Pan$ are in normal form, then $a\cfn b$ if and only if $b=a^\sig$ for some permutation $\sig\in S_n$.
\end{theorem}
\begin{proof}
	By the above considerations, it suffices to show that an element $b\in \PBan$ is in $\cfn^{PB}$-normal form if and only if it is in $\cfn^P$-normal form.
	
	Suppose that $b$ is in $\cfn^{PB}$-normal form. Then any nontrivial connected subset $A$ has size $2$, is transversal-free, and one of the $2$ conditions from Definition \ref{d:PBnorm} hold on $A$. It follows that $b$ is in $\cfn^P$-normal form.
	
	Conversely, let $b$ be in $\cfn^P$-normal form. Suppose that $\{x,y\}$ is a block. By normality, $x'$ and $y'$ lie in distinct nonsingleton $b$-blocks. Suppose one, say $x'$, does not lie in a transversal block. Then there is a $z\neq x,y$ such that $\{x',z'\}$ is a block. Consider $B= \{x,y,z\}$. We have that $B$ is connected and nontrivial. If $\{y',z\}$ is a $b$-block, then $b$ would violate the second condition of Definition \ref{dnor}, for a contradiction. However, if $\{y',z\}$ is not a block, then $b_B$ is transversal free, and it is not possible to satisfy the first condition of Definition \ref{dnor}. By contradiction, both $x'$ and $y'$ lie in transversal blocks.
	
	If $\{x',y'\}$ is a block, then a dual argument shows that $x$ and $y$ lie in transversal blocks. The result follows.
\end{proof}

We now turn to the Brauer monoid $\Ban$. Unlike in the previous case, we need modified versions of Lemmas \ref{l:2-bridge} and \ref{l:1-bridge}.

\begin{lemma}\label{l:bridgeB}
	Let $b\in \Ban$ such that $b_A$ is connected with $|A|=3$, say $A = \{x,y,z\}$ with blocks $\{x,y\}$ and $\{y',z'\}$.
	
	If $\{x',z\}$ is not a block, then $b\cfn c$, where $c$ contains the blocks $\{x,y\}, \{x',y'\}, [z]_b, ([x']_b\cup z')\setminus \{x'\}$ as well as all $b$-blocks not intersecting $A\cup A'\cup [z]_b\cup [x']_b$.
	
	If $\{x',z\}$ is a block, then $b\cfn c$, where $c$ contains the blocks $\{x,y\}, \{x',y'\}, \{z,z'\}$ as well as all $b$-blocks not intersecting $A\cup A'$.
\end{lemma}
\begin{proof}
	Define $g\in \Ban$ with blocks $\{x,y\}, \{z,z'\}, \{x',y'\}$ and $\{w,w'\}$ for all $w \not\in A$; define $h\in \Ban$ with blocks $\{x,y\}, \{z,x'\}, \{y',z'\}$ and $\{w,w'\}$ for all $w\not\in A$.
	In either of the above cases, it is straightforward to check that $g,h$ witness $b\cfn c$.
\end{proof}

\begin{defi}\label{d:Bnorm}
	Let $b\in \Ban$. We say that $b$ is in $\frn$-\emph{normal form} if the following conditions hold:
	\begin{enumerate}
		\item if $\{x,y\}$ is a block, then either $\{x',y'\}$ is a block, or $x'$ and $y'$ lie in (necessarily distinct) transversal blocks;
		\item if $\{x',y'\}$ is a block, then either $\{x,y\}$ is a block, or $x$ and $y$ lie in (necessarily distinct) transversal blocks.
	\end{enumerate}
\end{defi}

\begin{theorem}\label{tcfnB}
	In the Brauer monoid $\Ban$, every $\frn$-conjugacy class contains an element in normal form. Moreover, if $a,b\in\Pan$ are in normal form, then $a\cfn b$ if and only if $b = a^\sig$ for some permutation $\sig\in S_n$.
\end{theorem}
\begin{proof}
	Let $b\in \Ban$. If $B$ is a connected subset of $b$ with $|B|\geq 3$, then there is a connected set $A\subseteq B$ that satisfies the conditions of Lemma \ref{l:bridgeB}. Any application of the lemma will increase the number of maximal connected subsets. Hence after repeated application of the lemma, we reach a conjugate $c$ of $b$ that only contains connected subsets of size at most $2$. This is equivalent to $c$ being in normal form.
	
	Assume now that $b\cfn^B c$ with $b,c$ in $\frn^B$-normal form. Then $b\cfn^P c$. Let $b^*, c^*$ be some $\frn^P$-normal forms of $b,c$ obtained by repeated application of Lemmas \ref{l:2-bridge} and \ref{l:1-bridge}.
	
	By Theorem \ref{tcfn}, $b^* = \lambda_\omega c^* \lambda_\omega^{-1}$ for some permutation $\omega$. By replacing $c$ with $c^{\omega}$, we may assume without loss of generality that $b^* = c^*$. Because $b,c$ are in $\cfn^B$-normal form, the only nontrivial applications of Lemmas \ref{l:2-bridge} and \ref{l:1-bridge} to $b,c$ involve Lemma \ref{l:2-bridge} on a connected set $A = \{x,y\}$ with blocks $\{x,y\}$ and $\{x',y'\}$. The same also holds for the outcome of such an application. It follows that $b^*,c^*$ are obtained from $b,c$ by replacing all blocks in such subpartitions with singletons.
	
	Let $D\subseteq \bn$ be the largest set for which $b_D^* = c_D^*$ consist of singleton blocks. Then $|D|$ is even, and there are two partition $D_i^b$, $D_j^c$ of $D$ into blocks of size two such that $b_{D_i^b}$, $c_{D_j^c}$ consist of two nontransversal blocks each, for all $i$ and $j$. In addition, on the complement $\bar{D} = \bn\setminus D$, we have that $b_{\bar D} = b^*_{\bar D} = c^*_{\bar D} = c_{\bar D}$. The result now follows.
\end{proof}

\subsection{Conjugacy $\ctr$ in $\Pan$, $\PBan$, and $\Ban$}
\label{subctr}

To characterize trace conjugacy $\ctr$ (see \eqref{ectr}) in $\Pan$, we first need to describe the group elements of $\Pan$. Let $S$ be any semigroup. The maximal subgroups of $S$ are the $\gh$-classes $H_e$ of $S$ such that $e$ is an idempotent \cite[Ex.~1, p.~61]{ClPr64}. An element $a\in S$ is a \emph{group element} of $S$ if $a\in H_e$ for some idempotent $e\in S$. These elements are also called completely regular, as in \S\ref{Sec:epi}.

\begin{lemma}\label{lgrp}
	Let $a,b\in \Pan$. Then:
	\begin{itemize}
		\item[\textup{(1)}] $a\greenR b\iff \ker(a) = \ker(b)$ and $\ker^t(a) = \ker^t(b)$;
		\item[\textup{(2)}] $a\greenL b\iff \coker(a) = \coker(b)$ and $\coker^t(a) = \coker^t(b)$.
	\end{itemize}
\end{lemma}
\begin{proof}
	By \cite[Prop.~4.2]{EaGr21}, (1) and (2) are true if $\ker^t$ and $\coker^t$ are replaced by $\dom$ and $\codom^{\wedge}$,
	respectively. If $\ker(a) = \ker(b)$, then $\dom(a) = \dom(b)\iff \ker^t(a) = \ker^t(b)$; and if $\coker(a) = \coker(b)$, then $\codom^{\wedge}(a) = \codom^{\wedge}(b)\iff \coker^t(a) = \coker^t(b)$. The result follows.
\end{proof}

We also have $a\greenD b\iff\rank(a)=\rank(b)$, and $\gd=\gj$ \cite[Prop.~4.2]{EaGr21}.

For equivalence relations $\rho_1$ and $\rho_2$ on $X$, the \emph{join} $\rho_1\lor \rho_2$ of $\rho_1$ and $\rho_2$ is the smallest equivalence relation containing the union $\rho_1\cup \rho_2$. To describe the group elements of $\Pan$, we will need the join $\ker(a)\lor \coker^{\wedge}(a)$, where $a\in \Pan$.

First, the idempotents of $\Pan$ were described in \cite[Thm.~5]{DEAFHHL15}.

\begin{lemma}\label{lide}
	Let $e\in \Pan$. Then $e$ is an idempotent if and and only if the following two conditions are satisfied:
	\begin{itemize}
		\item[\textup{(1)}] for every transversal block $A\cup B'$ of $e$, there exists a (necessarily unique) block $P$ of $\ker(e)\lor \coker^{\wedge}(e)$ such that $A\cup B'\subseteq P\cup P'$;
		\item[\textup{(2)}] for every block $P$ of $\ker(e)\lor \coker^{\wedge}(e)$, $P\cup P'$ contains at most one transversal block of $e$.
	\end{itemize}
\end{lemma}

\begin{prop}\label{pgrel}
	Let $a\in \Pan$. Then $a$ is an element of a group $\gh$-class of $\Pan$ if and only if for every block $P$ of $\ker(a)\lor \coker^{\wedge}(a)$ one of the following conditions holds:
	\begin{itemize}
		\item[\textup{(a)}] neither $P$ nor $P'$ intersects a transversal block of $a$; or
		\item[\textup{(b)}] each of $P$ and $P'$ intersects exactly one (not necessarily the same) transversal block of $a$.
	\end{itemize}
\end{prop}
\begin{proof}
	Suppose that $a$ is an element of a group $\gh$-class $H$ of $\Pan$. Let $e$ be the identity of $H$, so $a\greenH e$. By Lemma~\ref{lgrp}, $\ker(a)\lor \coker^{\wedge}(a) = \ker(e)\lor \coker^{\wedge}(e)$, $\ker^t(a) = \ker^t(e)$, and $\coker^t(a) = \coker^t(e)$. Let $P$ be a block of $\ker(a)\lor \coker^{\wedge}(a)$.
	
	Suppose that $P$ does not intersect any transversal block of $a$. Suppose to the contrary that $P'$ intersects some transversal block $A\cup B'$ of $a$. Then $B'\subseteq P'$ and $B'\in \coker^t(a)$.
	Since $\coker^t(a) = \coker^t(e)$, it follows by Lemma~\ref{lide}
	that there is $C\in \ker^t(e)$ such that $C\cup B'\subseteq P\cup P'$. Since $\ker^t(e) = \ker^t(a)$ and $C\subseteq P$,
	the block $P$ intersects some transversal block of $a$, which is a contradiction. We have proved that if $P$ does not intersect any transversal block of $a$, then (a) holds. Similarly, (a) holds
	if $P'$ does not intersect any transversal block of $a$.
	
	Suppose (a) does not hold. Then $P$ intersects some transversal block $A\cup B'$ of $a$. If it also intersects another transversal block of $a$, say $C\cup D'$, then we would have $A,C\in \ker(e)$, $A,C\subseteq P$, and $A\ne C$, which would contradict Lemma~\ref{lide}(2). A similar argument can be applied to $P'$, which implies that (b) holds.
	
	Conversely, suppose that for every block $P$ of $\ker(a)\lor \coker^{\wedge}(a)$, (a) or (b) holds. Let $k(a)$ be the number of blocks $P$ such that $P$ intersects a transversal block $A\cup B'$ of $a$, and $P'$ intersects a different transversal block $C\cup D'$ of $a$. If $k(a) = 0$, then $a$ is an idempotent (and so a group element) by Lemma~\ref{lide}. Let $k(a)\geq 1$ and consider $P$, $A\cup B'$, and $C\cup D'$ as above. Then, $A\subseteq P$, $D'\subseteq P'$, $B'\subseteq Q'$, and $C\subseteq R$, where $Q$ and $R$ are blocks of $\ker(a)\lor \coker^{\wedge}(a)$ such that $P\notin \{Q,R\}$. Construct $a_1\in \Pan$ by replacing in $a$ the transversal blocks $A\cup B'$ and $C\cup D'$ by $A\cup D'$ and $C\cup B'$. Then $k(a_1) < k(a)$ (since $P$ and $P'$ both
	intersect the same transversal block of $a_1$, namely $A\cup D'$), and it is straightforward to check, using Lemma~\ref{lgrp}, that $a\greenH a_1$. Applying this construction repeatedly, we obtain (after at most $k(a)$ steps) an element $e\in \Pan$ such that $k(e) = 0$ (so $e$ is an idempotent) and $a\greenH e$. Hence $a$ is a group element.
\end{proof}

Let $\sig\in S_m$, where $S_m$ is the symmetric group of permutations on $[m]=\{1,\ldots,m\}$. We allow
$m$ to be zero, in which case $[m]=\emptyset$, $S_m=\{\emptyset\}$, and $\sig=\emptyset$.
The \emph{cycle type} of $\sig$ is the sequence
$(k_1,\ldots,k_m)$, where $k_i$ is the number of cycles of length $i$ in the cycle-decomposition of $\sig$.
If $m=0$, then we define the cycle type of $\sig$ as $(0)$.

\begin{defi}\label{dpia}
	Let $a\in \Pan$ be a group element. By Proposition~\ref{pgrel}, for every block $P$ of $\ker(a)\lor \coker^{\wedge}(a)$, either $P$ does not intersect any transversal block of $a$ or there is a unique $A\in \ker^t(a)$ such that $A\subseteq P$. Let $\{P_1,\ldots,P_m\}$ be the set of all blocks of $\ker(a)\lor \coker^{\wedge}(a)$ that intersect some transversal block of $a$. For each $i\in[m]$, let $A_i$ be a unique element of $\ker^{t}(a)$ such that $A_i\subseteq P_i$. Note that $\ker^t(a) = \{A_1,\ldots,A_m\}$. By Proposition~\ref{pgrel} again, each $P_i'$ contains a unique $B_i'\in \coker^t(a)$ and $\coker^t(a) = \{B_1',\ldots,B_m'\}$.
	Note that $m$ can be $0$, which happens when $\ker^t(a) = \coker^t(a) = \emptyset$.
	
	Define $\tau_a:[m]\to [m]$ by
	\[
	i\tau_a=j\ \iff\ A_i\cup B_j'\text{ is a transversal block of } a\,.
	\]
	By Proposition~\ref{pgrel}, $\tau_a\in S_m$. We define the \emph{cycle type} of $a$ to be the cycle type of $\tau_a$. Note that while $\tau_a$ depends on the chosen ordering of $\{P_1,\ldots,P_m\}$, its cycle type does not, and hence the cycle type of $a$ is well-defined. We moreover remark that if $a = \lambda_\omega$ for some $\omega\in S_n$, then the cycle types of $\omega$ and $a$ agree, as can be easily checked.
	
	Let $e$ be the idempotent in the group $\gh$-class of $a$. Then the transversal blocks of $e$ are $A_1\cup B'_1,\ldots,A_m\cup B'_m$,
	and the transversal blocks of $a$ are $A_1\cup B'_{1\tau_a},\ldots,A_m\cup B'_{m\tau_a}$.
\end{defi}

\begin{lemma}\label{lkte}
	Let $e,f,g,h\in \Pan$ such that $e$ and $f$ are idempotents, $gh = e$, $hg = f$, $ghg = g$, and $hgh = h$. Then $\ker^t(g) = \ker^t(e)$ and $\coker^t(g) = \coker^t(f)$.
\end{lemma}
\begin{proof}
	We have $g\greenR e$ (since $gh = e$ and $eg = ghg = g$) and $g\greenL f$ (since $hg = f$ and $gf = ghg = g$). Thus by Lemma~\ref{lgrp}, $\ker^t(g) = \ker^t(e)$ and $\coker^t(g) = \coker^t(f)$.
\end{proof}

We can now characterize trace conjugacy $\ctr$ in $\Pan$.

\begin{theorem}\label{tctr}
	Let $a,b\in \Pan$. Then $a\ctr b$ if and only if $a^{\ome+1}$ and $b^{\ome+1}$ have the same cycle type.
\end{theorem}
\begin{proof}
	Let $e = a^{\ome}$, $f = b^{\ome}$, $u = a^{\ome+1}$, and $v = b^{\ome+1}$. Suppose that $a\ctr b$. By \eqref{ectr}, there exist
	$g,h\in \Pan$ such that
	\[
	ghg = g,\, hgh = h,\, gh = e,\, hg = f,\text{ and } hug = v.
	\]
	We also have $gvh = ghugh = eue = u$. By Lemma~\ref{lkte} and the fact that $u\greenH e$ and $v\greenH f$, we have $\ker^t(g) = \ker^t(e) = \ker^t(u)$, $\coker^t(g) = \coker^t(f) = \coker^t(v)$,
	$\ker^t(h) = \ker^t(f) = \ker^t(v)$, and $\coker^t(h) = \coker^t(e) = \coker^t(u)$. Let $m = |\ker^t(e)|$. Then by the above equations, $|\ker^t(f)| = |\ker^t(u)| = |\ker^t(v)| = |\ker^t(g)| = |\ker^t(h)| = m$.
	
	Let $\{P_1,\ldots,P_m\}$ be the set of all blocks of $\ker(e)\lor \coker^{\wedge}(e)$ that intersect some transversal block of $e$, and let $\{Q_1,\ldots,Q_m\}$ be the set of all blocks of $\ker(f)\lor \coker^{\wedge}(f)$ that intersect some transversal block of $f$ (see Definition~\ref{dpia}). (We have the same $m$ since $|\ker^t(e)| = |\ker^t(f)| = m$.) Since $e$ and $f$ are idempotents, the transversal blocks of $e$ and of $f$ are, respectively, $A_i\cup B'_i$ with $A_i\subseteq P_i$ and $B'_i\subseteq P'_i$, and $C_i\cup D'_i$ with $C_i\subseteq Q_i$ and $D'_i\subseteq Q'_i$, where $i\in [m]$. Since $u\in H_e$ and $v\in H_f$, the transversal blocks of $u$ and of $v$ are, respectively, $A_i\cup B'_{i\tau_u}$ and $C_i\cup D'_{i\tau_v}$, where $i\in [m]$ (see Definition~\ref{dpia}). Since $\ker^t(g) = \ker^t(e)$ and $\coker^t(g) = \coker^t(f)$, there is $\sig\in S_m$ such that the transversal blocks of $g$ are $A_i\cup D'_{i\sig}$, where $i\in [m]$.
	Finally, since $\ker^t(h) = \ker^t(f)$ and $\coker^t(h) = \coker^t(e)$, there is $\del\in S_m$ such that the transversal blocks of $h$ are $C_i\cup B'_{i\del}$, where $i\in [m]$.
	
	We claim that $\sig = \del^{-1}$. Let $i\in [m]$. Since $A_i\cup D'_{i\sig}$ is a block of $g$ and $C_{i\sig}\cup B'_{i(\sig\del)}$ is a block of $h$, we conclude that $A_i\cup B'_{i(\sig\del)}$ is a block of $gh$. Further, $e = gh$ and $A_i\cup B'_i$ is a block of $e$, which implies $i(\sig\del) = i$. Hence $\sig = \del^{-1}$.
	
	Our second claim is that $\sig\tau_u\del = \tau_v$. Let $i\in [m]$. Since $A_i\cup D'_{i\sig}$ is a block of $g$ and $C_{i\sig}\cup D'_{i(\sig\tau_v)}$ is a block of $v$, we conclude that $A_i\cup D'_{i(\sig\tau_v)}$ is a block of $gv$. Thus since $C_{i(\sig\tau_v)}\cup B'_{i(\sig\tau_v\del)}$ is a block of $h$, it follows that $A_i\cup B'_{i(\sig\tau_v\del)}$ is a block of $gvh$. But $gvh = u$ and $A_i\cup B'_{i\tau_u}$ is a block of $u$, which implies $i(\sig\tau_v\del) = i\tau_u$. Hence $\sig\tau_u\del = \tau_v$.
	
	Thus $\del^{-1}\tau_u\del = \tau_v$, and so $\tau_u$ and $\tau_v$ are group conjugate in $S_m$. Hence $\tau_u$ and $\tau_v$ have the same cycle type, and so $a^{\ome+1}\,\,(=u)$ and $b^{\ome+1}\,\,(=v)$ have the same cycle type (see Definition~\ref{dpia}).
	
	Conversely, suppose that $a^{\ome+1}$ and $b^{\ome+1}$ have the same cycle type. Then $\tau_u$ and $\tau_v$ are group conjugate in $S_m$, that is, there are $\sig,\del\in S_m$ such that $\sig = \del^{-1}$ and $\sig\tau_u\del = \tau_v$. With the notation for the transversal blocks of $e$, $f$, $u$, and $v$ as in the first part of the proof,
	let $g\in \Pan$ be such that $\ker(g) = \ker(e)\,\,(=\ker(u))$, $\coker(g) = \coker(f)\,\,(=\coker(v))$, and the transversal blocks
	of $g$ are $A_i\cup D_{i\sig}$, where $i\in [m]$. Similarly, let $h\in \Pan$ be such that $\ker(h) = \ker(f)\,\,(=\ker(v))$,
	$\coker(h) = \coker(e)\,\,(=\coker(u))$, and the transversal blocks
	of $h$ are $C_i\cup B_{i\del}$, where $i\in [m]$. Simple calculations (similar to the ones in the first part of the proof)
	show that $ghg = g$, $hgh = h$, $gh = e$, $hg = f$, and $hug = v$. Hence $a\ctr b$.
\end{proof}

Turning to $\PBan$ and $\Ban$, it is clear that $\ctr^B\, \subseteq\,\, \ctr^{PB}\, \subseteq\,\, \ctr^P$, and hence for two $\ctr$-conjugate partitions $a,b\in \PBan$ or $\Ban$, $a^{\omega+1}$ and $b^{\omega+1}$ have the same cycle type. Conversely, if $a,b$ are two such partitions in $\PBan$ [in $\Ban$], it is straightforward to check that the  conjugators $g,h$ constructed in the second part of Theorem \ref{tctr} lie in $\PBan$ [in $\Ban]$. Hence we obtain the following characterization.

\begin{theorem}\label{t:trPBB}
	Let $a,b\in \PBan$ or $a,b\in \Ban$. Then $a\ctr b$ if and only if $a^{\ome+1}$ and $b^{\ome+1}$ have the same cycle type.
\end{theorem}

\subsection{Conjugacy $\cpns$ in $\Pan$, $\PBan$, and $\Ban$}
\label{subcpn}

In any epigroup, $\cpns\,\,\subseteq\,\,\ctr$ \cite[Thm.~4.8]{ArKiKoMaTA}. The reverse inclusion is not true in the class of epigroups \cite[Thm.~4.15]{ArKiKoMaTA}. The goal of this subsection is to show that in $\Pan$, $\cpns\,\,=\,\,\ctr$.
(See \eqref{ecp} and \eqref{econp} for the definitions of $\cpn$ and $\cpns$.) We note that conjugacy $\cpns$ in $\Pan$, $\Ban$, and $\PBan$ was characterized, with a different proof than the one given in this subsection, in \cite{KuMa07}.

\begin{lemma}\label{lcpnsa}
	Let $a\in \Pan$, and $s\subseteq \bn$ a nontransversal $a$-block, such that $s'$ intersects one (or more) transversal $a$-blocks. Then $a$ has a $\cpn$-conjugate $c\in \Pan$ such that $c_s$ is transversal-free, and such that $c$ has more blocks than $a$.
\end{lemma}
\begin{proof}
	Let $u\in \Pan$ have the blocks $s$, $\{z'\}$, where $z\in s$, and $\{k,k'\}$, where $k\notin s$. By straightforward calculations, we check that $ua = a$. The partition $c = au$ has blocks $t\setminus s'$, for every $a$-block $t$ satisfying $t\not\subseteq s'$, and $\{z'\}$ for $z \in s$. Clearly $c_s$ is transversal-free. As we assumed that at least one transversal $a$-block intersects $s'$, $c$ has more blocks than $a$.
\end{proof}

A dual result clearly holds if $s'$ is a nontransversal block such that $s$ intersects a transversal block.

\begin{lemma}\label{lcpnsb}
	Let $a\in \Pan$, $s$ an $a$-block, $A = s\cap \bn$, such that $A'$ intersects two different $a$-blocks $t_1, t_2$ (one of which might be $s$). Then $a\cpn c$, where $c$ is obtained from $a$ by merging the blocks $t_1, t_2$.
\end{lemma}
\begin{proof}
	Let $x, y\in A$, with $x'\in t_1, y'\in t_2$. Let $v\in \Pan$ have the blocks $\{x,y,x',y'\}$ and $\{z,z'\}$, where $z\notin \{x,y\}$.
	By straightforward calculations, we check that $va = a$ and that $av$ has the desired properties.
\end{proof}

Once again, the dual version of the Lemma~\ref{lcpnsb} clearly holds as well.

\begin{prop}\label{pacpnsc}
	Let $a\in \Pan$. Then there exists a group element $c\in \Pan$ such that $a\cpns c$.
\end{prop}
\begin{proof}
	We recursively apply Lemma \ref{lcpnsa} [or its dual] to $a$, as long as we find a nontransversal block $s$ [resp. $s'$] such that $s'$ [resp. $s$] intersects a transversal block. Because the number of blocks increases at each step, this process must stop with a partition $b\cpns a$ for which $\dom(b) = \codom^{\wedge}(b)$.
	
	We next apply Lemma \ref{lcpnsb} (or its dual) to all cases in which the involved blocks $t_1, t_2$ are transversal (note that this means that $s$ is also transversal). Each such application will preserve the condition $\dom(\cdot) = \codom^{\wedge}(\cdot)$, as only transversal blocks will be merged. As this decreases the number of blocks, this process will stop with an element $c\cpns b\cpns a$
	such that
	\begin{enumerate}
		\item $\dom(c) = \codom^{\wedge}(c)$;
		\item if $s$ is a transversal $c$-block, $A = s\cap \bn$, then $A'$ intersects at most one transversal $c$-block;
		\item if $s$ is a transversal $c$-block, $A'= s\cap \bn'$, then $A$ intersects at most one transversal $c$-block.
	\end{enumerate}
	We will show that these conditions imply that $c$ is a group element. Let $P$ be a block of $\ker(c)\vee \coker^{\wedge}(c)$. If $P$ does not intersect any transversal block of $c$, then, by 1., neither does $P'$ (and vice versa).
	
	Suppose that $s = A\cup B'$ is a transversal $c$-block, and let $P$ and $Q$ be the blocks of $\ker(c)\vee \coker^{\wedge}(c)$ such that $A\subseteq P$ and $B'\subseteq Q'$. We claim that $s = P\cup Q'$. By item 1, any block intersected by $A'$ must be transversal. Thus by item 2, there exists a transversal $c$-block $t$ such that $A'\subseteq C'$, where $C' = t\cap \bn'$. Applying the dual argument to $C'$ and using item 3, we obtain a transversal $c$-block $w$ such that $C\subseteq D$, where $D = w\cap \bn$. Since $A'\subseteq C'$, we have $A\subseteq C\subseteq D$, so $A\subseteq s\cap w$.  Thus, $s = w$, $A = C = D$, and $A' = C' = D'$.
	
	We will now prove that $A = P$. Let $x\in P$ and select any $y\in A$. Since $A\subseteq P$, we have $(y,x)\in \ker(c)\vee \coker^{\wedge}(c)$, and so there are $y = z_0,z_1,\ldots,z_k = x$ in $\bn$ such that for every $i\in\{0,\ldots,k-1\}$, either $(z_i,z_{i+1})\in \ker(c)$ or $(z_i,z_{i+1})\in \coker^{\wedge}(c)$. Let $i\in\{0,\ldots,k-1\}$ and suppose that $z_i\in A$. If $(z_i,z_{i+1})\in \ker(c)$, then $z_{i+1}\in A$. Suppose $(z_i,z_{i+1})\in \coker^{\wedge}(c)$, that is, $(z'_i,z'_{i+1})\in \coker(c)$. Then $z'_i\in C'$ (since $A' = C'$), and so $z'_{i+1}\in C'$ (since $C'\subseteq t$). Thus $z_{i+1}\in C$, and so $z_{i+1}\in A$. Since $y = z_0\in A$, it follows that $x = z_k\in A$, and so $P = A$.
	
	By a dual argument, $B' = Q'$, and so $s = P\cup Q'$. Hence $c$ is a group element by Proposition~\ref{pgrel}.
\end{proof}

\begin{theorem}\label{tcpns}
	In $\Pan$, $\cpns\,\,=\,\,\ctr$. That is, for $a,b \in \Pan$, $a \cpns b$ if and only if $a^{\ome+1}$ and $b^{\ome+1}$ have the same cycle type.
\end{theorem}
\begin{proof}
	Let $a,b\in\Pan$. Suppose that $a\ctr b$. By Proposition~\ref{pacpnsc}, there are group elements $c$ and $d$ of $\Pan$ such that $a\cpns c$ and $b\cpns d$. Since $\cpns\,\,\subseteq\,\,\ctr$, we have $c\ctr a\ctr b\ctr d$, and so $c\ctr d$. By \cite[Thm.~4.15]{ArKiKoMaTA}, as relations on the group elements of any semigroup, $\cpn\,\,=\,\,\cpns\,\,=\,\,\ctr$.
	Thus $c\cpn d$, and so $a\cpns c\cpn d\cpns b$, which implies $a\cpns b$. We have proved that $\ctr\,\,\subseteq\,\,\cpns$.
	Since $\cpns\,\,\subseteq\,\,\ctr$ in any epigroup, $\cpns\,\,=\,\,\ctr$.
\end{proof}

Let $a,b\in \Pan$. We can check if $a$ and $b$ are $\mathrm{p}^*$-conjugate (equivalently, $\mathrm{tr}$-conjugate)
in two ways. We can calculate the successive positive powers of $a$ and $b$ until we obtain idempotents $e$ and $f$, respectively. Then we check if $ea\,\,(=a^{\ome+1})$ and $fb\,\,(=b^{\ome+1})$ have the same cycle type. Or, using Proposition~\ref{pacpnsc} and Lemmas~\ref{lcpnsa} and~\ref{lcpnsb}, we calculate group elements $c,d$ such that $a\cpns c$ and $b\cpns d$, and we check if $c$ and~$d$ have the same cycle type.

We now turn to $\PBan$ and $\Ban$. Let $a\in \PBan$. In this case, the partition $u$ constructed in Lemma \ref{lcpnsa} is an element of $\PBan$ as well, and therefore Lemma \ref{lcpnsa} and its dual also hold in $\PBan$. We can now repeat the proof of Proposition \ref{pacpnsc}, noting that the situations in which Lemma \ref{lcpnsb} or its dual are used cannot arise in $\PBan$: if $s$ is transversal, then $A = s\cap \bn$ is a singleton, so $A'$ cannot intersect different blocks $t_1, t_2$. As in Theorem \ref{tcpns}, we obtain:

\begin{theorem}\label{tcpnsPB}
	In $\PBan$, $\cpns\,\,=\,\,\ctr$. That is, for $a,b\in \PBan$, $a\cpns b$ if and only if $a^{\ome+1}$ and $b^{\ome+1}$ have the same cycle type.
\end{theorem}

\begin{lemma}\label{lcpnsB}
	Suppose that $a\in \Ban$ and $\{x,y\}\subseteq \bn$ is a block of $a$ such that $x',y'$ lie in (necessarily distinct) transversal blocks. Then $a\cpn c$ for some $c\in \Ban$ with lower rank than $a$.
\end{lemma}
\begin{proof}
	Let $\{v,x'\}, \{w,y'\}$ be the blocks containing $x',y'$, and $k$ the number of upper blocks of $a$. As $a$ is a partition in $\Ban$, $k$ is also the number of lower blocks. Consider $u\in \Ban$ with the following blocks: $s$ and $s'$ for each upper block $s$ of $a$, and $\{z,z'\}$ for each $z \in \bn$ that does not intersect an upper block of $a$.
	
	It is straightforward to check that $ua = a$. Let $c = au$, so $c\cpn a$. The $k$ upper blocks of $a$ are also upper blocks of $c$. In addition, $\{v,w\}$ is an upper block of $c$. So $c$ has more than $k$ upper blocks, and hence also more than $k$ lower blocks. It follows that it has fewer transversal blocks than $a$, as required.
\end{proof}

Clearly, the dual version of Lemma \ref{lcpnsB} holds as well.

\begin{prop} \label{pacpnscB}
	Let $a\in\Ban$. Then there exists a group element $c\in \Ban$ such that $a\cpns c$.
\end{prop}
\begin{proof}
	Recall that $\cfn\,\, \subseteq\,\, \cpns$. Let $a\in \Ban$. Then $a \cfn b$ (and hence $a\cpns b$) for some $b$ in $\frn$-normal form. Suppose that there is a $b$-block $\{x,y\}$ as in Lemma~\ref{lcpnsB}. We can then use Lemma~\ref{lcpnsB} to obtain
	an element $c$ of lower rank than $b$ such that $b\cpns c$. If instead there is a $b$-block $\{x',y'\}$ such that $x,y$ lie in transversal $b$-blocks, than we can find such $c$ using the dual version of Lemma~\ref{lcpnsB}. We next obtain a partition $a_1\in \Ban$ in $\frn$-normal form satisfying $c\cfn a_1$. Note that $c$ and $a_1$ have the same rank as
	$\cfn\,\, \subseteq\,\, \mathcal{D}$ (by Proposition~\ref{Prp:D}).
	
	We have constructed an element $a_1\in\Ban$ in $\frn$-normal form such that $a\cpns a_1$ and $a_1$ has a lower rank than $a$.
	We keep repeating this construction until we obtain a partition $d\in \Ban$ such that $a\cpns d$, $d$ is in $\frn$-normal form, and neither Lemma \ref{lcpnsB} nor its dual can be applied to $d$.
	(Note that $d$ may be $b$ if neither Lemma \ref{lcpnsB} nor its dual can be applied to $b$.) By Definition \ref{d:Bnorm}, $\{x,y\}$ is an upper block of $d$ if and only if $\{x',y'\}$ is a lower block of $d$. Hence $d$ is a group element.
\end{proof}

As in Theorem \ref{tcpns}, we obtain:

\begin{theorem} \label{tcpnsB}
	In $\Ban$, $\cpns\,\,=\,\,\ctr$. That is, for $a,b \in \Ban$, $a\cpns b$ if and only if $a^{\ome+1}$ and $b^{\ome+1}$ have the same cycle type.
\end{theorem}

\subsection{Conjugacies $\coon$ and $\con$ in $\Pan$, $\PBan$, and $\Ban$}
\label{subcoon}

The conjugacy $\coon$ \eqref{econo} is the largest of the conjugacies considered in this paper. In any semigroup, $\cfn\,\,\subseteq\,\,\cpn\,\,\subseteq\,\,\cpns\,\,\subseteq\,\,\coon$ and $\cfn\,\,\subseteq\,\,\con\,\,\subseteq\,\,\coon$ \cite[Prop.~2.3]{Ko18}. In any epigroup, $\cfn\,\,\subseteq\,\,\cpn\,\,\subseteq\,\,\cpns\,\,\subseteq\,\,\ctr\,\,\subseteq\,\,\coon$ \cite[Thm~4.8]{ArKiKoMaTA}. Moreover, for any semigroup $S$, $\coon$ is the universal relation if $S$ has a zero, and $\coon\,\,=\,\,\con$ if $S$ has no zero.

It is known that $\coon$ is the identity relation on a semigroup $S$ if and only if $S$ is commutative and cancellative \cite[Thm.~5.6]{ArKiKoMaTA}. There is no characterization of the semigroups (with no zero) in which $\coon$ is the universal relation. In the finite partition monoids, which have no zero for $n \ge 2$, $\coon$
is the universal relation.

\begin{theorem}\label{tcoonu}
	In $\Pan$, $\coon\,\,=\,\,\Pan\times\Pan$.
\end{theorem}
\begin{proof}
	Let $e = \{\{x,x'\} : x\in[n]\}$ be the identity in $\Pan$ and let $a\in \Pan$ be arbitrary. We want to find $g\in \Pan$ such that $ag = ge$. Consider $g\in \Pan$ such that $\ker(g) = \ker(a^{\ome})$, $\coker(g) = \{\{x'\} : x'\in[n']\}$, and $g$ does not have any transversal blocks. Then $\ker(ag) = \ker(aa^{\ome}) = \ker(a^{\ome+1}) = \ker(a^{\ome}) = \ker(g)$, where the last but one equality follows from the fact that $a^{\ome+1}\greenH a^{\ome}$. Since $\coker(g)$ is trivial and $g$ has no transversal blocks, $\coker(ag)$ is also trivial and $ag$ has no transversal blocks either. Thus $ag = g = ge$. Similarly, for $h\in \Pan$ such that $\coker(h) = \coker(a^{\ome})$, $\ker(h) = \{\{x\} : x\in[n]\}$,
	and $h$ does not have any transversal blocks, we have $ha = h = eh$.
	We have proved that for every $a\in\Pan$, $a\coon e$. Hence $\coon\,\,=\,\,\Pan\times \Pan$ since $\coon$ is an equivalence relation.
\end{proof}

In the case that $a\in \PBan$, the elements $g$ and $h$ constructed as above are in $\PBan$ as well. Hence we immediately obtain the following classification.

\begin{theorem}\label{tcoonu2}
	In $\PBan$, $\coon\,\,=\,\,\PBan\times\PBan$.
\end{theorem}

We now consider $\coon$ for a Brauer monoid $\Ban$. As $\ctr\,\, \subseteq\,\, \coon$, it follows from Theorem \ref{t:trPBB} that there is a partition $Q$ of the set of available cycle types, such that $a\coon b$ if and only if the cycle types of $a^{\ome+1}$ and $b^{\ome+1}$ lie in the same part of $Q$. Moreover, as $\cfn\,\, \subseteq\,\, \coon$, Theorem \ref{tcfnB} shows that $a$ has a $\coon$-conjugate $c$ in $\frn$-normal form (see Definition \ref{d:Bnorm}). We will show below that this element can be chosen as a group element.

The following lemma provides a description of such partitions, which follows directly from Theorem \ref{tcfnB} and Definition \ref{d:Bnorm}.

\begin{lemma}\label{l:n+gr}
	Suppose that $c\in \Ban$ is both a group element and in $\frn$-normal form. Then there is a partition $\bn = A\cup B$ such that $A\cup A'$ contains all transversal $b$-blocks and $B\cup B'$ contains all nontransversal $b$-blocks (where we allow $A = \emptyset$ or $B = \emptyset$).
	
	Moreover, there is a partition of $B$ into subsets $B_i$ of size $2$, such that $B_i$ and $B_i'$ are blocks for all $i$.
\end{lemma}

We remark that $|B|$ is even, and that we may identify $c_A$ with a permutation in $\sym_A$.

\begin{lemma}\label{l:gr-n-norm}
	Let $a\in \Ban$ be a group element. Then there is a partition $b\in \Ban$ in $\frn$-normal form such that $b$ is a group element with the same cycle type as $a$.
\end{lemma}
\begin{proof}
	Let $k$ be the number of blocks of $\ker(a)\lor \coker^{\wedge}(a)$ that are used in the construction of the permutation corresponding to $a$ (that is, the blocks of $\ker(a)\lor \coker^{\wedge}(a)$ that intersect a transversal block of $a$). Pick a $k$-subset $A$ of $\bn$. Using only transversal blocks, we can construct a partition $b_A$ on $A\cup A'$ that has the same cycle type as $a$ (and which we might consider to be an element of $\sym_A$).
	
	In $\Ban$, a block of $\ker(a)\lor \coker^{\wedge}(a)$ that intersects one transversal of $a$ has odd cardinality, while a block of $\ker(a)\lor \coker^{\wedge}(a)$ that does not has even cardinality. It follows that $|\bn\setminus A|$ is even.
	
	Partitioning $B =\bn\setminus A$ into $2$-element sets $B_i$, we can extend $b_a$ to a partition $b\in \Ban$ by adding the blocks $B_i, B_i'$ for each $i$. The result follows.
\end{proof}

If the permutation associated with $b_A$ contains a cycle of size $l$, it is clear that we may identify a subset $C$ of $A$ such that $b_{C}$ represents this cycle. When we speak of such a representation, we will always assume that $|C|=l$ (so unlike in the standard use of ``cycle'', we do not allow any additional $1$-cycles to be represented in $C$).

\begin{lemma}\label{l:evencyc}
	Let $a\in \Ban$ be a group element in $\frn$-normal form, and suppose that $C\subseteq \bn$ is such that $a_C$ represents a cycle of even length $l$. Then there is a partition of $C$ into $2$-subsets $C_i$ and $b\in \Ban$ such $a\coon b$, $b$ contains the blocks $C_i,C_i'$ for all $i$, and $a_D = b_D$ for $D =\bn\setminus C$.
\end{lemma}
\begin{proof}
	Order the elements of $C$ as $c_1,\dots, c_l$, such that the $a$-blocks intersecting $C$ are $\{c_l,c_1'\}$ and $\{c_i,c_{i+1}'\}$ for $i=1, \dots, l-1$.
	
	Partition $C$ into blocks $C_i = \{c_i, c_{i+l/2}\}$ for $i=1, \dots l/2$, define $g\in \Ban$ with blocks $C_i, C_i'$ and $\{z,z'\}$ for $z\notin C$, and set $g = h$. It is straightforward to check that $g,h$ witness $a\coon b$.
\end{proof}

\begin{lemma}\label{l:2oddcyc}
	Let $a\in \Ban$ be a group element in $\frn$-normal form, and suppose that $C,D\subseteq \bn$, $C\neq D$ are such that $a_C, a_D$ represents cycles of the same length $l$. Then there is a partition of $C\cup D$ into $2$-subsets $G_i$ and $b\in \Ban$ such $a\coon b$, $b$ contains the blocks $G_i,G_i'$ for all $i$, and $a_L = b_L$ for $L =\bn\setminus (C\cup D)$.
\end{lemma}
\begin{proof}
	Suppose that $C = \{c_1,c_2,\ldots, c_l\}$, $D = \{d_1,d_2,\ldots, d_l\}$ are ordered such that $\{c_l,c_1'\}$, $\{d_l,d_1'\}$, $\{c_i,c_{i+1}'\}$ and $\{d_i,d_{i+1}'\}$, $i=1,\ldots, l-1$,  are the $a$-blocks intersecting $C \cup D$.
	
	Partition $C\cup D$ into blocks $G_i=\{c_i, d_i\}$ for $i=1,\ldots l$, define $g\in \Ban$ to have blocks $G_i, G_i'$ and $\{z,z'\}$ for $z\notin C\cup D$, and set $g = h$. It is straightforward to check that $g,h$ witness $a\coon b$.
\end{proof}

\begin{theorem}\label{t:oBn}
	Let $a,b\in \Ban$, such that $a^{\ome+1}$ and $b^{\ome+1}$ have cycle types $(k_1,\ldots, k_n)$ and $(l_1,\ldots, l_n)$, respectively. Then $a\coon b$ if and only if $k_i\equiv l_i\pmod 2$ for each odd $i$.
\end{theorem}
\begin{proof}
	Suppose that $k_i\equiv l_i\pmod 2$ for each odd $i$. Because $\ctr\,\, \subseteq\,\, \coon$, and by Lemma \ref{l:gr-n-norm}, there exist partitions $a'\coon a, b'\coon b$, such that $a',b'$ are both group elements in $\cfn$-normal form with the same cycle type as $a,b$. 
	
	By repeated applications of the constructions from Lemmas \ref{l:evencyc} and \ref{l:2oddcyc}, we obtain partitions $a'' \coon a', b'' \coon b'$, such that $a',b'$ are both group elements in $\cfn$-normal form, and such the permutations corresponding to $a'',b''$ contain no even cycles and at most one $j$-cycle for each odd $j$. Moreover, $a''$ [$b''$] contains an odd $j$-cycle exactly if $k_j$ [$l_j$] is odd. As we assumed that $k_i\equiv l_i\pmod 2$ for each odd $i$, we see that $a''$ and $b''$ have the same cycle type. It follows that $a''\ctr b''$, thus $a''\coon b''$, and hence $a\coon b$, as required.
	
	Conversely, suppose that $a\coon b$. Arguing as in the first part of the proof, we may assume that $a,b$ are group elements in $\frn$-normal form that do not contain any even cycles and at most one $j$-cycle for each odd $j$. By Lemma \ref{l:n+gr}, we see that in $a$ and $b$, $\{x,y\}$ is a block if and only if $\{x',y'\}$ is a block.
	
	Let $a',b'$ be constructed from $a$ and $b$ by replacing all pairs of blocks of the form $\{x,y\},\{x',y'\}$ with blocks $\{x,x'\}, \{y,y'\}$. Note that $a'$ and $b'$ are units of $\Ban$. As our construction introduces an even number of $1$-cycles, it follows that $a'\coon a\coon b\coon b'$ by the first part of this proof. Moreover, the cycle types of $a$ and $a'$ [$b$ and $b'$] agree everywhere except in the first entry, where they agree up to parity. 
	
	Let $g,h\in \Ban$ witness $a'\coon b'$, and let $X_1 = \dom(g)$,  $X_2 = \codom^{\wedge}(g)$. As $g\in \Ban$, $|X_1| = |X_2|$. Therefore there exists $\omega\in \sym_n$ such that $X_2\omega = X_1$. Let $u = \lambda_\omega$, the unit corresponding to $\omega$. Setting $a'' = ua'u^{-1}, g' = ug, h' = hu^{-1}$, it is straightforward to check that $g',h'$ witness  $a''\coon b'$. Moreover, $a''$ has the same cycle type as $a'$, and $\dom(g) = \codom^{\wedge}(g)$.
	
	By abuse of notation, rename $a'',b',g',h'$ as $a,b,g,h$. It suffices to show that $a$ and $b$ have the same cycle type. Let $\alpha\in \sym_n$ be such that $a = \lambda_\alpha$. 
	
	Now let $\{x,y\}$ be a nontransversal upper block of $g$. Then $\{x\lambda_\alpha^{-1}, y\lambda_\alpha^{-1}\}$ is an upper block of $ag$. As $ag = gb$, this is only possible if it also an upper block of $g$, where we note that $b$ is a unit. Therefore we have that $\{x\lambda_\alpha^{-1}, y\lambda_\alpha^{-1}\}\subseteq \bn \setminus \dom(g)$. Repeating the argument, we see that $\{x\lambda_\alpha^{-k}, y\lambda_\alpha^{-k}\}$ is an upper block of $g$ for all $k\in \mathbb{N}$. It follows that the orbits $O_x$ of $x$ and $O_y$ of $y$ under $\lambda_\alpha$ have the same size. If $O_x = O_y$, then the blocks $\{x\lambda_\alpha^{-k}, y\lambda_\alpha^{-k}\}$ would cover this orbit. However, this is impossible, as $\lambda_\alpha$ has only odd cycles. Hence $O_x$ and $O_y$ are different orbits of the same size, which implies that $|O_x| = |O_y| = 1$. It follows that $\{x,x'\}$ is an $a$-block, whenever $x\not\in \dom(g)$. A dual argument shows that $\{x,x'\}$ is a $b$-block whenever $x\not\in \codom^{\wedge}(g) = \dom(g)$.
	
	By our construction, no block of $a,b,$ or $g$ intersects both $\dom(g)\,\cup \,\dom(g)'$ and its complement. Hence the restrictions $a_{\dom(g)},b_{\dom(g)},g_{\dom(g)}$ are in $\mathcal{B}_{\dom(g)}$, and we obtain 
	\[
	a_{\dom(g)}g_{\dom(g)}=g_{\dom(g)}b_{\dom(g)}.
	\] 
	Because $g_{\dom(g)}$ has only transversal blocks, it is a unit of $\mathcal{B}_{\dom(g)}$, and we obtain
	\[
	g_{\dom(g)}^{-1}a_{\dom(g)}g_{\dom(g)}=b_{\dom(g)}.
	\]
	Replacing these units with their corresponding permutations from $S_{\dom(g)}$, we see that $a_{\dom(g)}$ and $b_{\dom(g)}$ have the same cycle type. Because $a$ and $b$ agree outside of  $\dom(g)\cup \dom(g)'$, they also have identical cycle types, as required. The result follows.
\end{proof}

Since $\con\,\,=\,\,\coon$ in any semigroup that does not have a zero, we obtain the following result. The listed exceptional cases contain a zero and can be confirmed by direct calculation. (See \eqref{econc} for the definition of~$\con$.)

\begin{theorem}
	In $\Pan, \PBan$, and $\Ban$, $\coon\,\,=\,\,\con$, except for $\mathcal{P}_1$, $\mathcal{PB}_1$, $\mathcal{B}_2$, where $\con$ is equality. That is, in $\Pan$ and $\PBan$, $\con$ is the universal relation, except for $\mathcal{P}_1$, $\mathcal{PB}_1$, where $\con$ is equality.
	
	If $a,b\in \Ban$, $n\neq 2$, such that $a^{\ome+1}$ and $b^{\ome+1}$ have cycle types $(k_1,\ldots, k_n)$ and $(l_1,\ldots , l_n)$, then $a\con b$ if and only if $k_i\equiv l_i\pmod 2$ for each odd $i$. On $\mathcal{B}_2$, $\con$ is equality.
\end{theorem}

\section{Conjugacy growth in polycyclic monoids}
\label{Sec:polycyclic}

The study of conjugacy in polycyclic monoids was initiated in \cite{AKKM18} by some of the authors of this article. Polycyclic monoids are inverse monoids with zero so $\coon$ is the universal relation and $\cin\ =\ \cfn$. In \cite{AKKM18} the notions of $\cpn$ \eqref{ecp}, and $\con$ \eqref{econc} were characterized. In this section we study $\cfn$ \eqref{e1dcon2}, especially in the following context.

The \emph{conjugacy growth function} of a finitely generated group $G$ counts the number of conjugacy classes intersecting the ball of radius $n$ in the Cayley graph of $G$ centered at the identity, for all $n\geq 0$. It has been studied for free groups \cite{R04,C05,R10}, hyperbolic groups \cite{CK02,CK04}, solvable groups \cite{BdC10}, linear groups in \cite{BCLM13}, acylindrically hyperbolic groups \cite{HO13,AC16h}, certain branch groups \cite{F14}, in the higher Heisenberg groups in \cite{Evetts23}, and several other classes of groups \cite{GS_10}.

Given a notion of conjugacy for monoids that is an equivalence relation, the conjugacy growth function for groups can be extended to finitely presented monoids. In this section we will present the conjugacy growth functions of the polycyclic monoids for the conjugacies $\cfn$, $\con$, and $\cpns$.

In recent years, the \emph{conjugacy growth series} (the generating series associated with the conjugacy growth functions) have been computed for several classes of groups based on the description of sets of minimal length representatives from all conjugacy classes \cite{CHHR14,CH14,AC16h,CEH20,CHM23}. The paper \cite{E19} supports the conjecture that virtually abelian groups are the only ones with rational conjugacy series. Historically, one of the initial motivations for counting conjugacy classes of a given length came from counting closed geodesics of bounded length in compact Riemannian manifolds \cite{M69}.

We first need some preliminaries.

\subsection{Characterization of the conjugacy relations in $\Poln$}

Let $n\geq 2$. Consider a set $A_n=\{p_1,\ldots, p_n\}$, and denote by
$A_n^{-1}$ a disjoint copy $\{p_1^{-1},\ldots, p_n^{-1}\}$. Let $\Sigma=A_n\cup A_n^{-1}$. The \emph{polycyclic monoid} $\Poln$ is the monoid with zero defined by the monoid presentation $\Poln=\lan \Sigma_0 \mid p^{-1}_ip_i=1\text{ and }p^{-1}_ip_j=0, i\ne j\}\ran$,
where $\Sigma_0=\Sigma\cup\{0\}$ and $0$ is a new symbol not in $\Sigma$, defined to be the zero of the monoid.

Given $x\in \Sigma$, we define $x^{-1}$ to be $p_i^{-1}$ if $x=p_i\in A_n$, and to be $p_i$ if $x=p_i^{-1}\in A_n^{-1}$.
We define $1^{-1}=1$ and $(xw)^{-1}=w^{-1}x^{-1}$, for all $x\in A_n$ and $w\in A_n^*$. It is well known (e.g., \cite[\S~9.3]{Lawson2009}) that every nonzero
element of $\Poln$ has a unique representation of the form $yx^{-1}$ with $y,x\in A_n^*$. Whenever we write $a=yx^{-1}$, it will be understood that $x,y\in A_n^*$. We identify nonzero elements of $\Poln$ with words of this form. The explicit multiplication is provided by the following lemma. We say that words $x,v\in A_n^*$ are \emph{prefix comparable} if one is a prefix of the other.

\begin{lemma}[\!\!\textbf{\cite[Lem.~3.2]{AKKM18}}]
\label{lbas}
Let $yx^{-1}$ and $vu^{-1}$ be nonzero elements of $\Poln$. Then:
\begin{itemize}
	\item[\textup{(1)}] $yx^{-1}\cdot vu^{-1}\neq 0$ if and only if $x$ and $v$ are
		prefix comparable;
	\item[\textup{(2)}] if $yx^{-1}\cdot vu^{-1}\neq 0$, then
		\[
		y x^{-1} \cdot v u^{-1} =
		\begin{cases}
			\ yzu^{-1} & \text{if } v=xz\,, \\
			\ y(uz)^{-1} & \text{if } x=vz\,.
		\end{cases}
		\]
	\item[\rm(3)] $y=v$ in $\Poln$ if and only if $y=v$ in $A_n^*$, and $x^{-1}=u^{-1}$ in $\Poln$ if and only if $x=u$ in $A_n^*$.
	\end{itemize}
\end{lemma}

A word $w\in \Poln$ is said to be \emph{cyclically reduced} if $w=0$ or $w=yx^{-1}$, where $x$ and $y$ have no common prefix other than $1$.
Every nonzero element of $\Poln$ can be written in the form
$ryx^{-1}r^{-1}$,  with $r\in A_n^*$ and $yx^{-1}$ a cyclically reduced word. From any $a\in \Poln$, we compute a cyclically reduced word $\widetilde{a}$ as follows: if $a=0$, we let $\widetilde{a}$ equal $0$; otherwise, $a=ryx^{-1}r^{-1}$ as above, so we let $\widetilde{a}$ be the (possibly empty) cyclically reduced word $yx^{-1}$.

We now characterize conjugacy $\cfn$ in $\Poln$. Since $\Poln$ is an inverse monoid, we have $\cfn=\cin$ by Proposition \ref{Prp:inverse},
that is, for all $a,b\in \Poln$, $a\cfn b$ if and only if there exists $g\in \Poln$ such that $g^{-1}ag=b$ and $gbg^{-1}=a$.

\begin{theorem}\label{Thm:cfn_Polycyclic}
	Let $a,b\in \Poln$. Then $a\cfn b$ if and only if $a=b=0$ or $\widetilde{a}=\widetilde{b}$.
\end{theorem}
\begin{proof}
	Since $[0]_{\frn}=\{0\}$, it remains to establish criteria for nonzero $a,b\in \Poln$ to be $\frn$-conjugate. In the calculations below, it will be convenient to write $a=yx^{-1}$ as $a=a_+a_-^{-1}$.
	
	Let $a=a_+a_-^{-1},b=b_+b_-^{-1}\in \Poln$ with $a\cfn b$. Then there exists $g=g_+g_-^{-1}\in \Poln$ such that
	\begin{equation}\label{eqiii}
		g_-g_+^{-1}a_+a_-^{-1}g_+g_-^{-1}=b_+b_-^{-1}\quad \textup{and}\quad g_+g_-^{-1}b_+b_-^{-1}g_-g_+^{-1}=a_+a_-^{-1}.
	\end{equation}
	Since $b_+b_-^{-1}\ne 0$, it follows by Lemma~\ref{lbas} that $a_-$ and $g_+$ are prefix-comparable, $g_+$ and $a_+$ are also prefix comparable, and
	\[
	g_-g_+^{-1}a_+a_-^{-1}g_+g_-^{-1}=
	\begin{cases}
	g_-g_+^{-1}a_+rg_-^{-1} & \textup{if  }g_+=a_-r,\,=
	\begin{cases}
			g_-sg_-^{-1} & \textup{if  }a_+r=g_+s\\
			g_-(g_-s)^{-1} & \textup{if  }g_+=a_+r
	\end{cases}\\
	g_-g_+^{-1}a_+(g_-r)^{-1} & \textup{if  }a_-=g_+r,\,=
	\begin{cases}
			g_-(g_-rs)^{-1} & \textup{if  }g_+=a_+s\\
			g_-s(g_-r)^{-1} & \textup{if  }a_+=g_+s,
	\end{cases}
	\end{cases}
	\]
	where $r,s\in A_n^*$. By these calculations, the first equality in \eqref{eqiii}, and Lemma~\ref{lbas}(4), we obtain:
	\begin{align*}
		g_-s=b_+\textup{ and } g_-=b_-\quad \textup{ if }&\quad a_+r=g_+s\textup{ and }g_+=a_-r,\\
		g_-=b_+ \textup{ and }g_-s=b_- \quad\textup{ if }&\quad g_+=a_+rs\textup{ and }g_+=a_-r,\\
		g_-=b_+\textup{ and }g_-rs=b_- \quad\textup{ if }&\quad g_+=a_+s\textup{ and }a_-=g_+r,\\
		g_-s=b_+\textup{ and }g_-r=b_- \quad\textup{ if }&\quad a_+=g_+s\textup{ and }a_-=g_+r.
	\end{align*}
	Thus we have four cases to consider, and in each case we can draw conclusions using the second equality in \eqref{eqiii}
	and Lemma~\ref{lbas}(4).
	\vskip 1mm
	\noindent\textbf{Case 1.}  $g_-s=b_+$, $g_-=b_-$, $a_+r=g_+s$, $g_+=a_-r$.
	\vskip 1mm
	Then $a_+a_-^{-1}=g_+g_-^{-1}b_+b_-^{-1}g_-g_+^{-1}=g_+sg_+^{-1}$, so $r=1$, and hence $a=g_+sg_+^{-1}$ and $b=g_-sg_-^{-1}$.
	\vskip 1mm
	\noindent\textbf{Case 2.} $g_-=b_+$, $g_-s=b_-$, $g_+=a_+rs$, $g_+=a_-r$.
	\vskip 1mm
	Then $a_+a_-^{-1}=g_+g_-^{-1}b_+b_-^{-1}g_-g_+^{-1}=g_+(g_+s)^{-1}$, so $s=r=1$, and hence $a=g_+g_+^{-1}$ and $b=g_-g_-^{-1}$.
	\vskip 1mm
	\noindent\textbf{Case 3.} $g_-=b_+$, $g_-rs=b_-$, $g_+=a_+s$, $a_-=g_+r$.
	\vskip 1mm
	Then $a_+a_-^{-1}=g_+g_-^{-1}b_+b_-^{-1}g_-g_+^{-1}=g_+(g_+rs)^{-1}$, so $s=1$, and hence $a=g_+(g_+r)^{-1}$ and $b=g_-(g_-r)^{-1}$.
	\vskip 1mm
	\noindent\textbf{Case 4.} $g_-s=b_+$, $g_-r=b_-$, $a_+=g_+s$, $a_-=g_+r$.
	\vskip 1mm
	Then $a_+a_-^{-1}=g_+g_-^{-1}b_+b_-^{-1}g_-g_+^{-1}=g_+s(g_+r)^{-1}$, and hence $a=g_+s(g_+r)^{-1}$ and $b=g_-s(g_-r)^{-1}$.
	\vskip 1mm
	Note that the forms of $a$ and $b$ deduced in Cases 1--3 are special cases of the forms deduced in Case~4. Therefore, if $a\cfn b$, then $a=g_+s(g_+r)^{-1}$ and $b=g_-s(g_-r)^{-1}$ for some $g_+,g_-,r,s\in A_n^*$. Conversely,	if $a=g_+s(g_+r)^{-1}$ and $b=g_-s(g_-r)^{-1}$ for some $g_+,g_-,r,s\in A_n^*$, then it is straightforward
	to verify $g^{-1}ag=b$ and $gbg^{-1}=a$ for $g=g_+g_-$. We have proved the result.
\end{proof}

Note that for any representative $a\in \Poln$ we have $a\cfn \widetilde{a}$. This gives the following corollary.

\begin{cor}\label{Cor:cgl_cfn_Polycyclic}
	The set of cyclically reduced words is a set of representatives of minimal length of the partition $\Poln/\!\!\cfn$.
\end{cor}

For a nonzero representative $a=yx^{-1}\in \Poln$, we denote by $\tla$ the representative word of $x^{-1}y$ in $\Poln$. We also set $\rho (0)=0$. Note that $\tla\in A_n^*\cup (A_n^{-1})^*\cup\{0\}$, for any representative $a\in \Poln$. Also note that $\tla=\widetilde{a}$ if and only if $\widetilde{a}\in  A_n^*\cup (A_n^{-1})^*\cup\{0\}$.

Let us recall the characterizations of $\con$ and $\cpn$ from  \cite{AKKM18}.

\begin{lemma}[\!\!\textbf{\cite[Thm.~3.9]{AKKM18}}]
	\label{p56}
	Let $a,b\in \Poln$. Then $a\con b$ if and only if one of the following holds:
	\begin{itemize}
		\item[\textup{(a)}] $a=b=0$;
		\item[\textup{(b)}] $\widetilde{a}=\widetilde{b}$; or
		\item[\textup{(c)}] $\widetilde{a}, \widetilde{b}\in (A_n^{-1})^*$ and $\widetilde{a}\cpn\widetilde{b}$ in the free monoid $(A_n^{-1})^*$.
	\end{itemize}
\end{lemma}

In particular, if an element of $\Poln$ is not in $(A_n^{-1})^*\cup\{0\}$ then it is $\con$-conjugate to a unique element $yx^{-1}$ such that $y\neq 1$ and $x$ and $y$ have no common prefix other than $1$.

For a given alphabet $X$, let $L_p(X)$ denote a set of representatives of minimal length of the partition resulting of the quotient of the free monoid $X^*$ on $X$ by the equivalence relation $\cpn$ on $X^*$.

\begin{cor}\label{Cor:cgl_con_Polycyclic}
	The set of cyclically reduced words with a prefix in $A_n\cup\{0\}$ together with the set $L_p(A_n^{-1})$ is a set of representatives of minimal length of the partition $\Poln/\!\!\con$.
\end{cor}

Any distinct $a,b\in \Poln$ such that $a,b\in A_n^*$ or $a,b\in (A_n^{-1})^*$ are never $\frn$-conjugate. This shows that in $\Poln$, conjugacy $\cfn$ is \emph{strictly} included in $\con$ and $\cpn$ (see \cite[Corollary~3.10]{AKKM18}).

\begin{lemma}[\!\!\textbf{\cite[Thm.~3.6]{AKKM18}}]
	\label{ccp}
	Let $a,b\in \Poln$. Then $a\cpn b$ if and only if one of the following holds:
	\begin{itemize}
		\item[\textup{(a)}] $a=\tlb=0$ or $\tla=b=0$;
		\item[\textup{(b)}] $\tla=\tlb=0$ and $\widetilde{a}=\widetilde{b}$;
		\item[\textup{(c)}] $\widetilde{a}, \widetilde{b}\in A_n^*$ and $\widetilde{a}\cpn\widetilde{b}$ in the free monoid $A_n^*$; or
		\item[\rm(d)] $\widetilde{a}, \widetilde{b}\in (A_n^{-1})^*$ and $\widetilde{a}\cpn\widetilde{b}$ in the free monoid $(A_n^{-1})^*$.
	\end{itemize}
\end{lemma}

From Lemma~\ref{ccp} and other results in \cite{AKKM18}, we can deduce a characterization of $\cpns$ in $\Poln$.

\begin{prop}
	\label{pcps}
	Let $a,b\in \Poln$. Then $a\cpns b$ if and only if one of the following conditions holds:
	\begin{itemize}
		\item[\textup{(a)}] $\tla=\tlb=0$;
		\item[\textup{(b)}] $\widetilde{a}, \widetilde{b}\in A_n^*$ and $\widetilde{a}\cpn\widetilde{b}$ in the free monoid $A_n^*$; or
		\item[\textup{(c)}] $\widetilde{a}, \widetilde{b}\in (A_n^{-1})^*$ and $\widetilde{a}\cpn\widetilde{b}$ in the free monoid $(A_n^{-1})^*$.
	\end{itemize}
\end{prop}
\begin{proof}
	Suppose $a\cpns b$. Then, by \cite[Thm.~3.7]{AKKM18}, either $a\cpn b$ or $a\cpn 0\cpn b$. In the former case, (a), (b), or (c) is satisfied by Lemma~\ref{ccp}. Suppose $a\cpn 0\cpn b$. Then $\rho(a)=\rho(b)=0$ by \cite[Lem.~3.4]{AKKM18}, and so (a) is satisfied.
	
	Conversely, suppose that one of (a), (b), (c) holds. If (b) or (c) holds, then $a\cpn b$ by Lemma~\ref{ccp}, and so $a\cpns b$. Suppose (a) is satisfied. Then, by \cite[Lem.~3.4]{AKKM18} again, $a\cpn 0\cpn b$, and so $a\cpns b$.
\end{proof}

In particular, if a representative element of $\Poln$ is not in $A_n^*\cup (A_n^{-1})^*$, then it is $\cpns$-conjugate to $0$.

\begin{cor}\label{Cor:cgl_cpns_Polycyclic}
	The set $L_p(A_n)\cup L_p(A_n^{-1})\cup \{0,1\}$ is a set of representatives of minimal length of the partition $\Poln/\!\!\cpns$.
\end{cor}

\subsection{Conjugacy growth functions in $\Poln$}

Let $M$ be a monoid generated by a finite set $X$. Then every element of $M$ can be represented as a word in $X^*$. The length of an element $a\in M$ is the minimum length of a word in $X^*$ that represents $a$, written $|a|_X$ or just $|a|$ if the context is clear.

Given an equivalence relation $\sim$ in $M$, we extend the notion  of length to equivalence classes in the quotient set $M/\sim$ by setting
\[
|[a]_{\sim}|_X=\min\{|b|_X:b\in[a]_{\sim}\}.
\]
Note that if $\sim$ is the identity relation, then $|[a]_{\sim}|_X=|a|_X$.

Since $X$ is finite, for each integer $m\geq 0$, there are only finitely many elements of $M$ of length $m$. This leads to the following definition.

\begin{defi}
	Let $M$ be a monoid with finite generating set $X$, and let $\sim$ be an equivalence relation on $M$. The \emph{strict growth function} of $M$ \emph{relative to} $\sim$ (with respect to $X$) is defined by
	\[
	\sgf_{M/\sim,X}(n)=\#\{a\in M:  |[a]_{\sim}|_X=n\}
	\]
	for each $n\in \mathbb{N}_0$.
\end{defi}
When $\sim$ is the identity relation, the function $\sgf_{M/\sim,X}$ is denoted simply by $\sgf_{M,X}$ and is known as the \emph{strict growth function} of $M$ with respect to $X$.

Regarding the characterization of representatives of the polycyclic monoid given in the previous subsection, we obtain the following result:
\begin{prop}
	The strict growth function of $\Poln$ is given by
	\[
	\sigma_{\Poln,\Sigma_0}(0)=1,\, \sigma_{\Poln,\Sigma_0}(1)=2n+1,\text{ and }\sigma_{\Poln,\Sigma_0}(m)=(m+1)n^m\text{ for }m\geq 2\,.
	\]
\end{prop}

In the cases where the equivalence relation $\sim$ corresponds to one of the conjugacies $\cfn$, $\con$, or $\cpns$, we refer to the strict growth function of $M$ relative to $\sim$ as the \emph{strict $\sim$-conjugacy growth function} of the monoid $M$, or simply the strict \emph{conjugacy} growth function of $M$ relative to $\sim$.

We now compute the conjugacy growth functions the polycyclic monoids   for the conjugacies  $\cfn$, $\con$, and $\cpns$.

\begin{theorem}\label{Thm:cgf_cfn_Polycyclic}
	The strict $\cfn$-conjugacy growth function of $\Poln$ is given by $\sgf_{\Poln/\cfn,\Sigma_0}(0)=1$, $\sgf_{\Poln/\cfn,\Sigma_0}(1)=2n+1$, and $\sgf_{\Poln/\cfn,\Sigma_0}(m)=2n^m+(m-1)n^{m-1}(n-1)$, for $m\geq 2$.
\end{theorem}
\begin{proof}
	We use Corollary~\ref{Cor:cgl_cfn_Polycyclic} to deduce the result. The cases for $m=0$ and $m=1$ are easy. For $m\geq 2$, we can distinguish the case when the cyclically reduced word is in $A_n^*\cup (A_n^{-1})^*$, for which we get $2n^m$ cyclically reduced words of length $m$, from the cases where the cyclically reduced word of length $m$ has the form  $yx^{-1}$, with $x$ and $y$ nonempty and with no common prefix.
\end{proof}

To be able to compute the conjugacy growth functions of $\con$ and $\cpns$, we need to compute the $\cpns$-conjugacy growth function of the free monoid on a given alphabet $X$.

\begin{theorem}\label{Thm:cgf_cpn_free_monoid}
	Let $X$ be an alphabet with $|X|=n$. The $\cpns$-conjugacy growth function of the free monoid $X^*$ on $X$ is
	\[
	\sgf_{X^*/\cpns,X}(m)=\sum_{d|m}\sum_{e|d}\mu\left(\frac{d}{e}\right)	\frac{n^e}{d},\qquad m\geq 1,
	\]
	where $\mu$ is the M\"{o}bius function.
\end{theorem}

\begin{proof}
	The number of words in $X^*$ of length $m$ is $n^m$. Given a word $a$ in $X$ of length $m$, a $\cpn$-conjugate word to $a$ will be a cyclic permutation of $a$, that is, it will be some $b\in X^*$ with $a=uv$ and $b=vu$, for some $u, v\in X^*$. Thus we need to know how many distinct cyclic permutations $a$ may have. We know that  $a=uv=vu$ with $u,v\neq 1$ if and only if $a=w^k$ for some $w\neq 1$ and $k>1$ \cite[Corollary~5.3]{La79}.
	
	A word $p$ is called primitive if whenever $p=w^k$ for some $w\in X^*$, it follows that $k=1$. The root of a word $a$, denoted $\sqrt{a}$, is the unique primitive word $p$ such that $a=p^k$. Hence a word $a$ has $|\sqrt{a}|_X$ distinct cyclic permutations.
	
	Denote by $f(d)$ the number of primitive words in $X$ of length $d$. Then the number $a_m$ of $\cpn$-conjugate elements in $X^*$ of length $m$ is
	\[
	a_m=\sum_{d|m}\frac{f(d)}{d}\,.
	\]
	Now the number of words in $X^*$ of length $m$ is given by
	\[
	n^m=\sum_{d|m} f(d)\,.
	\]
	Therefore by the M\"obius inversion formula,
	\[f(m)=\sum_{d|m}\mu\left(\frac{m}{d}\right)n^d,\]
	where $\mu$ is the M\"obius function.
	The result follows.
\end{proof}

\begin{theorem}
	The strict $\con$-conjugacy growth function of $\Poln$ is given by $\sgf_{\Poln/\con,\Sigma_0}(0)=1$, $\sgf_{\Poln/\con,\Sigma_0}(1)=2n+1$, and $\sgf_{\Poln/\con,\Sigma_0}(m)=n^m+(m-1)n^{m-1}(n-1)+\sgf_{A_n^*/\cpns, A_n}(m)$, for $m\geq 2$.
\end{theorem}
\begin{proof}
	We use Corollary~\ref{Cor:cgl_con_Polycyclic} and the previous theorem to deduce the result. The proof follows the same reasoning as the proof of Theorem~\ref{Thm:cgf_cfn_Polycyclic}.
\end{proof}

\begin{theorem}
	The strict $\cpns$-conjugacy growth function of $\Poln$ is given by $\sgf_{\Poln/\cpns,\Sigma_0}(0)=1$, $\sgf_{\Poln/\cpns,\Sigma_0}(1)=2n+1$, and $\sgf_{\Poln/\cpns,\Sigma_0}   (m)=2\sgf_{A_n^*/\cpns, A_n}(m)$, for $m\geq 2$.
\end{theorem}
\begin{proof}
	This follows from Corollary~\ref{Cor:cgl_cpns_Polycyclic} and Theorem~\ref{Thm:cgf_cpn_free_monoid}.
\end{proof}

\subsection{Conjugacy growth series of $\Poln$}

In this subsection we describe the different growth series of the polycyclic monoids. We begin by introducing the concepts.

\begin{defi}
	Let $M$ be a monoid generated by a finite set $X$, and let $\sim$ be an equivalence relation on $M$.
	The \emph{growth series} of $M$ \emph{relative to} $\sim$ is the following power series with indeterminate $z$:
	\[
	\sgs_{M/\sim,X}(z)=\sum_{m\geq 0}\sgf_{M/\sim,X}(m) z^m,
	\]
	where $\sgf_{M/\sim,X}$ is the strict growth function of $M$ relative to $\sim$ with respect to $X$.
\end{defi}

When the relation $\sim$ is one of the conjugacies  $\cfn$, $\con$, or  $\cpns$, we refer to the growth series of $M$ relative to $\sim$ as the \emph{$\sim$-conjugacy growth series} of $M$, or simply as the  \emph{conjugacy} growth series of $M$ relative to $\sim$.

Note that even if one cannot define a growth function for infinitely generated groups, it is possible to make sense of conjugacy growth series for some such groups; see \cite{BdlH18}.

From Theorem~\ref{Thm:cgf_cpn_free_monoid} we deduce the following:

\begin{theorem}
	Let $X$ be an alphabet with $|X|=n$. The $\cpn$-conjugacy growth series of the free monoid on $X$ is
	\[
	\sgs_{X^*/\cpn, X}(z)=\sum_{r,s\geq 1}\frac{n^r}{rs} \varphi\left(s\right)z^{rs},
	\]
	where $\varphi$ is the Euler totient Euler function.
\end{theorem}

We now give an explicit formula for the conjugacy growth series of $\Poln$ for the conjugacies $\cfn$,  $\con$ and $\cpns$.

\begin{theorem}
	The $\cfn$-conjugacy growth series of $\Poln$ is
	\[
	\sgs_{\Poln/\cfn,\Sigma_0}(z) = z + \frac{1-nz^2}{{(1-nz^2)}^2}.
	\]
\end{theorem}
\begin{proof}
	Following Corollary~\ref{Cor:cgl_cfn_Polycyclic}, we count the number of words $sr^{-1}$, where $r$ and $s$ do not have a common prefix other than the empty word. We do this by counting all words $yx^{-1}\in \Poln$ and then removing those for which $x$ and $y$ have at least one common beginning letter from $A_n$. Finally, we count the element $0$, whose conjugacy class contributes $z$.  This gives
	\[
	z+ \frac{1}{{(1-nz)}^2}-nz^2\frac{1}{{(1-nz)}^2},
	\]
	which completes the proof.
\end{proof}

\begin{theorem}
	The $\con$-conjugacy growth series of $\Poln$ is given by
	\[
	\sgs_{\Poln/\con,\Sigma_0}(z)=\frac{1}{1-nz}+z+\frac{(n^2-n)z^2}{{(1-nz)}^2}+\sgs_{A_n^*/\cpn, A_n}(z).
	\]
\end{theorem}
\begin{proof}
	Following Corollary~\ref{Cor:cgl_con_Polycyclic}, we count the number of cyclically reduced words with a prefix in $A_n\cup\{0\}$ and the words in the set $L_p(A_n^{-1})$.
	The conjugacy classes of the elements of $A_n^*$ contribute $\frac{1}{1-nz}$ to the series, and
	the conjugacy class of 0 contributes $z$. Further, the conjugacy classes of the elements $yx^{-1}$ such that both $x$ and $y$ are not empty and have no common prefix other than $1$ contribute $\frac{{(nz)}^2}{{(1-nz)}^2}-\frac{nz^2}{{(1-nz)}^2}$ to the series. Finally, the conjugacy classes of the elements in $(A_n^{-1})^*\setminus\{1\}$ contribute $\sgs_{A_n^*/\cpn, A_n}(z)$.
\end{proof}

For completeness, we present the analogous result for $\cpns$-conjugacy.

\begin{theorem}
	The $\cpns$-conjugacy growth series of $\Poln$ is given by
	\[
	\sgs_{\Poln/\cpns,\Sigma_0}(z)=1+z+2\sgs_{A_n^*/\cpn, A_n}(z).
	\]
\end{theorem}
\begin{proof}
	The conjugacy class of the empty word contributes 1 to the series, and the conjugacy class of 0 contributes $z$. Further,
	the conjugacy classes of the elements of $A_n^*\setminus\{1\}$ and the conjugacy classes of the elements in $(A_n^{-1})^*\setminus\{1\}$ both contribute $\sgs_{A_n^*/\cpn, A_n}(z)$.
\end{proof}

\section{Problems}
\label{Sec:questions}

Let $G$ be a group of permutations on a finite set $X$, and let $t\in T(X)$ be a transformation. Denote by $\langle G,t\rangle$ the semigroup generated by $G$ and $t$. When $G$ has special properties (as defined in the classes of groups and classification theorems in
\cite{AABCS21,ABC21,ABC21b,ABC19,ACS17,AACDHL17,ABCRS16,AAC16,AC16,AC14,ABC13}), the semigroups  $\langle G,t\rangle$, for all non-permutations $t\in T(X)$, have a rich structure and deep interconnections with permutation groups, automata theory, combinatorics, and geometry. Therefore it seems instructive to solve the following problem.

\begin{que}
Let $G$ be a group of permutations on a finite set $X$ and assume $G$ belongs to an interesting class of permutation groups or appears in the classification theorems of the papers mentioned above. Let $t\in T(X)$ and let $\sim$ be any conjugacy considered in this paper. Characterize $\sim$ in $\langle G,t\rangle$.
\end{que}

\begin{que}
Characterize the permutation groups $G$ on a finite set $X$ such that $\sim_p$ is transitive in $\langle G,t\rangle$ for all $t\in T(X)$.
\end{que}

These two problems will certainly require delicate considerations of primitive permutation groups.

We have characterized several conjugacies in the partition monoid and two of its friends.

\begin{que}
Characterize the conjugacy relations for other friends of the partition monoid (such as Planar, Jones, Kauffman, Martin, Temperley and Lieb, etc).
\end{que}

Regarding $G$-sets, we assumed that $G$ is abelian.
\begin{que}
Extend the results in this paper on the endomorphism monoid of a $G$-sets to the case of non-abelian groups $G$.
\end{que}

Let $\sim$ be an equivalence relation. Define ${\sim}^{\greenJ}$ and ${\sim}^{\greenD}$ as follows:
\begin{align*}
a\,\,{\sim}^{\greenJ}\, b\iff (\ a\sim b\quad\text{and}\quad a^k\greenJ b^k\text{ for every integer }k\geq 1\ )\,, \\
a\,\,{\sim}^{\greenD}\, b\iff (\ a\sim b\quad\text{and}\quad a^k\greenD b^k\text{ for every integer }k\geq 1\ )\,.
\end{align*}
Clearly ${\sim}^{\greenD}\,\subseteq\,{\sim}^{\greenJ}$ and ${\sim}^{\greenD}\,=\,{\sim}^{\greenJ}$ in epigroups. By Proposition \ref{Prp:D}, $\cfn\,=\,{\cfn}^{\greenD}\,=\,{\cfn}^{\greenJ}$. 
Also, $\cli\,\,=\,{\ctr}^{\greenD}\,=\,{\ctr}^{\greenJ}$. Both ${\cwn}^{\greenD}$ and ${\cwn}^{\greenJ}$ can be viewed as natural generalizations of $\cli$ to arbitrary semigroups.

\begin{que}
Let $\sim$ be any conjugacy considered in this paper (except for $\cfn$, $\cli$ or $\ctr$). Study ${\sim}^{\greenD}$ and ${\sim}^{\greenJ}$, and describe their conjugacy classes in transformation semigroups and in the partition monoid and its friends.
\end{que}

We know that there exist finitely generated groups for which the word problem is solvable, but the conjugacy problem is not. Hence there exist semigroups for which the word problem is solvable, while (for various notions of conjugacy) the conjugacy problem is not. This leads us to the following problem.
\begin{que}
 Is there a finitely generated semigroup with solvable $\frn$-conjugacy problem and with unsolvable word problem?
\end{que}
Regarding the above problem, we note that, because of Remark \ref{Rem:n_conj_class_01}, given a monoid $M$ with some non-idempotent elements, we cannot embed it into a larger monoid $M_1$ such that all elements of $M$ become $\frn$-conjugate in $M_1$. Hence, the construction in the proof of \cite[Thm.~5.2]{AKKM18} for conjugacy $\con$ will not work for $\frn$-conjugacy.

\begin{que}
 Can we identify the set of $\frn$-normal forms as a species in the sense of \cite{BLL98} in such a way as to count the number of $\frn$-conjugacy classes in the partition monoid by counting the isomorphism type series of this species?
\end{que}


More generally, we have the following problem.

\begin{que}
Let $X$ be a finite set, and let $\sim$ be any of conjugacy considered in this paper. Find a closed formula that gives the number of conjugacy classes of $\sim$ for transformation monoids, and for
the partition monoids and its friends.
\end{que}

The number of conjugacy classes of $\con$ and $\cfn$ in $T(X)$, where $X$ is infinite, were found in \cite[Thm.~6.4]{AKM14} and \cite[Thm.~6.1]{Ko18}, respectively.

\begin{que}
Let $X$ be an infinite set. Find the number of conjugacy classes of other conjugacies in the full transformation monoid $T(X)$.
\end{que}

The conjugacies discussed here have been characterized in various transformation semigroups (see \cite{ArKiKo_Inv,ArKiKoMaTA,AKM14,GaMa10,Ko18,KuMa07,Steinberg15,St19} and the results in \S\ref{Sec:tra}).
\begin{que}
Extend these characterizations to other transformation semigroups, for example, those appearing in the problem list of \cite[\S6]{AK13} or
those appearing in the large list of transformation semigroups included in \cite{F02}.
\end{que}

The theorems and problems in this paper have natural linear counterparts.
\begin{que} Let $\sim$ be any conjugacy considered in this paper.
Characterize $\sim$ in the endomorphism monoid of a finite dimensional vector space over a field.
This problem was solved for $\cpns$ in \cite{KuMa08}.
\end{que}

When results hold for both finite sets and finite dimensional vector spaces, it is natural to try to extend the results to independence algebras \cite{G95,G07,CS00,ABCKK22}.

\begin{que} Let $\sim$ be any conjugacy considered in this paper.
Characterize $\sim$ in the endomorphism monoid of a finite dimensional independence algebra (over a finite or infinite field). The classification of independence algebras \cite{ABCKK22} may be useful for this. \cite{ABCKK22}.
\end{que}

\begin{que} Let $\sim$ be any conjugacy considered in this paper.
Characterize $\sim$ in the endomorphism monoids of free objects \cite{AMS03} or in the endomorphism monoid of algebras admitting some general notion of independence \cite{AW09}. Regarding the latter, we propose the problem of finding the conjugacy classes in the endomorphisms of $\mathit{MC}$-algebras, $M S$-algebras, $S C$-algebras, and $S C$-ranked algebras \cite[Ch.~8]{AW09}. The first step would be to solve the conjugacy problem for the endomorphism monoid of an $S C$-ranked free $M$-act \cite[Ch.~9]{AW09}, and for an $S C$-ranked free module over an $\aleph_1$-Noetherian ring \cite[Ch.~10]{AW09}.
\end{que}

Related to the previous problem, we have the following.
\begin{que}
Since all varieties of bands are known, describe the conjugacy classes of the endomorphism monoid of the free objects of each variety of bands (for details and references, see \cite{AK07}).
\end{que}

\begin{que} Let $\sim$ be any conjugacy considered in this paper.
Characterize the semigroups in which $\sim$ is a congruence [the identity relation, the universal relation].
\end{que}

Trevor Jack \cite{TJack} recently found an infinite chain, with respect to inclusion, of first-order definable notions of conjugacy, answering a question from \cite{ArKiKoMaTA} and an early draft of this work.
\begin{que}
Is it possible to find an infinite set of first-order definable notions of conjugacy for semigroups such that these notions form an infinite anti-chain?
\end{que}

Many results provide partial characterizations, but it would be good to have complete necessary and sufficient conditions.

\begin{que}
Complete the following results: Corollary \ref{Cor:band}, Proposition \ref{Prp:bisimple}, Proposition \ref{Prp:trivial} and Proposition \ref{Prp:w-univ}.
\end{que}

There exist unstable semigroups which do not contain isomorphic copies of the bicyclic monoid, but we do not know the answer to the following.

\begin{que}
Regarding Theorem~\ref{thm:stable}, do either of the converse implications (2)$\implies$(1) or (3)$\implies$(2) hold? (They cannot both hold.)
\end{que}

Group conjugacy and normal subgroups are linked by the fact that a subgroup of a group is normal precisely when it is closed under group conjugacy. For each semigroup conjugacy $\sim$, we can define the \emph{normal subsemigroups} of a semigroup $S$ with respect to $\sim$
as those subsemigroups of $S$ that are closed under $\sim$ \cite{KoTASF}.
The normal subsemigroups of $P_n$, $T_n$, and $\mi_n$, with respect to
conjugacies $\cfn$, $\cpns$, $\coon$, and $\con$, were described in \cite{KoTASF}.

\begin{que}
Let $\sim$ be any conjugacy considered in this paper. Describe the normal subsemigroups of various semigroups, with respect to $\sim$.
\end{que}

Finally, we should mention that some of the problems stated in the earlier paper \cite{ArKiKoMaTA} have been solved recently by T. Jack \cite{TJack2}.

\section*{Acknowledgements}
We thank Laura Ciobanu and Susan Hermiller for the idea of studying conjugacy growth in finitely generated groups, which extends naturally to finitely generated monoids. We also thank Trevor Jack for pointing out an error in a previous version and thank the referee for their thorough
and insightful report.

MK thanks Nic Ormes for discussions regarding strong shift equivalence and shift equivalence in dynamical systems.
\smallskip

This work is funded by national funds through the FCT - Fundaç\~{a}o para a Ci\^{e}ncia e a Tecnologia, I.P., under the scope of the projects UIDB/00297/2020 and UIDP/00297/2020 (Center for Mathematics and Applications).


\end{document}